\RequirePackage{fix-cm}
\documentclass[smallextended]{svjour3}       
\smartqed  
\usepackage{graphicx}
\usepackage{algpseudocode,algorithm,algorithmicx}

\usepackage{amsmath}
\usepackage{subfloat}
\usepackage{caption}
\usepackage{subcaption}
\usepackage{amssymb}
\usepackage{appendix}
\usepackage{hyperref}
\usepackage{enumerate}
\usepackage{enumitem}
\usepackage[colorinlistoftodos,bordercolor=orange,backgroundcolor=orange!20,linecolor=orange,textsize=scriptsize]{todonotes}

\newtheorem{assume}{Assumption}

\newcommand{\st}{\operatorname{s.t.}}
\newcommand{\zero}{\mathbf{0}}

\newcommand{\dom}{\operatorname{dom}}

\newcommand{\dist}{\operatorname{dist}}

\newcommand{\GPPAstep}{\operatorname{IGPPAstep}}
\newcommand{\AGPPA}{\operatorname{AGPPA}}

\newcommand{\nnz}{\operatorname{nnz}}
\newcommand\norm[1]{\left\lVert#1\right\rVert}
\newcommand\bracket[1]{\left (#1\right)}

\newcommand{\bR}{\mathbb{R}}

\newcommand{\cX}{\mathcal{X}}

\newcommand{\bE}{\mathbb{E}}

\newcommand{\bB}{\mathbb{B}}

\newcommand{\cS}{\mathcal{S}}
\newcommand{\cA}{\mathcal{A}}
\newcommand{\cD}{\mathcal{D}}
\newcommand{\cB}{\mathcal{B}}

\newcommand{\cQ}{\mathcal{Q}}

\newcommand{\cE}{\mathcal{E}}
\newcommand{\cM}{\mathcal{M}}
\newcommand{\cO}{O}

\newcommand{\cG}{\mathcal{G}}

\newcommand{\cI}{I}
\newcommand{\cJ}{\mathcal{J}}
\newcommand{\cT}{\mathcal{T}}

\algnewcommand\INPUT{\item[\algorithmicinput]}
\algrenewcommand\algorithmicensure{\textbf{Initialize:}}
\algnewcommand\algorithmicinput{\textbf{Input:}}
\algrenewcommand\algorithmicrequire{\textbf{Output:}}

\makeatletter
\newenvironment{megaalgorithm}[1][htb]{%
    \renewcommand{\ALG@name}{}
   \begin{algorithm}[]%
  }{\end{algorithm}}
\makeatother

\begin{document}

\title{An adaptive proximal point algorithm framework and application to large-scale optimization\thanks{The first author was supported by Postgraduate Scholarship from Hong Kong University No. 0100014984. The second author was supported by Early Career Scheme from Hong Kong Research Grants Council No. 27302016. }}


\author{Meng Lu         \and
        Zheng Qu
}


\institute{M.Lu \at
               Department of Mathematics\\
               The University of Hong Kong\\
                Pokfulam Road, Hong Kong\\
              \email{menglu\_16@connect.hku.hk}  
           \and
           Z.Qu \at
           Department of Mathematics\\
               The University of Hong Kong\\
                Pokfulam Road, Hong Kong\\
             \email{zhengqu@maths.hku.hk}\\
}

\date{the date of receipt and acceptance should be inserted later}

\maketitle

\begin{abstract}
We investigate the inexact proximal point algorithm (PPA)  under  the bounded metric subregularity condition.  Global linear convergence of inexact PPA requires the knowledge of the bounded metric subregularity parameter, which is in general difficult to estimate. We propose an adaptive generalized proximal point algorithm (AGPPA), which adaptively updates the proximal regularization parameters based on a sequence of implementable criteria. We show that AGPPA achieves  linear convergence without any knowledge of the bounded metric subregularity parameter, and that the rate only differs from the optimal one by a logarithmic term. 
 We apply AGPPA to  convex minimization problem and analyze the iteration complexity bound. Our framework and the complexity results apply for arbitrary linearly convergent inner solvers, and are flexible enough to allow for a hybrid with any method. We illustrate the performance of AGPPA by applying it to  large-scale linear programming (LP) problem. The resulting complexity bound has weaker dependence on the Hoffman constant and scales with the dimension better than linearized ADMM. In numerical experiments, our algorithm demonstrates improved performance in obtaining solution of medium accuracy on large-scale LP problems.
\keywords{proximal point algorithm \and error bound condition \and  adaptive proximal regularization parameter \and  iteration complexity \and large-scale optimization \and linear programming}
\end{abstract}

\section{Introduction}
\label{intro}
\subsection{Problem and motivation}\label{sec:pam}
Let $\cX$ be a finite dimensional Hilbert space 
endowed with inner product  $\langle \cdot , \cdot \rangle$ and induced norm $ \norm{\cdot}$. 
Let $T: \cX \rightrightarrows \cX$ be a maximal monotone operator.  We aim to find a solution $z\in \cX$ such that
\begin{equation}\label{eq:inclusion problem0}
\zero \in T (z),
\end{equation}
where $\zero$ denotes the origin of the space $\cX$.
We shall assume  throughout the paper that the solution set $\Omega : = T^{-1}(\zero)$ is nonempty. 

The proximal point algorithm  (PPA) framework,  which traces back to~\cite{Martinet1970Re,1965Proximite}, is devoted to solve the inclusion problem~\eqref{eq:inclusion problem0} and plays a highly influential role in the optimization history.
It has widespread applications in  various fields and has inspired tremendous creative work in the design and analysis of optimization methods. This framework chooses a sequence of \textit{proximal regularization parameters} $\{ \sigma_k \}_{k \ge 0}$ and generates  $\{z^k\}_{k \ge 0}$ from an arbitrary initial point $z^0$  by the following  rule:
\[
  z^{k+1} = \cJ_{\sigma_k T}  (z^k):=  (I+\sigma_k T)^{-1}(z^k),\enspace \forall k\geq 0.
\]
Here, $I:\cX\rightarrow \cX$ denotes the identity operator.
One inexact version of PPA approximates  $\cJ_{\sigma_k T}  (z^k)$:
\[
  z^{k+1} \approx \cJ_{\sigma_k T}  (z^k),\enspace \forall k\geq 0,
\]
based on  the following conceptual inexactness criteria:
\[
 \| z^{k+1}-\cJ_{\sigma_k T}  (z^k) \| \leq \min \left\{ \eta_k, \delta_k \|z^{k+1} -z^k \| \right\},\enspace \forall k\geq 0.
\]
Here, $\{ \eta_k\}_{k \ge 0}$ and $\{\delta_k \}_{k \ge 0}$ are \textit{error parameters} that control the precision of the approximation. 

In this paper, we consider  the following generalization of inexact PPA:
\begin{align}\label{a:igppa}
 w^{k} \approx \cJ_{\sigma_k \cM^{-1} T}  (z^k), ~z^{k+1} = \gamma w^k +  (1-\gamma) z^k,\enspace \forall k\geq 0,
\end{align}
along with the following inexactness  criteria:
\begin{align}\label{a:cs}
 \| w^{k}-\cJ_{\sigma_k \cM^{-1}T}  (z^k) \|_{\cM} \leq \min \left\{ \eta_k, \delta_k \|w^k-z^k \|_{\cM} \right\},\enspace \forall k\geq 0.
\end{align}
Here, $\gamma \in  (0,2)$ is a \textit{relaxation factor}, $\cM$ is some self-adjoint positive definite linear operator over $\cX$, known as a \textit{preconditioner}, and $\norm{z}_{\cM}=  \sqrt{ \langle z, \cM z \rangle}$. We call~\eqref{a:igppa} the inexact generalized PPA (IGPPA).  In the literature, such generalization has been studied in various contexts and is known by different names: \textit{relaxed PPA} ($\cM=I$)~\cite{Eckstein1992On,Gabay1983Chapter},~\textit{preconditionned PPA} ($\gamma=1$)~\cite{Li2019An}, and \textit{relaxed customized PPA}~\cite{RelaxedCustomizedPPA,GuCustomized}. It is  observed  that using a relaxation factor  $\gamma>1$  can speed up convergence in practice~\cite{fang2015generalized,RelaxedCustomizedPPA},  while a suitable preconditioner  allows us to exploit the specific structures and alleviate the difficulty of solving the inner problems~\cite{RelaxedCustomizedPPA,GuCustomized}.  IGPPA covers an even  wider range of applications compared with the classical PPA, and prompts numerous works~\cite{marino2004convergence,he2012convergence,ma2018class,RelaxedCustomizedPPA,Yuan2014A} to study algorithm designs and  convergence analysis. 
Let us consider the following assumption on the operator $T$.
\begin{assume}\label{assume:growthcondition}
For any $r > 0$, there exists $\kappa_r >0$ such that 
\begin{align}\label{growthcondition}
\dist \bracket{z,\Omega} \le \kappa_r \dist \bracket{\zero, T (z)}, ~\forall z\in \{x\in\cX: \norm{x} \le r\}. 
\end{align}
\end{assume}
Following~\cite{zheng2014metric}, we shall refer to Assumption~\ref{assume:growthcondition} as the bounded metric subregularity condition.  Relations with other error bound conditions will be discussed later in Section~\ref{subsection:EB}. 

As a revisit of the existing (local) linear convergence results on PPA,  we establish in Theorem~\ref{thm:linearrate}, an upper bound $\rho_k$ on the ratio 
 \[
 \dist_{\cM}(z^{k+1},\Omega)/\dist_{\cM}(z^k,\Omega)  \leq \rho_k,
 \]
 for IGPPA~\eqref{a:igppa} under  Assumption~\ref{assume:growthcondition}.  The result recovers the linear convergence rate obtained in~\cite[Theorem 3.5]{Min2016On} and~\cite[Theorem 2]{Li2019An} as special cases.
The bound $\rho_k$ can be  made arbitrarily close to  $\sqrt{1 - \min \{ \gamma, 2\gamma -\gamma^2 \}}$,  if $\sigma_k$ is sufficiently large and $\delta_k$ is sufficiently small. When $\gamma = 1$, this corresponds to the  well-known  superlinear convergence property of PPA.  For exact PPA, i.e., when the error parameters $\delta_k=0$ for all $k\geq 0$ in~\eqref{a:cs}, the value of $\sup_k \rho_k$ is 
strictly smaller than 1 and linear convergence is guaranteed. 
 However,    to ensure the linear convergence of  inexact PPA (i.e., to make $\sup_k \rho_k<1$), it is necessary to know  the value of the  bounded metric subregularity parameter $\kappa_r$  in order to  choose appropriate proximal regularization parameters $\{\sigma_k\}_{k \ge 0}$,  which is an unrealistic assumption for most of the problems.  The main objective of this paper is to study the convergence  and applications of IGPPA under Assumption~\ref{assume:growthcondition} without assuming the knowledge of $\kappa_r$.

\subsection{Contribution}

To deal with the unknown 
$\kappa_r$, we propose to 
 adaptively choose the proximal regularization parameters $\{\sigma_k\}_{k\geq 0}$ by verifying if $\left\{\|z^{k+1}-z^k\|_{\cM}\right\}_{k \ge 0}$ decreases sufficiently quickly. The resulting Algorithm~\ref{alg:AGPPA}, called \textit{adaptive generalized PPA} (AGPPA), is able to find an $\epsilon$-solution within 
\begin{align}\label{a:Oka}
O\left(\ln \kappa_r \ln\frac{r\kappa_r}{\epsilon}\right)
\end{align}
number of IGPPA steps~\eqref{a:igppa}, without requiring any a priori knowledge of $\kappa_r$ (see Theorem~\ref{thm:AGPPA} and Corollary~\ref{cor:wrd} for more details). Here, $r$ is an upper bound on $\sup_{k\geq 0} \norm{z^k}$, which is finite as long as $\{\eta_k\}_{k\geq 0}$ is summable (see~\eqref{a:zkzo} and~\eqref{r:def}).  

We apply AGPPA to solve the convex optimization model:
\begin{equation}\label{prob:constrained0}
\min_{x\in \bR^n}f_0 (x)+g (x)+h (Ax),
\end{equation}
where $f_0$ is a smooth convex function, $A\in \bR^{m\times n}$, and  $g$ and $h$ are proper, closed and convex functions. 
 We take the standard maximal monotone operator $T_{\ell}$ associated with the Lagrangian function of~\eqref{prob:constrained0} (see~\eqref{def:Tell}). Assuming that $T_{\ell}$ satisfies Assumption~\ref{assume:growthcondition}, 
we analyze the complexity of the resulting  proximal method of multipliers with adaptive proximal regularization parameters. We concretize the conceptual inexactness criteria~\eqref{a:cs} with some frequently used implementable stopping criteria for the inner problems (see~\eqref{stop:computable}).  Then, by merely requiring the inner solver to satisfy the so-called homogeneous objective decrease (HOOD) property (see Assumption~\ref{assume:Hood}),  we deduce in Theorem~\ref{thm:subcomp} an upper bound on the number of inner iterations to reach~\eqref{stop:computable}, and hence~\eqref{a:cs}. 
This upper bound directly yields a complexity bound for AGPPA with any  inner solver satisfying the HOOD property, including randomized methods (see Theorem~\ref{thm:6}). Specifically, with probability at least $1-p$, the batch complexity (i.e., the number of passes over the data matrix $A$) of the resulting proximal method of multipliers to reach an $\epsilon$-solution is (see~\eqref{a:floglo}):
\begin{equation}\label{intro:overallcomplexity}
O\left ( \vartheta_1  \kappa_r   \ln \kappa_r\ln\frac{ r\kappa_r }{\epsilon} \ln \left( \frac{r\kappa_r}{p}  \ln \frac{r \kappa_r}{\epsilon}\right) \right),
\end{equation}
where $ \vartheta_1$ is an inner solver related constant  as defined in~\eqref{a:cde}, and $\kappa_r$ is the bounded metric subregularity parameter of $T_{\ell}$ with $r$ being an upper bound on the norm of all iteration points of AGPPA. 
Our theoretical complexity bound continues to apply if  an arbitrary inner solver is used, provided that it is carefully combined with a qualified first-order solver (see  Section~\ref{subsection:hybrid}).  Note that 
\[
 \vartheta_1 \le \norm{A},
\] 
if we choose some appropriate inner solver (see Section~\ref{subsection:overall}).

Examples of~\eqref{prob:constrained0} with the associated maximal monotone operator $T_{\ell}$ satisfying Assumption~\ref{assume:growthcondition} features linear-quadratic programming problems.
Conditions on $f_0$, $g$ and $h$ so that $T_{\ell}$ satisfies Assumption~\ref{assume:growthcondition}  require future study, which is out of the scope of this paper. We point to~\cite{YuanADMM,necoara2019linear} for possible other relevant models. 
In this paper, we  illustrate the application of AGPPA and its complexity results to large-scale LP problems. We show in Theorem~\ref{LP:bacth complexity} that with probability at least $1-p$, the batch complexity of AGPPA to obtain an $\epsilon$-KKT solution of the LP problem~\eqref{P} is 
\begin{align}\label{a:introA}
O\bracket{\min \left (\max_{i\in [n]} \|a_i\|,  \frac{\norm{A}_F }{\sqrt{m}} \right)  \theta r \ln(\theta r) \ln \frac{\theta r}{\epsilon} \ln \left(\frac{\theta r }{p}\ln \frac{\theta r}{\epsilon}\right) },
\end{align}
where  $a_i$ is the $i$th  column vector of $A$, $r$  is an upper bound on the norm of all iteration points of AGPPA,  and $\theta$ is the constant satisfying~\eqref{a:thetaSdef}, upper bounded by the Hoffman constant  associated with the KKT system. 

We test the practical performance of our method on different LP problems, using real and synthetic data sets. 
 We compare our algorithm with an ALM based method AL\_CD~\cite{Yen2015Sparse}, an ADMM based solver SCS~\cite{SCS}, and the Gurobi software~\cite{gurobi}, up to accuracy $10^{-3}$ and $10^{-5}$ for the normalized KKT residual as defined in~\eqref{eq:normalized KKT residual}. 
 Based on the experimental results, we observe superior performance of AGPPA in both memory usage  and time efficiency for large-scale data sets (see Section~\ref{section:numericalresults} for more details). Moreover, we also demonstrate that transforming first into the standard form (see~\eqref{Sun:P}) or its dual form (see~\eqref{eq:D})  before applying AGPPA will lead to a worse complexity bound and may significantly slow down convergence in experiments. This highlights the advantage of AGPPA  compared with  other  closely related algorithms~\cite{Li2019An,NIPS2017_ADMM,SCS} limited to the standard form or its dual form.

\subsection{Related work}\label{sec:rel}
The proximal point method and in particular its applications to convex programming have been widely studied in the literature. We have mentioned some related work above. We give a more detailed comparison of our contribution with existing work in this section. 

\begin{enumerate}
\item  While there are many works proving the (local) linear convergence of PPA,
 very few of them discuss  how the a priori knowledge of the parameter $\kappa_r$ influences the
 overall complexity of PPA. As mentioned earlier, exact PPA and  its applications such as the
  alternating direction method of multipliers (ADMM)~\cite{Glowinski1978Finite,GlowinskiSur} converge linearly without the knowlege of $\kappa_r$. 
  However, this is not the case for inexact PPA and in particular the augmented Lagrangian method (ALM)~\cite{PPA}. 
  It is widely known that when $\delta_k$ is sufficiently close to 0 and $\sigma_k$ is sufficiently large, ALM converges
  linearly under certain error bound condition such as Assumption~\ref{assume:growthcondition}. 
  However, deliberately decreasing $\delta_k$ and increasing $\sigma_k$ will add to the difficulty of finding $w^k$ satisfying~\eqref{a:cs}.
A theoretical study on the complexity of inexact PPA without assuming the knowledge of the parameter $\kappa_r$ is necessary
to guide the choice of $\{\delta_k\}_{k\geq 0}$ and $\{\sigma_k\}_{k\geq 0}$.

 \item 
  There exist works on adaptive update of proximal regularization parameters $\{\sigma_k\}_{k\geq 0}$ for some specific applications of PPA.  For example, He et al.~\cite{He2000Alternating} introduced  an ADMM with self-adaptive proximal regularization parameters, based on the value of primal and dual error residuals. 
 Similar adaptive techniques are  also frequently used in ALM based solvers~\cite{birgin2008structured,andreani2008augmented,bueno2019towards,birgin2008improving,birgin2012augmented,birgin2020complexity} for box constrained optimization to increase the proximal regularization parameter when the progress, in terms of feasibility and complementarity, is not sufficient, and to decrease it when sufficient progress  is detected. 
 However,  all the mentioned papers do not carry out rigorous complexity analysis of the proposed adaptive methods.
 In fact, to the best of our knowledge,~\eqref{intro:overallcomplexity} is the first iteration complexity result of ALM of order $O(\ln (1/\epsilon)\ln(\ln(1/\epsilon)))$ under an error bound type assumption, without
 assuming the knowledge of the error bound parameter $\kappa_r$.  Note that without error bound type assumption, the complexity of ALM is known to be $O(1/\epsilon)$~\cite{xu2019iteration}.  There is a recent work by Necoara and Fercoq~\cite{necoara2019linear}  which also achieves  a comparable complexity result as~\eqref{intro:overallcomplexity} for
 the coordinate descent algorithm applied to the dual of a projection-like problem, 
without assuming the knowledge of the regularity  parameter of the constraint set.

\item 
 The (local) linear convergence of ADMM and its variants  under error bound type conditions have been well studied, see e.g.~\cite{davis2017faster,deng2016global,nishihara2015general}.  One of its variants, the linearized ADMM (LADMM), appears to be particularly interesting for large-scale computing since its  inner step  can be  solved exactly without resorting to matrix factorization. The best known batch complexity of LADMM  to obtain an $\epsilon$-KKT solution for solving problem~\eqref{prob:constrained0} is~\cite{YuanADMM}:
 \begin{align}\label{a:LADMMc1}
 \cO \bracket{\norm{A}^2\kappa_{r'}^2 \ln \frac{1}{\epsilon}},
 \end{align}
where  $r'$ is instead an upper bound on the norm of all iteration points of LADMM. Comparing~\eqref{a:LADMMc1} with the complexity bound~\eqref{intro:overallcomplexity} of AGPPA, we see that AGPPA  has much weaker dependence on the  bounded metric subregularity parameter.

When specifying~\eqref{a:LADMMc1} to an LP problem, we obtain the following complexity bound:
\begin{equation}\label{a:LADMMc}
 \cO \bracket{\norm{A}^2\theta^2 (r')^2 \ln \frac{1}{\epsilon}},
 \end{equation}
 Comparing~\eqref{a:LADMMc} with the complexity
    bound~\eqref{a:introA} of AGPPA, we see that~\eqref{a:introA} can be understood as an \textit{acceleration} due to the weaker dependence on the constant $\theta$.
It is also easy to see that
\[
\min \left (\max_{i\in [n]} \|a_i\|,  \frac{\norm{A}_F }{\sqrt{m}} \right) \leq \|A\|,
\]
and can be much smaller if either $n$ or $m$ are large. This is why we expect AGPPA to perform better than LADMM on large-scale LP problems.

  \item  The application of PPA to solve LP problems has been studied for more than fifty years and is known under different names: Tikhonov regularization~\cite{doi:10.1137/1021044}, nonlinear perturbation~\cite{mangasarian1979nonlinear,mangasarian1981iterative}, penalty method~\cite{Bertsekas75}, method of multipliers~\cite{PolyakT72}, etc.  The connection of these methods with PPA was clarified by Rockafellar in~\cite{Rockafellar1976Augmented}. There also has been a lot of effort in the design of algorithms for solving the inner problems, which includes the  Newton-type methods proposed by Mangasarian  in~\cite{Mangasarian2004A} and Kanzow et al. in~\cite{kanzow2003minimum}  and the active set method proposed by Hager et al. in~\cite{hager1992dual,davis2008sparse,davis2008dual}. { Recent works} focus on large-scale LP 
  and promote the
  (proximal) ALM~\cite{Li2019An,Yen2015Sparse}  and ADMM~\cite{YuanADMM,NIPS2017_ADMM} type methods for large-scale LP problem. 
  The asymptotic behaviour of ALM for LP problem is known to be superlinearly convergent~\cite{Li2019An}
   if a Newton type inner solver is employed. However, in the existing literature there is no complexity result of PPA for solving  LP problem comparable with~\eqref{a:introA}. 
   A detailed comparison with recent (proximal) ALM~\cite{Li2019An,Yen2015Sparse}  and ADMM~\cite{YuanADMM,NIPS2017_ADMM}  based large-scale LP solvers can be found  in Section~\ref{section:LPcomplexity}.

\end{enumerate}

\subsection{Contents and notations}
The paper is organized as follows. In Section~\ref{section:GPPA}, we revisit IGPPA and present some convergence results for preparation. In Section~\ref{section:AGPPA}, we introduce AGPPA and give an upper bound of the number of IGPPA steps. In Section~\ref{sec:app}, we apply AGPPA to the convex optimization problem and show the overall iteration complexity bound. In Section~\ref{section:LP}, we apply our method and complexity results to LP problem. In Section~\ref{section:numericalresults}, we present numerical results. In Section~\ref{section:conclusion}, we make some conclusions. Missing proofs can be found in the Appendix.

\textbf{Notations.}
The set of self-adjoint positive definite linear operators  over $\cX$ is denoted by $\cS^{++}$.  For  $\cM\in\cS^{++}$, $\cM^{-1}:\cX\rightarrow \cX$ denotes the inverse operator of $\cM$. For any   $z, z' \in \cX$ and $\cM\in\cS^{++}$, denote $
\langle z, z' \rangle_{\cM} = \langle z,\cM z' \rangle$ and $\norm{z}_{\cM}=  \sqrt{ \langle z, \cM z \rangle}.
$  For a closed set $D\subset \cX$, denote  the weighted distance from $z$ to $D$ by $\dist_{\cM} (z, D) = \min_{d \in D} \norm{d-z}_{\cM}$. If $\cM$ is the identity operator $\cI$, we omit it from the subscript.  The origin of the space $\cX$ is denoted by $\zero$. For any $z\in \cX$  and $r>0$, $\bB(z;r)=\{x\in \cX:\|x-z\|\leq r\}$, $\bB:=\bB(\zero;1)$ and $\bB(\Omega;r)=\{x\in \cX: \dist(x, \Omega)\leq r\}$.

We use $\norm{\cdot}$ to denote the standard Euclidean norm for vector and spectral norm for matrix.  The set of $n$-by-$n$ positive definite matrices  is denoted by $\cS_n^{++}$.
For any  $k > 0$, define $[k]: = \{1, \dots, k\}.$ 
For any $x\in \bR^n$ and $k\in[n]$, denote by $[x]_+^k$ the projection of $x$ into $\bR_+^k\times \bR^{n-k}$. 
The same, $[x]_-^k$ means the projection of $x$ into $\bR_{-}^k\times \bR^{n-k}$.
For any $k \ge 1$, $x_1\in\bR^{n_1}, \dots, x_k \in \bR^{n_k}$,  we write $[x_1;\ldots;x_k]$ the vector in $\bR^{n_1+\dots+n_k}$ obtained by concatenating $x_1,\ldots, x_k$.
Similarly, for any two matrices $A \in \bR^{m_1\times n}$ and $B \in \bR^{m_2 \times n}$, $[A;B]$ is the matrix in $\bR^{ (m_1+m_2)\times n}$ obtained by concatenating $A$ and $B$.

\section{Inexact Generalized PPA (IGPPA)}\label{section:GPPA}

In this section, we revisit IGPPA. 
 Let $\cM \in \cS^{++}$, then  the  operator $\cM^{-1}T: \cX\rightrightarrows \cX $ is a maximal monotone operator in the Hilbert space $\cX$ endowed with inner product $\langle \cdot,\cdot \rangle_\cM$.  Consider the  \textit{resolvent operator} of $\cM^{-1}T$:
\[
\cJ_{\sigma \cM^{-1} T}:= (I+\sigma \cM^{-1} T)^{-1},
\]
with {parameter} $\sigma>0$. 
 Without loss of generality, we assume\footnote{There is a slight redundancy in using both the parameter $\sigma$ and the preconditioner $\cM$. We could set the  preconditioner as $\cM/\lambda_{\max} (\cM)$  and the proximal regularization parameter as $\sigma/\lambda_{\max} (\cM)$ to yield the same resolvent operator. }
\begin{align}
\label{a:lambdamaxM1}\lambda_{\max} (\cM)=1.
\end{align}
An IGPPA step first approximately applies the resolvent operator $\cJ_{\sigma \cM^{-1} T}$, and then makes an affine combination with the current iteration point for some relaxation factor $\gamma \in  (0,2)$. A more specific inexactness condition is described in the following procedures.

  \renewcommand{\thealgorithm}{}
\begin{megaalgorithm}[H]
 \caption{$z^+=\GPPAstep (z,\sigma, \eta, \delta,\gamma, \cM)$ }
	\begin{algorithmic}
	\State 1. Compute an approximate solution $
	w \approx \cJ_{\sigma \cM^{-1} T} (z)
	$ such that 
	\begin{equation}\label{GPPA:stop}
	\norm{w- \cJ_{\sigma \cM^{-1} T} (z)}_{\cM} \le \min \left\{\eta,\delta \norm{w- z}_{\cM} \right\}.
	\end{equation}
	\State 2. Compute
	\begin{equation}\label{GPPA:next iterate}
	z^{+} = \gamma w +  (1-\gamma) z.
	\end{equation}
	\State 3. Output $z^{+}$.
	\end{algorithmic}
      \label{alg:GPPAstep}
\end{megaalgorithm}
The inexactness is controlled by~\eqref{GPPA:stop} along with  two error parameters $\eta$ and $\delta$. 
When relaxation factor $\gamma=1$ and preconditioner $\cM=\cI$, the above procedure reduces to the classical inexact PPA~\cite{PPA}.  
Eckstein and Bertsekas~\cite[Theorem 3]{Eckstein1992On} established the convergence result for the relaxed PPA ($\cM = \cI$), and Li et al~\cite[Theorem 1]{Li2019An} established the convergence result for the preconditioned PPA ($\gamma = 1$). In the following theorem, we present the general convergence result for IGPPA based on the two theorems above.
\begin{theorem}[\cite{Eckstein1992On,Li2019An}]\label{GPPA:convergence}Let $\{z^k\}_{k \ge 0}$ be a sequence in $\cX$ such that
\begin{equation}\label{eq:GPPAframework}
z^{k+1} = \GPPAstep (z^k, \sigma_k,  \eta_k,\delta_k, \gamma, \cM),\enspace \forall k\geq 0,
\end{equation}
where  $\{\sigma_k\}_{k=0}^{\infty}$, $\{\eta_k\}_{k=0}^{\infty}$, $\{\delta_k\}_{k=0}^{\infty}$ are nonnegative sequences such that 
\[\begin{array}{l}
   \sum_{k= 0}^{\infty} \eta_k < + \infty,\enspace
   \inf_k \sigma_k>0,\enspace \sup_k\delta_k <1.
 \end{array}\]
 Then,  for any $z^* \in \Omega$ , we have 
\begin{equation}\label{eq:zbound1}
\norm{z^{k+1} - z^*}_{\cM} \le \norm{z^k -z^*}_{\cM}    + \gamma\eta_k,\enspace \forall k\geq 0.
\end{equation}
In addition, $\{z^k\}_{k \ge 0} $ converges to a point $z^{\infty} \in \Omega$. 
\end{theorem}

The fact that~\eqref{eq:zbound1} holds for any $z^*\in \Omega$  implies
\[
\norm{z^{k+1} -  \bar z^0}_\cM \le  \norm{ z^{k} - \bar z^0}_\cM+\gamma \eta_k,\enspace \forall k\geq 0,
\]
where $\bar z^0$ is the projection of the initial point $z^0$ into the solution set $\Omega$.
It follows that the sequence $\{z^k\}_{k \ge 0}$ generated by~\eqref{eq:GPPAframework} satisfies~
\begin{align}\label{a:zkzo}
\sup_k \norm{z^{k} -  \bar z^0} \le \frac{1}{ \lambda_{\min}(\cM) } \left (\dist_{\cM} (z^0, \Omega) + \gamma \sum_{k = 0}^{\infty}  \eta_k \right),
\end{align}
and hence
\begin{equation}\label{r:def}
 \sup_k \norm{z^k} \le  \norm{\bar z^0} + \frac{1}{ \lambda_{\min}(\cM) } \left (\dist_{\cM} (z^0, \Omega) + \gamma \sum_{k = 0}^{\infty}  \eta_k \right).
\end{equation}

\begin{remark}
Eckstein and Bertsekas allowed the relaxation factor $\gamma$ to vary with $k$ in~\cite{Eckstein1992On}. They proved the same convergence results under the condition $0<\inf_k \gamma_k\leq \sup_k \gamma_k <2$.  For the sake of simplicity, we restrict our discussion to  constant relaxation factor  $\gamma$.  Similarly, referring to~\cite{Li2019An}, we could also allow $\cM$ to vary with $k$ and all the results can be extended immediately if there exist $\lambda_u \geq \lambda_l>0$ such that $\lambda_l \cI \preceq \cM_{k+1} \preceq \cM_{k} \preceq  \lambda_u \cI$ holds for any $k\geq 0$.
\end{remark}

\subsection{Error bound conditions}\label{subsection:EB}

The linear convergence of PPA has been extensively studied in the literature under various error bound conditions for both exact and inexact versions. For  comparison purpose, we recall some error bound conditions and the related convergence results.

Rockafellar~\cite{PPA} established the (local) linear convergence  property of the classical PPA under the assumption that  $T^{-1}(\zero)=\{\bar z\}$ is a singleton and
\begin{equation}\label{rockupperLip0}\exists r>0,\exists \kappa_r>0 \mathrm{~s.t.~}
\|z-\bar z\| \le  \kappa_{r} \|w\|,~\forall  z\in T^{-1}(w) \mathrm{~and~} w\in \bB(\zero;r).
\end{equation}
Luque~\cite{PPAcon} showed that the  uniqueness assumption on $T^{-1}(\zero)$ can be relaxed so that the (local) linear convergence 
holds if 
\begin{equation}\label{upperLip0}
\exists r>0,\exists \kappa_r>0 \mathrm{~s.t.~}
\dist(z, \Omega) \le  \kappa_{r} \|w\|, ~\forall  z\in T^{-1}(w) \mathrm{~and~} w\in \bB(\zero;r),
\end{equation}
which is same as requiring the locally upper Lipschitzian property of $T^{-1}$ at point $\zero$ as defined in Robinson's paper~\cite{Robinson1979PM}:
\begin{align}\label{a:Robinson}\exists r>0,\exists \kappa_r>0 \mathrm{~s.t.~}
T^{-1}(w) \subset T^{-1}(\zero)+ \kappa_r \| w\| \bB,\enspace \forall w\in \bB(\zero;r).
\end{align}
Leventhal~\cite{LEVENTHAL2009681} proved the (local) linear convergence of the classical PPA   under the following metric subregularity condition of $T$:
\begin{align}\label{a:Leventhal}
\exists r>0,\exists \kappa_r>0 \mathrm{~s.t.~}
\dist \bracket{z, \Omega} \le \kappa_r \dist \bracket{\zero, T (z)}, ~\forall z\in \bB(\bar z;r),
\end{align}
for some $\bar z\in \Omega$.
 Li et al.~\cite{Li2019An} computed the linear convergence rate of the inexact preconditioned PPA ($\gamma= 1$)  under the following error bound condition:
\begin{equation}\label{eq:satt}\forall r>0,\exists \kappa_r>0 \mathrm{~s.t.~}
\dist \bracket{z, \Omega} \le \kappa_r \dist \bracket{0, T (z)}, ~\forall  z\in \bB(\Omega;r). 
\end{equation}
The next lemma gives the relation between these different conditions and Assumption~\ref{assume:growthcondition}. We write  $\mathrm{A} \Longrightarrow \mathrm{B}$ if condition A implies condition B and $\mathrm{A} \Longleftrightarrow \mathrm{B}$ is condition A is equivalent to condition B.
\begin{lemma}\label{lem:relationconditions}
$$
~\eqref{upperLip0} \Longleftrightarrow~\eqref{a:Robinson}  \Longrightarrow ~\eqref{eq:satt}  \Longrightarrow \mathrm{Assumption}~\ref{assume:growthcondition} \Longrightarrow ~\eqref{a:Leventhal}.
$$
\end{lemma}
\begin{proof}
The only nontrivial part to prove is \eqref{a:Robinson} $ \Longrightarrow$ ~\eqref{eq:satt}  . By the corollary in~\cite{Robinson1979PM},~\eqref{a:Robinson} implies the existence of  $r>0$ and $\kappa_r>0$ such that
\[
\dist(z, \Omega) \le  \kappa_{r} \dist(\zero,T(z)),
\]
for any $z$ satisfying $\dist(\zero,T(z)) \le r$. By~\cite[Lemma 2]{Li2019An}, the last condition implies~\eqref{eq:satt}.
\end{proof}
In particular, if $T$ is  a polyhedral multifunction, then~\eqref{a:Robinson} holds~\cite{Robinson1979PM} and thus Assumption~\ref{assume:growthcondition} holds.   The class of maximal monotone operators satisfying Assumption~\ref{assume:growthcondition} is much larger than the class of polyhedral multifunctions. We refer to~\cite{PPAcon} and more recently~\cite{YuanADMM,necoara2019linear} for more dicussion on this topic, which is out of the scope of this paper. 

 \subsection{Conditional linear convergence of IGPPA}
In this subsection, we establish the linear convergence of IGPPA under Assumption~\ref{assume:growthcondition}.  In the following context, without further specification, $\{z^k\}_{k \ge 0}$ denotes the sequence generated by~\eqref{eq:GPPAframework}. Let $r>0$ be any upper bound on the  right-hand side of~\eqref{r:def}, and $\kappa_r$ be the constant satisfying~\eqref{growthcondition} in  Assumption~\ref{assume:growthcondition}. Then, for  any sequence of proximal regularization parameters $\{\sigma_k\}_{k\geq 0}$, and any sequence of error parameters $\{\delta_k\}_{k\geq 0}$,
 \begin{equation}\label{eq:dfdf}
 \dist \bracket{z^k,\Omega} \le \kappa_r \dist \bracket{0, T \bracket{z^k}},\enspace \forall k\geq 0.
 \end{equation}
 We recall the following critical property for proving the linear convergence of IGPPA. It follows directly from~\cite[Lemma 5.3]{Min2016On}  by considering the maximal monotone operator $\cM^{-1}T$ in the Hilbert space $\cX$ with inner product $\langle \cdot,\cdot \rangle_\cM$, and  was also proved in~\cite[Theorem 2]{Li2019An}. 
 \begin{lemma}[\cite{Min2016On,Li2019An}\label{lemma:key}] For any $k\ge 0$, we have
\begin{equation}\label{eq:key}
\dist_{\cM}\bracket{\cJ_{\sigma_k \cM^{-1} T} (z^k), \Omega} \le \frac{\kappa_r }{\sqrt{\sigma_k^2 + \kappa_r^2 }}\dist_{\cM}\bracket{z^k, \Omega}. 
\end{equation}
\end{lemma}
Based on Lemma~\ref{lemma:key}, we obtain the following recursive inequality on the distance to the solution set.
\begin{theorem}\label{thm:linearrate} For any $k\geq 0$, we have 
\[
\dist_{\cM} \bracket{z^{k+1},  \Omega} \le \rho_k \dist_{\cM} \bracket{z^k,  \Omega},
\]
with 
\begin{align}\label{a:rhok}
 \rho_k: = \frac{1}{1-\delta_k}\bracket{ \sqrt{1-\frac{  \min \{ \gamma, 2 \gamma - \gamma^2 \} \sigma_k^2}{ \sigma_k^2 + \kappa_r^2  }} + \delta_k \bracket{ \frac{ \min\{\gamma,1\} \kappa_r  }{ \sqrt{\sigma_k^2 + \kappa_r^2 }}  +1}} .
 \end{align}
\end{theorem}
\begin{remark} 
When $\gamma=1$, the rate~\eqref{a:rhok} reduces to the rate obtained in~\cite[Theorem 2]{Li2019An}.
\end{remark}
\begin{remark}
In the exact case, i.e., when $\eta_k\equiv 0$ and $\delta_k\equiv 0$, the rate~\eqref{a:rhok} reduces to the rate given in~\cite[Theorem 3.5]{Min2016On}, established for the generalized PPA under the  Lipschitz continuity assumption of $T^{-1}$ at $ \bf 0$. 
\end{remark}

Note that the factor $\rho_k$  given in~\eqref{a:rhok} depends explicitly on the error parameter $\delta_k$, and implicitly on the error parameter $\eta_k$ via $\kappa_r$. In addition, $\rho_k$ increases with $\delta_k$, and for any parameter $\sigma_k>0$, there is $\delta_k>0$ such that $\rho_k <1$ and the linear convergence of IGPPA holds. This corresponds to the commonly known  fact that the linear convergence is guaranteed if the subproblem $w \approx \cJ_{\sigma \cM^{-1} T} (z)$ is solved with sufficiently high accuracy.  On the other hand, when the error parameter $\delta_k$ is fixed in $[0,1/2)$, we can also make $\rho_k<1$ by choosing a sufficiently large proximal regularization parameter $\sigma_k$. 

Hereinafter, for simplicity we take constant $\delta_k\equiv \delta \in [0, {1}/{2})$  for all $k\geq 0$.  Let $\alpha>0$ such that \begin{equation}\label{eq:rho}
\rho:=\frac{1}{1-\delta}\bracket{ \sqrt{1-\frac{  \min \{ \gamma, 2 \gamma - \gamma^2 \} \alpha^2}{ \alpha^2 + 1} }+ \delta \bracket{  \frac{ \min\{\gamma,1\}}{\sqrt{\alpha^2+1}} +1} }<1.
\end{equation}
It is easy to see that for any $k\geq 0$,  if $\sigma_k\geq \kappa_r \alpha$, then $\rho_k \leq \rho$. Hence,
 we have the following  corollary.
   \begin{corollary}\label{cor:dfe} 
   If \begin{align}\label{a:sigmabound3}\sigma_k \geq \kappa_r \alpha,\enspace \forall k\geq 0,\end{align} then
\[
\dist_{\cM} \bracket{z^{k},  \Omega} \le  \rho^k\dist_{\cM} \bracket{z^0,  \Omega},\enspace \forall k\geq 0.
\]
\end{corollary}
   
 Corollary~\ref{cor:dfe} establishes the \textit{conditional linear convergence} of IGPPA. If $\kappa_r$ is known, then we choose $\{\sigma_k\}_{k\geq 0}$ satisfying~\eqref{a:sigmabound3} and the algorithm converges linearly with rate $\rho$. However, in general,  $\kappa_r$ is not known.  Recall that for each iteration $k$, we need to find an approximate solution 
$ 
w \approx \cJ_{\sigma_k \cM^{-1} T} (z)
$.
In principle, the larger $\sigma_k$ is,  the harder the inner problem is (for concrete examples, see Section~\ref{subsection:overall}). Then, we have  the dilemma to deal with: on the one hand, we tend to choose very large  parameters $\{\sigma_k\}_{k\geq 0}$  so that~\eqref{a:sigmabound3} holds to guarantee the linear convergence; on the other hand, we do not want excessively large $\{\sigma_k\}_{k\geq 0}$ in order to control   the inner problem  complexity.

 \subsection{Verification of linear convergence} \label{section:key}
	We shall rely on the following  property which relates the unknown value $\dist_{\cM} \bracket{z^{k},  \Omega}$ with the computable value $\norm{z^{k+1} - z^k}_{\cM}$.
\begin{proposition}\label{prop:leftright}For any $k\geq 0$, we have
\[
\frac{1-\delta}{\gamma}\norm{z^{k+1} - z^k}_{\cM} \le  \dist_{\cM} \bracket{z^{k},  \Omega} \le \frac{1 + \delta}{\gamma \left (1-  \sqrt{\frac{\kappa_r^2}{\sigma_k^2 + \kappa_r^2 }}\right)}\norm{z^{k+1} - z^k}_{\cM}.
\]
\end{proposition}
\begin{corollary}\label{cor:checkablecondition}
  If~\eqref{a:sigmabound3} holds, then
\begin{equation}\label{eq:GPPAlinearrate3}
\norm{z^{k+1} - z^k}_{\cM} \le  C \rho^k \norm{z^{1} - z^0}_{\cM},\enspace\forall k\geq 0,
\end{equation}
where
$$
C:=\frac{1 + \delta}{\bracket{1-\delta}\bracket{1-   \sqrt{\frac{ 1}{\alpha^2 +  1}}}}.
$$
\end{corollary}
Corollary~\ref{cor:checkablecondition} is practically more interesting than Corollary~\ref{cor:dfe} since~\eqref{eq:GPPAlinearrate3} can always be verified at each iteration.

Let $\cE:\cX \rightarrow \bR_{+}$ be a computable error residual function that we use to measure the approximation to $\Omega$.  We shall assume the existence of a constant $\zeta>0$ such that 
\begin{align}\label{a:cEl}
\cE (z) \le \zeta \dist \bracket{z,\Omega},~\forall z \in \cX.
\end{align}
 If inequality~\eqref{eq:GPPAlinearrate3} holds for all $k\geq 0$, then  $\left\{\norm{z^{k+1} - z^k}_{\cM}\right\}_{k\ge0}$ and consequently $\{\cE (z^k)\}_{k\ge0}$  decreases linearly with rate $\rho$.
\begin{proposition}\label{prop:d}
If~\eqref{eq:GPPAlinearrate3} holds, then $\cE (z^k) \le \epsilon$ after 
\[
k\ge 
\log_{\frac{1}{\rho}}\bracket{\frac{R (\sigma_k)\zeta \dist_{\cM}\bracket{z^0,\Omega}}{\lambda_{\min}(\cM) \epsilon}}
\]
number of IGPPA steps, where
\begin{equation}\label{eq:R}
R (\sigma):= \frac{C\bracket{1 + \delta}}{ \bracket{1-\delta}\bracket{1-  \sqrt{ \frac{\kappa_r^2 }{\sigma^2 + \kappa_r^2 }}}}.
 \end{equation}
\end{proposition}

If instead~\eqref{eq:GPPAlinearrate3} does not hold, we know that~\eqref{a:sigmabound3} is false.   Then, we get a certificate of $\sigma_k$ being too small,  which suggests us to increase the next parameter $\sigma_{k+1}$. 
Note that Corollary~\ref{cor:checkablecondition} can be strengthened as follows.
\begin{corollary}\label{cor:checkablecondition2}
If~\eqref{a:sigmabound3} holds, then
\[
\norm{z^{k+1} - z^k}_{\cM} \le  C  \min_{0\leq j\leq k}\rho^{k-j} \norm{z^{j+1} - z^j}_{\cM},\enspace\forall k\geq 0.
\]
\end{corollary}

\section{Adaptive Generalized PPA  (AGPPA)}\label{section:AGPPA}
In this section, we apply the results in the previous section to  adaptively choose the proximal regularization parameters.
We propose a double loop algorithm with  $s$ and $t$ being respectively the number of outer  and inner iterations. 
Given an accuracy parameter $\epsilon>0$, the objective is to find a solution $z\in \cX$ such that $\cE (z)\leq \epsilon$. 
Choose some
$
\eta_0 > 0, \varsigma > 1, \varrho_{\eta} \in  (0,1)
$
, and define the sequence  $\{ \eta_{s,t} \}_{s\ge0, t\ge0}$  as
\begin{equation}\label{varepsilon:update}
\eta_{0,0}=\eta_0, \eta_{s+1,0} = \eta_{s,0} \varrho_{\eta},~\eta_{s,t} = \eta_{s,0} \bracket{1 + t}^{-\varsigma},~\forall s \ge 0, t \ge 0.
\end{equation}
Choose some  $\sigma_0>0$ and $\rho_\sigma>1$, and  define the sequence  $\{ \sigma_{s} \}_{s\ge0}$  as
 \begin{equation}\label{eq:updatesigma}
 \sigma_{s+1} = \sigma_s \varrho_{\sigma},~\forall s\ge0.
 \end{equation}
We generate the sequence  $\{ z^{s,t} \}_{s\ge0, t\ge0}$ from  an arbitrary  initial point $z^{0,0}$ by 
\begin{equation}\label{eq:generate_points}
\left\{\begin{array}{l}
z^{s,t+1} = \GPPAstep \bracket{z^{s,t}, \eta_{s,t}, \sigma_s, \delta, \gamma, \cM},~\forall   s\ge0, N_s\ge t \ge 0,\\
z^{s+1,0}=\arg\min \{ \cE \bracket{z} :  z\in \{z^{s,0},\dots,z^{s,N_s+1} \}\},
\end{array}
\right.
\end{equation}
where $N_s$ is the smallest $t$ such that either 
\begin{align}\label{a:cEE}
\cE (z^{s,t}) \le \epsilon,
\end{align}
or
\begin{align}\label{a:drerer}
\norm{ z^{s,t+1} - z^{s,t}}_{\cM}>  C \min_{0 \le j \le t} \left\{\rho^{t -j} \norm{ z^{s,j+1}- z^{s,j}}_{\cM}\right\}.
\end{align}
At each outer iteration $s$, we run  IGPPA with parameter $\sigma_s$ until either ~\eqref{a:cEE} or~\eqref{a:drerer} holds.   The sequence of the proximal regularization parameters $\{\sigma_s\}_{s \ge 0}$ is increased by a fixed factor $\rho_\sigma$  when~\eqref{a:cEE} holds before~\eqref{a:drerer} is reached. Each outer iteration starts from an iteration point that minimizes the error residual function $\cE$ among all the past iteration points.
We call the  algorithm~\eqref{eq:generate_points}  Adaptive Generalized Proximal Point Algorithm  (AGPPA). 
An equivalent description of AGPPA is given in Algorithm~\ref{alg:AGPPA}.
 \renewcommand{\thealgorithm}{1}
\begin{algorithm}
	\caption{$\AGPPA$}
	\label{alg:AGPPA}
	 \textbf{Input:} $\epsilon>0$, $z^0\in \cX$ 
	 
	 \textbf{Parameters:}  $\cM\in \cS^{++}$, $\eta_0 > 0$, $\varsigma > 1$, $\varrho_{\eta} \in  (0,1)$, $\sigma_0>0$,  $\varrho_{\sigma} >1$, $\alpha > 0$, $\delta \in [0, 1/2)$, $\gamma \in  (0,2)$
	 
	 \textbf{Initialize:}  $z^{0,0} = z^0, \eta_{0,0} = \eta_0, s=0$
	\begin{algorithmic}[1]
	\While{$\cE (z^{s,0}) > \epsilon$}
	\State $t = -1$
	\Repeat \label{eq:repeat}
	\State $t=t+1$
	\State $\eta_{s,t} = \eta_{s,0} \bracket{1 + t}^{-\varsigma}$
	\State \label{line:subproblem}$z^{s,t+1} =  \GPPAstep (z^{s,t}, \sigma_s, \eta_{s,t}, \delta,\gamma,\cM)$
	\Until{$\norm{ z^{s,t+1} - z^{s,t}}_{\cM}>  C \min_{0 \le j \le t} \left\{\rho^{t -j} \norm{ z^{s,j+1}- z^{s,j}}_{\cM}\right\}$ {\bf or} $\cE (z^{s,t}) \le \epsilon$ \label{line:checkingcondition} }
	\State
	$N_s = t $
			\If{$\cE (z^{s,t}) \le \epsilon$},  $z^{s+1,0} = z^{s,N_s}$
		\Else
         \State $z^{s+1,0}=\arg\min \{ \cE \bracket{z} :  z\in \{z^{s,0},\dots,z^{s,N_{s}+1} \}\}$
                \State $\sigma_{s+1} = \sigma_{s} \varrho_{\sigma}$ 
	\State $\eta_{s+1,0} = \eta_{s,0} \varrho_{\eta}$
	               \EndIf
         \State $s=s+1$
	\EndWhile
	\State $z_{\circ}=z^{s,0}$
	\If{$s == 0$}
	        \State $ s_{\circ}= 0$
	        \State $N_{\circ}=0$
     \Else
		\State $ s_{\circ}= s-1$
		\State $N_{\circ}=\sum_{s=0}^{ s_{\circ}}  (N_s+1)$
		\EndIf
	\State $\sigma_{\circ}=\sigma_{ s_{\circ}}$
	\State $\eta_{\circ}=\eta_{s_{\circ}, N_{s_\circ}}$
	\end{algorithmic}
	\textbf{Output:} $ (z_{\circ},    N_{\circ},  \sigma_{\circ},  \eta_{\circ} ,s_{\circ})$
\end{algorithm} 

Algorithm~\ref{alg:AGPPA} takes two inputs: an accuracy parameter $\epsilon>0$ and an initial point $z^0\in \cX$.  It terminates when an approximate solution $z_{\circ}\in \cX$ such that  $\cE (z_{\circ})\leq \epsilon$ is found. The output reports the solution $z_{\circ}$, as well as the total number of IGPPA steps  $N_{\circ}$, the last proximal regularization  parameter $\sigma_{\circ}$, the last error  parameter $\eta_{\circ}$, and the total number of outer iterations $ s_{\circ}$.  Note that $\rho$ in Line~\ref{line:checkingcondition}  of Algorithm~\ref{alg:AGPPA} is defined in~\eqref{eq:rho} using the parameters.  

Next we show that without any knowledge of the bounded metric subregularity parameter $\kappa_r$, the total number of IGPPA steps $N_{\circ}$ required before finding a solution $z$ such that  $\cE (z)\leq \epsilon$ satisfies
$$
N_{\circ}\leq O\left ( \ln\kappa_r \ln  \frac{ r \kappa_r  }{\epsilon}\right).
$$
\begin{theorem}
\label{thm:AGPPA}
 Suppose that the parameters required in  Algorithm~\ref{alg:AGPPA} are chosen such that $\rho$ defined in~\eqref{eq:rho} is strictly less than $1$. 
 For any initial point  $z^0$, there is a constant $\kappa_r>0$ such that for any $\epsilon>0$,
Algorithm~\ref{alg:AGPPA} terminates with output $ (z_{\circ},    N_{\circ},  \sigma_{\circ},  \eta_{\circ} ,s_{\circ})$ satisfying
\begin{equation}
\cE (z_{\circ}) \leq \epsilon,
\end{equation}
\begin{align}
& s_{\circ}\leq  \bar s:=\left\lceil \max\left ( \log_{\varrho_{\sigma}} \bracket{ \frac{\kappa_r  \alpha}{\sigma_0}} ,0\right)\right\rceil, \label{a:bars} 
\end{align}
\begin{align}
&\sigma_{\circ}\leq \bar\sigma:= \alpha \kappa_r \varrho_{\sigma}, \label{a:sigma} 
\end{align}
\begin{align}
&  N_{\circ}\leq \bar N:= \bar s \left\lceil \max\left (\log_{\frac{1}{\rho}}\bracket{\frac{\bar R}{ \epsilon}}, 0 \right)+1\right\rceil, \label{a:N}
\end{align}
\begin{align}
& \eta_{\circ}\geq \bar \eta:= \eta_0\varrho_{\eta}^{\bar s} \left\lceil \max\left (\log_{\frac{1}{\rho}}\bracket{\frac{\bar R}{ \epsilon}}, 0 \right)+1\right\rceil ^{-\varsigma}, \label{a:eta}  
\end{align}
where 
\begin{align}\label{a:defR}
\bar R:=\frac{  (1+\delta)^2\zeta r  }{ (1-\delta)^2 \bracket{1-   \sqrt{\frac{ 1}{\alpha^2 +  1}}}\bracket{1-  \sqrt{ \frac{\kappa_r^2 }{\sigma_0^2 + \kappa_r^2 }}} \lambda_{\min}(\cM)},
\end{align}
with $r$ given by
\begin{equation}\label{r:redef}
 r =\norm{\bar z^0} + \frac{1}{ \lambda_{\min}(\cM) } \left (\dist_{\cM} (z^{0}, \Omega) +  \frac{ \gamma\varsigma \eta_0}{ (\varsigma -1) (1-\varrho_{\eta})} \right).
\end{equation} 
\end{theorem}
\begin{proof}
It is easy to verify that $\{ \eta^{s,t} \}_{s \ge 0, t \ge 0}$  defined by~\eqref{varepsilon:update} is summable:
\[
 \sum_{s=0}^{\infty} \sum_{t = 0}^{\infty} \eta_{s,t}  \le   \frac{\varsigma \eta_0}{ (\varsigma -1) (1-\varrho_{\eta})} .
\]
 Therefore, the sequence $\{z^{s,t}\}_{s \ge 0, t \ge 0}$ generated by AGPPA~\eqref{eq:generate_points} falls into the general  framework as described in Theorem~\ref{GPPA:convergence}, and for any $s \ge 0, 0\leq t\leq N_s+1$, 
\begin{align}\label{a:zstb}
\norm{z^{s,t}}\le  \norm{\bar z^{0,0}} + \frac{1}{ \lambda_{\min}(\cM) } \left (\dist_{\cM} (z^{0,0}, \Omega) + \gamma 
 \sum_{s=0}^{\infty} \sum_{t = 0}^{\infty} \eta_{s,t} \right) \leq r.
 \end{align}  
Same as~\eqref{eq:dfdf}, 
it derives from Assumption~\ref{assume:growthcondition} that there is $\kappa_r>0$ such that for any  $s\geq 0, 0\leq t\leq N_s+1$,
\[
 \dist \bracket{z^{s,t},\Omega} \le \kappa_r \dist \bracket{0, T (z^{s,t})}.
\]
 Then, we learn from Proposition~\ref{prop:d} that  for any $ s\geq 0$,
 \begin{equation}\label{a:Ns0}
 \begin{aligned}
& N_s  \leq  \left\lceil \max \left\{ \log_{\frac{1}{\rho}}\bracket{\frac{R (\sigma_s)\zeta \dist_{\cM}\bracket{z^{s,0},\Omega}}{\lambda_{\min}(\cM) \epsilon}} , 0 \right\} \right\rceil \\
&\leq \left\lceil \max \left\{ \log_{\frac{1}{\rho}}\bracket{\frac{R (\sigma_0)\zeta r }{\lambda_{\min}(\cM) \epsilon}}, 0 \right\} \right\rceil.
\end{aligned}
 \end{equation}
  Here, the second inequality is based on~\eqref{a:zkzo},~\eqref{a:lambdamaxM1}, and the fact that $R (\sigma_s)\leq R (\sigma_0)$.
  Plugging in~\eqref{a:Ns0} the definition~\eqref{eq:R},  we obtain the following bound
  on the number of inner iterations: for any $ s\geq 0$,
   \begin{align}\label{a:Ns}
   N_s\leq \left\lceil \max\left (\log_{\frac{1}{\rho}}\bracket{\frac{\bar R}{ \epsilon}}, 0 \right)\right\rceil,
 \end{align}
 with $\bar R$ defined in~\eqref{a:defR}.
 Moreover, 
 in view of~\eqref{eq:updatesigma},  once 
  the outer iteration counter $s$ satisfies:
 $$
 s\geq    \left\lceil \max\left ( \log_{\varrho_{\sigma}} \bracket{ \frac{\kappa_r  \alpha}{\sigma_0}} ,0\right)\right\rceil,
 $$
 we will have
 $$
 \sigma_s\geq \kappa_r  \alpha.
 $$
Based on Corollary~\ref{cor:checkablecondition2},
 condition~\eqref{a:drerer} will not occur  when $\sigma_s\geq \kappa_r \alpha$, and the algorithm terminates at this outer iteration. Consequently, we obtain~\eqref{a:bars} and hence~\eqref{a:sigma}. 
  The bound~\eqref{a:N} on 
 the total number of IGPPA steps  $N_{\circ}$ derives from~\eqref{a:bars} and~\eqref{a:Ns}.  
 To obtain~\eqref{a:eta},  it suffices to note that
 $$
 \eta_{\circ}=\eta_{ s_{\circ},0} (1+N_{s_\circ})^{-\varsigma}\geq \eta_{0} \varrho_{\eta}^{\bar s} (1+\bar N)^{-\varsigma}.
 $$
\end{proof}
\begin{remark}
Note that the output $ (z_{\circ},    N_{\circ},  \sigma_{\circ},  \eta_{\circ} ,s_{\circ})$ of Algorithm~\ref{alg:AGPPA} can be random. Indeed, as we shall see below, the algorithm to solve the IGPPA step can be a randomized method and in this case, all the sequences produced by Algorithm~\ref{alg:AGPPA} are random. However, their bounds $ (\bar N, \bar \sigma, \bar \eta, \bar s)$ given 
in Theorem~\ref{thm:AGPPA}  are all deterministic.
\end{remark}

In the definitions~\eqref{a:bars},~\eqref{a:sigma},~\eqref{a:N} and~\eqref{a:eta}, the constants $\rho$, $\sigma_0$, $\varrho_{\sigma}$, $\alpha$, $\gamma$, $\delta$, $\eta_0$, $\varsigma$, $\varrho_{\eta}$ and $\lambda_{\min}(\cM)$ are all user-defined parameters.   
  The constant $\kappa_r$ is such that~\eqref{growthcondition} holds with $r$ defined in~\eqref{r:redef}. In the following context, to get a better understanding on the complexity, 
 we  ignore the user-defined constants to  extract out of~\eqref{a:bars},~\eqref{a:sigma},~\eqref{a:N} and~\eqref{a:eta} the
dependence on $r$,  $\kappa_r$, $\zeta$ and $\epsilon$. In particular, we consider  the user-defined constants $\rho$, $\sigma_0$, $\varrho_{\sigma}$, $\alpha$, $\gamma$, $\delta$, $\eta_0$, $\varsigma$, $\varrho_{\eta}$ and $\lambda_{\min}(\cM)$ as problem-independent constants.
\begin{corollary}\label{cor:wrd}
We have:
\begin{align*}
&\bar s=O\left (\ln \kappa_r\right),\\
& \bar \sigma= O\left ( \kappa_r\right),\\
& -\ln\bar \eta=O\left ( \ln \kappa_r+\ln\ln\left (\frac{\zeta  r \kappa_r  }{\epsilon}\right)\right),\\
&\bar N=O\left ( \ln\kappa_r \ln  \left (\frac{\zeta r \kappa_r  }{\epsilon}\right)\right),
\end{align*}
where the big O hides the problem-independent constants.
\end{corollary}

\section{Iteration complexity of proximal method of multipliers}\label{sec:app}
In this section, we apply AGPPA to solve the convex optimization problem:
\begin{equation}\label{prob:constrained}
\min_{x\in \bR^n}f_0 (x)+g (x)+h (Ax).
\end{equation}
Here, $A\in \bR^{m\times n}$, $f_0:\bR^n \rightarrow \bR$ is a convex and differentiable function, and the functions
 $g, h:\bR^n\rightarrow
\bR\cup \{+\infty\}$ are  proper, closed and  convex.  In addition, we assume that $h$ is a simple function, in the sense that its proximal operator can be easily computed. We write down the Lagrangian function as follows:
\begin{align}\label{a:defl}
\ell\bracket{x,\lambda} := f_0 (x) + g (x)+ \langle Ax, \lambda \rangle - h^*\bracket{\lambda},\enspace\forall x\in \bR^n, \lambda\in \bR^m,
\end{align}
where $h^*$ denotes the Fenchel conjugate function of $h$. Let $\cX:=\bR^{n+m}$ and
define the multivalued mapping $T_{\ell}: \cX\rightrightarrows \cX$  associated with the convex-concave function $\ell$ by
\begin{equation}\label{def:Tell}
T_{\ell} ( x, \lambda) := \left\{ (v, u) \in \bR^{n+m}:  (v, -u) \in \partial \ell (x,\lambda) \right\},\enspace\forall x\in \bR^n, \lambda\in \bR^m.
\end{equation}
It is known that $T_{\ell}$ is a maximal monotone operator and the  set $T_{\ell}^{-1} (\zero,\zero)$
is the set of saddle points of the Lagrangian function~\eqref{a:defl}, which then yields the primal and dual optimal solutions of the optimization problem~\eqref{prob:constrained} (see~\cite[Theorem 3.4.1]{LECTURENOTES}).

To apply AGPPA, we shall make the following assumptions on problem~\eqref{prob:constrained}.
\begin{assume}\label{ass:AGPPAF}
\begin{enumerate}
\item $\Omega:=T_{\ell}^{-1} (\zero,\zero)$ is nonempty.
\item The operator $T_{\ell}$
 satisfies the bounded metric subregularity as given in Assumption~\ref{assume:growthcondition}.
 \item  There is a  computable
 residual function $\cE:\cX \rightarrow \bR_+$   satisfying~\eqref{a:cEl} for some $\zeta>0$.
\end{enumerate}
\end{assume}

Based on Section~\ref{subsection:EB}, we know that if $T_{\ell}$ is a polyhedral multifunction, Assumption~\ref{assume:growthcondition} holds.  In addition, if $f_0, g, h$ are piecewise linear-quadratic, $T_{\ell}$ is a  polyhedral multifunction~\cite[Theorem 11.14, Proposition 12.30]{Rockafellar1998Variational}. 
Therefore, our algorithm is applicable to a wide range of convex optimization problems, such as the LP problem, the quadratic programming problem, the LASSO problem and so on. However,  we expect that the AGPPA framework can be applied to solve more examples of convex programming other than piecewise linear-quadratic programs.  
Verifying the bounded metric subregularity~\eqref{growthcondition} for the operator $T_{\ell}$ associated with problem~\eqref{prob:constrained} is  out of the scope of this paper, 
but we point out that the techniques in~\cite{YuanADMM,necoara2019linear} seem to shed light on this challenging task. 

 Under Assumption~\ref{ass:AGPPAF}, Algorithm~\ref{alg:AGPPA} can be applied to find a solution in $\Omega$. In the remaining of this section, we concretize the IGPPA step for  $T_{\ell}$, and
study the total complexity of AGPPA applied to find  an approximate
 solution of $\Omega$.

\subsection{Implementation of the IGPPA step}\label{sec:F}
Let $\cM^x \in S_n^{++}$, $\cM^{\lambda}\in S_m^{++}$ and
\[
\cM = \begin{pmatrix}
\cM^x &0 \\ 0 & \cM^{\lambda} 
\end{pmatrix}\in S_{n+m}^{++}.
\]
Let $\eta>0$ and $\delta>0$.
We fix base points $\bar x\in \bR^n$ and $\bar\lambda\in \bR^m$, and present an implementable form of
the IGPPA step for the maximal monotone operator $T_{\ell}$.
Recall that the IGPPA step amounts to  compute $\tilde x\in \bR^n$ and $\tilde \lambda \in \bR^m$ such that
\begin{equation}\label{GPPA:stop1}
	\norm{ (\tilde x,\tilde \lambda)- \cJ_{\sigma \cM^{-1} T_{\ell}} (\bar x, \bar\lambda)}_{\cM} \le \min \left\{\eta,\delta \norm{ (\tilde x,\tilde \lambda)- (\bar x,\bar\lambda)}_{\cM} \right\}.
	\end{equation}
	Note that~\eqref{GPPA:stop1}  cannot be verified.  
 Let 
\begin{align}\label{a:psi}\psi (u, \bar\lambda, \sigma):=\max_{\lambda \in \bR^m} \left\{
\langle u, \lambda \rangle - h^*\bracket{\lambda}- \frac{1}{2\sigma }\norm{ \lambda-\bar\lambda }_{\cM^{\lambda}}^2
\right\}.
\end{align}
Recall that  $\psi (\cdot,\bar\lambda,\sigma): \bR^m\rightarrow \bR^m$ is known as a  smoothing approximation of the possible nonsmooth function $h$ (see~\cite{Yu2005Smooth}). 
Furthermore, denote
\begin{equation}\label{def:LAMBDA}
\Lambda (x, \bar \lambda, \sigma):=\arg\max_{ \lambda \in \bR^m}  \left\{  \ell ( x, \lambda) - \frac{1}{2\sigma }\norm{ \lambda-\bar\lambda }_{\cM^{\lambda}}^2   \right\}.
\end{equation}
We define
 \begin{equation}\label{eq:defF}
\begin{aligned}
 F (x, \bar x, \bar \lambda, \sigma)&:=\max_{\lambda\in \bR^m}  \left\{  \ell ( x, \lambda) - \frac{1}{2\sigma }\norm{ \lambda-\bar\lambda }_{\cM^{\lambda}}^2   \right\} + \frac{1}{2\sigma } \norm{  x-\bar x}_{\cM^{x}}^2 \\
 &=f_0 ( x)+g ( x)+ \psi (Ax, \bar\lambda, \sigma)+ \frac{1}{2\sigma } \norm{  x- \bar x}_{\cM^{x}}^2 .
\end{aligned}
\end{equation}
In this subsection, since $\sigma>0$ and the base points $\bar x\in \bR^n$ and $\bar\lambda\in \bR^m$ are fixed, we denote for simplicity $F(\cdot):= F (\cdot, \bar x, \bar \lambda, \sigma)$.
The function $F(\cdot)$ is strongly convex  and has a unique minimizer $x^\star$. It is known that~\cite{Rockafellar1976Augmented}
\begin{align}\label{a:dfere}
  (x^{\star}, \Lambda (x^{\star}, \bar \lambda,\sigma))= \cJ_{\sigma \cM^{-1} T_{\ell}} (\bar x,\bar\lambda),
\end{align}
 and  thus the computation of an inexact minimizer of $F$  yields an inexact solution of the
 resolvent operator.
The following proposition is a generalization  of~\cite[Proposition 8]{Rockafellar1976Augmented}.
\begin{proposition}[\cite{Rockafellar1976Augmented}]\label{prop:stopguarantee}
For any $\tilde x\in \dom (F)$, we have
\[
  \norm{ (\tilde x,  \Lambda (\tilde x, \bar \lambda,\sigma) ) - \cJ_{\sigma \cM^{-1} T_{\ell}} (\bar x,\bar\lambda)}_{\cM}  
  \le \frac{\sigma}{ \sqrt{\lambda_{\min}(\cM^x)}} \dist (\zero, \partial F (\tilde x) ).
\]
\end{proposition}

We obtain directly from Proposition~\ref{prop:stopguarantee}  the following condition that ensures~\eqref{GPPA:stop1}:
\begin{align}	\label{stop:computable}
         \dist (\zero, \partial F (\tilde x) ) \le  \frac{ \sqrt{\lambda_{\min}(\cM^x)}}{\sigma} \min \left\{ \eta, \delta \norm{
	 \left ( \tilde x, \Lambda \bracket{\tilde x, \bar \lambda, \sigma}\right) - \left ( \bar x, \bar  \lambda\right)}_{\cM} \right\}.
	\end{align}
Hence, the IGPPA step reduces to  solve approximately:
\begin{align}\label{a:mF}
F^\star=\min_{ x\in \bR^n} F (x),
\end{align}
so that~\eqref{stop:computable} is satisfied. In the following, we assume that $ \dist (\zero, \partial F (\tilde x) )$ can be easily computed. This is satisfied  for example when AGPPA is applied to LP problem (see Section~\ref{LP:computable}). Thus, we say the stopping criterion~\eqref{stop:computable} is computable.

\subsection{Iteration complexity of the IGPPA step}\label{sec: ICIGPPA}
In this subsection, we investigate the complexity of finding an approximate 
minimizer of~\eqref{a:mF}
so that the computable stopping criterion~\eqref{stop:computable} is satisfied.  We  decompose the function $F$ defined in~\eqref{eq:defF} into two parts:
\begin{equation}\label{decompose:F}
F ( x)=f ( x)+\phi (x),
\end{equation}
where 
\[
f (x):=f_0 ( x)+ \psi (Ax,  \bar \lambda, \sigma),\enspace \phi ( x):=g ( x)+ \frac{1}{2\sigma } \norm{  x- \bar x}_{\cM^{x}}^2,~\forall x \in \bR^n.
\]
By~\cite[Theorem 1]{Yu2005Smooth}, we have
\begin{equation}\label{eq:nablaf}
\nabla f(x) = \nabla f_0(x) + A^{\top} \Lambda (x, \bar \lambda, \sigma),
\end{equation}
where $\Lambda$ is defined  in~\eqref{def:LAMBDA} and has concrete form as
\begin{equation}\label{eq:Lambdaconcrete}
 \Lambda (x, \bar \lambda, \sigma) = \cJ_{\sigma \cM_{\lambda}^{-1} \partial h^*} \bracket{\sigma \cM_{\lambda}^{-1} Ax + \bar \lambda}.
\end{equation}
We require the following additional assumption on the function $f_0$.
\begin{assume}\label{ass:f0L}
The gradient $\nabla f_0:\bR^n\rightarrow \bR^n$ is $L_0$-Lipschitz continuous.
\end{assume}
Then, it is easy to see that $\nabla f$ is $L$-Lipschitz continuous with
\begin{equation}\label{eq:Lforf}
L := L_0 + \frac{\sigma}{\lambda_{\min}(\cM^\lambda)} \norm{A}^2,
\end{equation}
and thus function $F$ can be decomposed into smooth function $f$ with $L$-Lipschitz continuous gradient and possibly non-smooth function $\phi$ that is strongly convex.

Our analysis is based on the existence of a globally convergent minimizer of $F$ that satisfies  the \textit{homogenous objective decrease}  (HOOD) property~\cite{NIPS2016_6364}. More precisely, we require the following assumption.
\begin{assume}\label{assume:Hood}
There is  an algorithm $\cA_F: \dom (g)\rightarrow \dom (g)$, which,  in a fixed number of operations,  returns  an   output  satisfying:
 \begin{align}\label{a:erd}
\bE[F (\cA_F (x))-F^\star] \leq e^{-1}\bracket{F\bracket{x}-F^\star},\enspace \forall x\in \dom (g).
\end{align}
\end{assume}

There are numerous first-order methods which can be applied to solve problem~\eqref{a:mF} with linear convergence and thus satisfy Assumption~\ref{assume:Hood},  including the proximal gradient method~\cite{Nesterov2013Gradient} and its accelerated versions~\cite{Nesterov2013Gradient,beck2009fast,tseng2008accelerated,jiang2012inexact,li2015accelerated}.  If the function $\phi(x)$ can be separated as $\phi(x) \equiv \sum_{i =1}^n \phi_i(x_i)$, then  the randomized coordinate descent (RCD) methods~\cite{nesterov2012efficiency,richtarik2016parallel,alacaoglu2017smooth} or the accelerated versions~\cite{lin2015accelerated,fercoq2015accelerated,Fercoq2018Restarting} are applicable. If the function $f(x) \equiv \frac{1}{m} \sum_{i = 1}^m f_i(x)$,  then the stochastic gradient descent (SGD) algorithms~\cite{schmidt2017minimizing,defazio2014saga,zinkevich2010parallelized} or the accelerated versions~\cite{nitanda2016accelerated,johnson2013accelerating,Allen2016Katyusha} are preferable. Thus, the inner solver  to solve problem~\eqref{a:mF} can be a randomized algorithm, which motivates us to require~\eqref{a:erd} in expectation.

The following two challenges are raised:
\begin{enumerate}
\item A sequence $\{ x^k \}_{k \ge 0}$ converging to the unique solution point of the problem~\eqref{a:mF}  does not necessarily satisfy $$\lim \inf_{k \rightarrow \infty } \dist \bracket{ {\bf 0}, \partial F(x^k)} \rightarrow 0.$$
\item  
First-order methods may fail to give  solution with high accuracy in practice. In addition, first-order methods  may suffer from slow convergence for ill-conditioned problems, especially when the proximal regularization parameter becomes large.
\end{enumerate}

Define the \textit{proximal gradient operator}
\begin{equation}\label{eq:ProximalGradientOperator}
\cG_F ( x) :=\arg\min_{y \in \bR^n}\left\{\frac{L}{2}\left\|y-\left (x-\frac{1}{L}\nabla f ( x)\right)\right\|^2 +\phi ( y) \right\},\enspace \forall x\in \bR^n.
\end{equation}
To overcome the first challenge, we apply the operator $\cG_F$ at each iteration of $\cA_{F}$ since if $\left\{x^{k}\right\}_{k \ge 0} \rightarrow x^{\star}$, then $ \dist \bracket{ {\bf 0}, \partial F(\cG_F(x^k))} \rightarrow 0$. The details will be shown  below. To overcome the second challenge, we propose a hybrid way to use any numerically efficient algorithm for the minimization
problem~\eqref{a:mF} (see Section~\ref{subsection:hybrid}).

We propose Algorithm~\ref{AlgAlg:APPROX} for finding $\tilde x$ satisfying~\eqref{stop:computable}.
 \renewcommand{\thealgorithm}{2}
\begin{algorithm}
	\caption{}
	\label{APPROX}
	\begin{algorithmic}[1]
	\Ensure $x^0 = \cG_F (\bar x)$, $k = 0$
	\Repeat 
	\State $y^k = x^k$
	\State $y^{k+1} \leftarrow \cA_F (y^k)$
	\If{$F (y^{k+1}) > F (y^k)$}
	  \State $y^{k+1} = y^k$
	\EndIf
	\State $ x^{k+1}=\cG_F (y^{k+1})$
	\State $k = k+1$
	\Until{ $ \dist (\zero, \partial F ( x^k) ) \leq  \frac{ \sqrt{\lambda_{\min}(\cM^x)}}{\sigma}  \min \left\{ \eta, \delta \norm{
	 \left (x^k, \Lambda \bracket{x^k, \bar \lambda,\sigma}\right) - \left ( \bar x, \bar \lambda\right)}_{\cM} \right\}$ }
	 \State $\tilde x =  x^k$
	\Require $ \tilde x$
	\end{algorithmic}
	\label{AlgAlg:APPROX}
\end{algorithm}
Note that~\eqref{a:erd} only holds in expectation and thus in general we 
can not guarantee $F (\cA_F (y^k)) \leq F (y^k)$. To ensure monotone decrease, we always compare 
$F (\cA_F (y^k))$ with $F (y^k)$, and let $y^{k+1}$ be $y^k$ if $F (\cA_F (y^k))> F (y^k)$. 
In Algorithm~\ref{APPROX}, we also need to apply the operator $\cG_F$ at each iteration. The motivation 
comes from the following result.
\begin{proposition}\label{prop:fdd}
Let $y\in \bR^n$ be a vector satisfying
\begin{equation}\label{hungry_stop1}
F (y) - F^\star \le  \frac{\lambda_{\min}(\cM^x)}{\sigma^2}\min \left\{ \frac{\eta^2}{4 L },  \frac{\delta^2\lambda_{\min}(\cM)}{2 (1+\delta)^2 L^2} \left ( F (\cG_F (\bar x)) - F^\star \right) \right\}.
\end{equation}
Then, the stopping criterion~\eqref{stop:computable}   holds with $\tilde x=\cG_F (y)$.
\end{proposition}
\begin{proof} 
 By~\eqref{a:dfere},   for any $ (x,\lambda)$ we have,
\begin{align}\label{a:tidf}
\norm{ x- x^{\star}}\leq  \norm{ (x,\lambda)- (x^{\star},\Lambda (x^{\star},\bar\lambda,\sigma)) } = \norm{ ( x, \lambda)-   \cJ_{\sigma \cM^{-1} T_{\ell}} (\bar x,\bar\lambda)  } .
\end{align}
We recall two properties about the proximal gradient operator~\cite[(2.13),  (2.21)]{Nesterov2013Gradient}:
\begin{align}\label{hungry_eq11}
&\frac{L}{2} \norm{\cG_F (y) - y}^2 \le F (y) -F^\star, \\
&F (\cG_F (\bar x)) - F^\star  \le \frac{L}{2} \norm{\bar x- x^{\star}}^2 \overset{\eqref{a:tidf}}\le \frac{L}{2}\norm{ ( \bar x, \bar \lambda)-   \cJ_{\sigma \cM^{-1} T_{\ell}} (\bar x,\bar\lambda)  }^2 .\label{hungry_eq12}
 \end{align}
Note that we have ${\bf 0} \in  \nabla f (y)   + L (\cG_F (y) -y) + \partial \phi (\cG_F (y))$ with~\eqref{eq:ProximalGradientOperator}, thus
\[
 \nabla f (\cG_F (y))-\nabla f (y) -L (\cG_F (y) -y)  \in \partial F ( \cG_F (y)).
\]
Then we have
\[\begin{aligned}
 &\dist ( {\bf 0}, \partial F (\cG_F (y)) )^2 {\le} 
 \norm{\nabla f (\cG_F (y))-\nabla f (y) -L (\cG_F (y) -y)}^2  \\
 &\le  2L^2 \norm{\cG_F (y) -y}^2  
 \overset{\eqref{hungry_eq11}}{\le}
 4L  (F (y) -F^\star)\\
&\overset{\eqref{hungry_stop1}}{\le} \frac{\lambda_{\min}(\cM^x)}{\sigma^2 } \min \left\{ \eta^2,  \frac{2\delta^2\lambda_{\min}(\cM)}{ (1+\delta)^2  L} \left ( F (\cG_F ( \bar x)) - F^\star \right) \right\},  \\
&\overset{\eqref{hungry_eq12}}{\le} \frac{\lambda_{\min}(\cM^x)}{\sigma^2 }
\min \left\{ \eta^2,\frac{\delta^2\lambda_{\min}(\cM)}{ (1+\delta)^2}  \norm{ (\bar x,\bar\lambda)-   \cJ_{\sigma \cM^{-1} T_{\ell}} (\bar x,\bar\lambda)  }^2 \right\}\\
&\le \frac{\lambda_{\min}(\cM^x)}{\sigma^2 } \min \left\{ \eta^2,\frac{\delta^2 }{ (1+\delta)^2} \norm{ (\bar x,\bar\lambda)-   \cJ_{\sigma \cM^{-1} T_{\ell}} (\bar x,\bar\lambda)  }_\cM^2 \right\},
\end{aligned}\]
and therefore,
\begin{equation}\label{a:erdvcv}
\begin{aligned}
&\dist  (\zero, \partial F (\tilde x)) \\
&\le \frac{\sqrt{\lambda_{\min}(\cM^x)}}{\sigma} \min \left\{ \eta ,\frac{\delta }{1+\delta}   \norm{ (\bar x,\bar\lambda)-  \cJ_{\sigma \cM^{-1} T_{\ell}} (\bar x,\bar\lambda)  }_\cM\right\}.
\end{aligned}
\end{equation}
In view of Proposition~\ref{prop:stopguarantee}, we have
\begin{align*}
 &\norm{ (\tilde x,  \Lambda (\tilde x,\bar\lambda,\sigma) ) - \cJ_{\sigma \cM^{-1} T_{\ell}} (\bar x,\bar\lambda)}_{\cM} \le \frac{\sigma}{\sqrt{\lambda_{\min}(\cM^x)}} \dist \bracket{0, \partial F (\tilde x)} \\
 & \overset{\eqref{a:erdvcv}}\le \min \left\{ \eta, \frac{\delta}{1+\delta}  \norm{ (\bar x,\bar\lambda)-  \cJ_{\sigma \cM^{-1} T_{\ell}} (\bar x,\bar\lambda)  }_\cM \right\}  
\end{align*}
It follows that
\begin{align*}
&\norm{ (\tilde x,  \Lambda (\tilde x, \bar x, \sigma) ) - \cJ_{\sigma \cM^{-1} T_{\ell}} (\bar x,\bar\lambda)}_{\cM} 
\\& \leq
\frac{\delta}{1+\delta} \left ( \norm{ (\tilde x,\Lambda (\tilde x, \bar x,\sigma))-   (\bar x,\bar\lambda)}_\cM+ \norm{ (\tilde x, 
\Lambda (\tilde x, \bar x,\sigma))-  \cJ_{\sigma \cM^{-1} T_{\ell}} (\bar x,\bar\lambda)  }_\cM \right),
\end{align*}
and thus
$$
\norm{ (\tilde x,  \Lambda (\tilde x, \bar x,\sigma) ) - \cJ_{\sigma \cM^{-1} T_{\ell}} (\bar x,\bar\lambda)}_{\cM} 
\leq \delta \norm{ (\tilde x,\Lambda (\tilde x, \bar x,\sigma))-   (\bar x,\bar\lambda)}_\cM.
$$
Consequently, 
\begin{align}\label{a:erere}
\norm{ (\bar x,\bar\lambda)-  \cJ_{\sigma \cM^{-1} T_{\ell}} (\bar x,\bar\lambda)  }_\cM \leq 
 (1+\delta)\norm{ (\tilde x,\Lambda (\tilde x, \bar x,\sigma))-   (\bar x,\bar\lambda)}_\cM.
\end{align}
Plugging~\eqref{a:erere} into~\eqref{a:erdvcv}, we get
\[
\dist  (\zero, \partial F (\tilde x)) \le \frac{\sqrt{\lambda_{\min}(\cM^x)}}{\sigma}\min \left\{ \eta, \delta \norm{ (\tilde x,\Lambda (\tilde x, \bar x,\sigma))-   (\bar x,\bar\lambda)}_\cM\right\}.
\]
\end{proof}

Now we give an upper bound on the number of iterations before Algorithm~\ref{APPROX} terminates.
\begin{theorem}\label{thm:subcomp} Let Assumption~\ref{ass:f0L} and Assumption~\ref{assume:Hood} hold. For any $0<p<1$, denote 
\begin{align}\label{a:Kpf}
 K (p) := \left\lceil \max \left ( \ln \bracket{\zeta_1\bracket{ \frac{L\sigma \dist\bracket{ (\bar x,\bar\lambda), \Omega} }{\eta \sqrt{p}}}}, \ln \bracket{\zeta_2\bracket{\frac{L\sigma}{\sqrt{p}}}},0 \right)\right\rceil,
 \end{align}
 with functions $\zeta_1, \zeta_2: \bR^{+} \rightarrow \bR^{+}$ defined as:
 \[
\zeta_1(q):=\frac{2q^2}{\lambda_{\min}(\cM^x)},~\zeta_2(q): = \frac{2 (1+\delta)^2 q^2}{\delta^2\lambda_{\min}(\cM^x)\lambda_{\min}(\cM)}.
 \]
Then, with probability  at least $1-p$, Algorithm~\ref{APPROX} terminates within $K (p)$ iterations.
\end{theorem}
\begin{proof}
Since $F (\cG_F (y^{k+1}))\leq F (y^{k+1})$ for any $k\geq 0$, it is obvious that
$$
\bE[F (x^k)-F^\star]\leq e^{-k} \left (F (x^0)-F^\star\right)=e^{-k} \left (F (\cG_F (\bar x))-F^\star\right), \enspace \forall k\geq 0.
$$
If 
$$
k\geq  \ln \bracket{\zeta_2\bracket{\frac{L\sigma}{\sqrt{p}}}} 
,
$$
then
\begin{equation}\label{a:etsaa}
\begin{aligned}
&\bE[F (x^k)-F^\star] \leq \frac{1}{ \zeta_2\bracket{L\sigma/\sqrt{p}}} \left ( F (\cG_F (\bar x)) - F^\star \right)
.
\end{aligned}
\end{equation}
We apply~\eqref{hungry_eq12} to obtain
$$
\bE[F (x^k)-F^\star]\leq\frac{L}{2e^k}\norm{ (\bar x,\bar\lambda)-   \cJ_{\sigma \cM^{-1} T_{\ell}} (\bar x,\bar\lambda)  }^2, \enspace \forall k\geq 0.
$$
By Proposition~\ref{prop:useful}~\ref{third},
 we have
\[
\norm{ (\bar x,\bar\lambda)-   \cJ_{\sigma \cM^{-1} T_{\ell}} (\bar x,\bar\lambda)  } \leq \dist \bracket{ (\bar x,\bar\lambda), \Omega }.
\]
Therefore,
\[
\bE[F (x^k)-F^\star]\leq\frac{L}{2e^k} \dist ^2\bracket{ (\bar x,\bar\lambda), \Omega}, \enspace \forall k\geq 0.
\]
It follows that if 
$$k\geq \ln \bracket{\zeta_1\bracket{\frac{L\sigma \dist\bracket{ (\bar x,\bar\lambda), \Omega}}{ \eta \sqrt{p}}}} = \ln  \frac{2L^2\sigma^2 \dist^2\bracket{ (\bar x,\bar\lambda), \Omega} }{p \eta^2\lambda_{\min}(\cM^x)}, $$
then
\begin{equation}\label{eq:smile1}
\bE[F (x^k)-F^\star]\leq \frac{p\eta^2\lambda_{\min}(\cM^x)}{4\sigma^2 L }.
\end{equation}
Based on~\eqref{a:etsaa} and~\eqref{eq:smile1}, we know that for any $k\geq K (p)$,
\[\begin{aligned}
\bE[F (x^k)-F^\star] \leq \min \left \{  \frac{p\eta^2\lambda_{\min}(\cM^x)}{4\sigma^2 L },  \frac{1}{ \zeta_2\bracket{L\sigma / \sqrt{p}}} \left ( F (\cG_F ( \bar x)) - F^\star \right)\right \}.
\end{aligned}\]
Using Markov's inequality we deduce that for any $k\geq K (p)$, with probability at least $1-p$, we have
\[F (x^k)-F^\star \leq  \min \left \{ \frac{\eta^2\lambda_{\min}(\cM^x)}{4\sigma^2 L },  \frac{\delta^2\lambda_{\min}(\cM^x)\lambda_{\min}(\cM)}{2 (1+\delta)^2 L^2\sigma^2} \left ( F (\cG_F ( \bar x)) - F^\star \right)\right \}.
\]
Now it suffices to apply Proposition~\ref{prop:fdd} and the result is proved.
\end{proof}

\begin{remark}\label{rm:ic}
At each iteration of Algorithm~\ref{APPROX}, it is required  to run the following operations:
\begin{enumerate}
\item Computation of $\cA_F (y^k)$;
\item Computation of $F (y^{k+1})$;
\item Computation of $\cG_F (y^{k+1})$;
\item Computation of $\dist (\zero, \partial F (x^k))$;
\item Computation of $\Lambda (x^k, \bar x, \sigma)$ and of $\norm{
	 \left ( x^k, \Lambda \bracket{x^k, \bar x,\sigma}\right) - \left (\bar x, \bar \lambda\right)}_{\cM} $.
\end{enumerate}
\end{remark}
The  per-iteration cost of Algorithm~\ref{APPROX} depends on the  function $F$ defined in~\eqref{decompose:F}, the condition number of which increases with the parameter $\sigma$. It is  reasonable to assume that the per-iteration cost  can be upper bounded by $\cT (\sigma)$, for some increasing function $\cT:\bR_+\rightarrow \bR_+$.
\begin{theorem}\label{thm:subcomp2}
Let  Assumption~\ref{ass:f0L} and Assumption~\ref{assume:Hood} hold and $\cT:\bR_+\rightarrow \bR_+$ be an increasing function such that the per-iteration cost  of Algorithm~\ref{APPROX} is bounded by $\cT (\sigma)$. Then, with probability at least $1-p$, 
the complexity of Algorithm~\ref{APPROX}  is bounded by 
 \begin{align}\label{a:K} O (K (p)\cT (\sigma)),\end{align}
with $K (p)$ defined in Theorem~\ref{thm:subcomp}.
\end{theorem}

\subsection{Hybrid inner solver}\label{subsection:hybrid}
In this subsection, we solve the second challenge posed in Section~\ref{sec: ICIGPPA} that the first-order methods may fail to give an approximate  solution $\tilde x$ to problem~\eqref{a:mF} satisfying~\eqref{stop:computable}  efficiently.
 There are many algorithms which  have high efficiency in practice to find  a  required $\tilde x$. For example, second-order methods 
 have local superlinear convergence property, and if the starting point is sufficiently 
close to the optimal solution, the convergence  is  very fast. However, algorithms with fast local convergence may fail to satisfy the HOOD property~\eqref{a:erd},  and thus  
adds inherent difficulty to the  complexity analysis.  Next, we discuss how to
  bridge the gap between the practical performance 
and the complexity analysis by
combining
Algorithm~\ref{APPROX} with any numerically efficient algorithm for the minimization
problem~\eqref{a:mF}.

Let $\cQ_F$ be any algorithm which is likely to  be more efficient than  Algorithm~\ref{APPROX}
 for solving~\eqref{a:mF}.
When extra computation resource is available, we can run 
 Algorithm~\ref{APPROX} in parallel with $\cQ_F$, and stop it when whichever returns a
 solution first. In this way, the iteration complexity is always upper bounded by~\eqref{a:K}, but the
 actual running time is  smaller than both  that of $\cQ_F$ and of Algorithm~\ref{APPROX}.
When there is no extra computation resource, we can first run $\cQ_F$, and switch it to  Algorithm~\ref{APPROX}  
if no solution satisfying~\eqref{stop:computable} is found within a limited number of elementary operations.   We describe this idea formally in Algorithm~\ref{cQ_cAcG}.
 \renewcommand{\thealgorithm}{3}
\begin{algorithm}
	\caption{
	}
	\label{cQ_cAcG}
	\textbf{Parameters:} $\cJ>0$
	\begin{algorithmic}[1]
	\Ensure $x^0 = \cG_F(\bar x), \bar x,~k = 0,~\cJ_0=\cJ 
$
	\Repeat 
	\State $ \bracket{ x^{k+1}, \cD} = \cQ_F\bracket{x^k, \cJ_{k}}$ \label{line:2}
	\State $\cJ_{k+1} = \cJ_k- \cD$
	\State $k=k+1$
	\Until{ $ \dist (\zero, \partial F ( x^k) ) \leq  \frac{ \sqrt{\lambda_{\min}(\cM^x)}}{\sigma}  \min \left\{ \eta, \delta \norm{
	 \left (x^k, \Lambda \bracket{x^k, \bar \lambda,\sigma}\right) - \left ( \bar x, \bar \lambda\right)}_{\cM} \right\}$  {\bf or} $\cJ_{k}\leq 0 $ }
	\If{$\cJ_k\leq 0$}
	\State Run Algorithm~\ref{APPROX} and obtain $\tilde x$
	\Else
	\State $\tilde x = x^k $
	\EndIf  
	\Require $ \tilde x$
	\end{algorithmic}
\end{algorithm}

 It should be emphasized that  there is no constraint on the choice of the algorithm $\cQ_F$. 
In principle, we opt for an algorithm, which, according to empirical observation,  is likely to require less computational time than Algorithm~\ref{APPROX}. We set a maximum number 
of elementary operations $\cJ$ allowed for running algorithm $\cQ_F$. 
The step in Line~\ref{line:2} of Algorithm~\ref{cQ_cAcG}  
corresponds to one iteration of the algorithm $\cQ_F$. The second output $\cD$ 
records the number of elementary operations used in this iteration.  
It is obvious that as long as we set $\cJ\leq O (K (p) \cT (\sigma))$, the complexity of Algorithm~\ref{cQ_cAcG}
is also upper bounded by~\eqref{a:K} with probability at least $1-p$.  
In conclusion, Algorithm~\ref{cQ_cAcG}  may benefit from 
the  fast convergence of $\cQ_F$ in practice, while  maintaining the same complexity bound as Algorithm~\ref{APPROX}. The remaining question is how to set $\cJ$. We can set
\begin{equation}\label{eq:cJsetting}
\cJ = O \bracket{ \max \left\{ \ln \bracket{ \zeta_{2} \bracket{\frac{L\sigma}{\sqrt{ p'}}}}, 0 \right\} \cT(\sigma) },
\end{equation}
with $p'$ computable and $p' \ge p$. For example, we can simply set $p' = 1$. Note that $\cT(\sigma)$ is computable (see the details in Section~\ref{subsection:overall}).  Thus $\cJ$ is computable and based on the definition of $K(p)$ as in~\eqref{a:Kpf}, we have
\[
\cJ \le O \bracket{K(p) \cT(\sigma)}.
\]

\subsection{Overall complexity of AGPPA}\label{subsection:overall}
In the previous sections, we showed that  the IGPPA step can reduce to
solving the optimization problem~\eqref{a:mF}, and analyzed the complexity 
assuming the existence of an algorithm $\cA_F$ satisfying the HOOD property~\eqref{a:erd} with per-iteration cost of Algorithm~\ref{APPROX} bounded by $\cT(\sigma)$.  
In this section, we assemble the previous results to give the overall complexity of AGPPA.  

Recall that at each outer iteration $s$  and inner iteration $t$ of AGPPA, we run the IGPPA step
with  base points $z^{s,t}= (x^{s,t}, \lambda^{s,t})\in \bR^{n+m}$, proximal regularization parameter $\sigma_{s}$, and error parameter $\eta_{s,t}$.  Based on Section~\ref{sec:F}, 
this problem reduces to find an approximate solution $\tilde x$ of  the following minimization problem:
\begin{align}\label{a:mFst}
\min_{ x\in \bR^n} \left[ F_{s,t} (x)\equiv   F(x, x^{s,t}, \lambda^{s,t}, \sigma_s) \right],
\end{align}
so that
\begin{equation}\label{stop:computablest}
\dist (\zero,\partial F_{s,t} (\tilde x))\le\frac{\sqrt{\lambda_{\min}(\cM^x)}}{\sigma_s}\min\left \{\eta_{s,t}, \delta\norm{  (\tilde x, \tilde{\lambda})- (x^{s,t},\lambda_{s,t})}_{\cM}\right\},
\end{equation}
with $\tilde \lambda :=  \Lambda \bracket{\tilde x,  \lambda_{s,t}, \sigma_s}$.
We decompose $F_{s,t}$  into 
\[
F_{s,t}(x) = f_{s,t}(x) + \phi_{s,t}(x),
\]
where
\[
f_{s,t} (x) :=  f_0 ( x)+ \psi (Ax,\lambda_{s,t}, \sigma_s),~\phi_{s,t} (x) := g (x)+ \frac{1}{2\sigma_s } \norm{  x- x^{s,t}}_{\cM^{x}}^2,~\forall x \in \bR^n.
\]
Let Assumption~\ref{ass:f0L} hold, then $\nabla f_{s,t}$ is $L_{s,t}$-Lipschitz continuous with
\begin{align}\label{Lab}
L_{s,t}: = L_0+ \frac{\sigma_{s}}{\lambda_{\min} (\cM^{\lambda})}\norm{A}^2,\enspace \forall s\geq 0, t\geq 0.
\end{align}
Denote 
\[
F_{s,t}^{\star}: = \min_{ x\in \bR^n} F_{s,t} (x).
\] 
We shall make  essentially  the same assumptions for  the function $F_{s,t}$ as  for the function $F$ 
in Section~\ref{sec: ICIGPPA}. More precisely, we additionally require the following assumptions.
\begin{assume}\label{assume:Hoodst}
\begin{enumerate}
 \item There exists an algorithm, named as $\cA_{F_{s,t}}$ for every inner problem~\eqref{a:mFst}, which returns an output satisfying
\[
\bE[F_{s,t} (\cA_{F_{s,t}} (x))-F_{s,t}^\star] \leq e^{-1}\bracket{F_{s,t}\bracket{x}-F_{s,t}^\star},\enspace \forall x\in \dom (g),
\]
within a fixed number of operations. 
\item The per-iteration of  Algorithm~\ref{APPROX} for every inner problem~\eqref{a:mFst} can be upper bounded by $\cT(\sigma_s)$ number of elementary iterations.
\end{enumerate}
\end{assume}

Then, we can apply Algorithm~\ref{APPROX} or Algorithm~\ref{cQ_cAcG}  to realize each IGPPA step, or equivalently to  find $\tilde x$ satisfying~\eqref{stop:computablest}. The complexity  of each IGPPA step derives directly from Theorem~\ref{thm:subcomp2}.
 For notational ease, denote
\begin{align}\label{def:ab}
  a := \frac{\norm{A}^2}{\lambda_{\min}(\cM^\lambda)}.
\end{align}
\begin{proposition}\label{prop:eff}
Under Assumption~\ref{ass:f0L} and Assumption~\ref{assume:Hoodst}, apply  AGPPA  (Algorithm~\ref{alg:AGPPA})  to the  maximal monotone operator $T_{\ell}$ defined in~\eqref{def:Tell}, with each IGPPA step being solved by Algorithm~\ref{APPROX} or 
Algorithm~\ref{cQ_cAcG}. Then,
with probability at least $1-p$, the complexity of the IGPPA step at  any iteration 
of AGPPA    is 
bounded by
\begin{align}\label{a:revb}
O\left (\left\lceil\max \left ( \ln  \bar \zeta_1, \ln \bar \zeta_2, 0 \right)\right\rceil \cT (\bar \sigma)\right),
\end{align}
where
\begin{equation}\label{def:zeta12}
\bar \zeta_1 := \frac{2   (L_0\bar \sigma+a\bar\sigma^2)^2r^2}{p \bar\eta^2\lambda_{\min}(\cM^x)},~\bar \zeta_2 := \frac{2 (1+\delta)^2 (L_0\bar \sigma+a\bar\sigma^2)^2}{p\delta^2\lambda_{\min}(\cM^x)\lambda_{\min}(\cM)},
\end{equation}
and $\bar \eta$, $\bar \sigma$ and $r$ are defined in  Theorem~\ref{thm:AGPPA} with $
z^0= (x^{0,0}, \lambda^{0,0})$ and $\Omega=T_{\ell}^{-1} (\zero,\zero)$.
\end{proposition}
\begin{proof}
Let $ (s,t)$ be any running iteration of AGPPA. More precisely,  let any $0\leq s\leq  s_{\circ}$ and $0\leq
t\leq N_{s}
$. 
 The complexity of the IGPPA step at iteration $ (s,t)$ of AGPPA is bounded by
 \begin{equation}\label{eq:Kst}
 K_{s,t} (p):= \left\lceil\max \left ( \ln \left( \zeta_1\bracket{\frac{L_{s,t}\sigma_s \dist((x^{s,t}, \lambda^{s,t}),\Omega) }{ \eta_{s,t} \sqrt{p}}}\right) ,   \ln\bracket{ \zeta_2\bracket{ \frac{L_{s,t} \sigma_s }{\sqrt{p}}}},0 \right)\right\rceil,
 \end{equation}
 which is obtained from~\eqref{a:Kpf}
 with $ (\bar x, \bar \lambda) =  (x^{s,t}, \lambda^{s,t})$,  $L=L_{s,t}$,
 $\sigma=\sigma_s$, and $\eta=\eta_{s,t}$.
First, it  derives from~\eqref{a:zkzo} that
\begin{align}\label{a:distr}
\dist\bracket{ (x^{s,t},\lambda^{s,t}), T^{-1}_{\ell} (\zero,\zero)}  \leq r.
\end{align}
Second, based on the update rule~\eqref{varepsilon:update}, we know that 
\begin{align}\label{a:etas}
\eta_{s,t}\geq \eta_{s_{\circ}, N_{ s_{\circ}}}\geq \bar \eta.
\end{align}
Third, for the proximal regularization parameter, we have
\begin{align}\label{a:sigmas}
\sigma_s\leq \sigma_{\bar s}\leq \bar \sigma.
\end{align}
By~\eqref{Lab},~\eqref{a:distr},~\eqref{a:etas}, and~\eqref{a:sigmas}, we have 
\begin{equation}\label{eq:zetabound}
 \zeta_1\bracket{\frac{L_{s,t}\sigma_s \dist((x^{s,t}, \lambda^{s,t}),\Omega) }{ \eta_{s,t} \sqrt{p}}} \le \bar \zeta_1,~\zeta_2\bracket{ \frac{L_{s,t} \sigma_s }{\sqrt{p}}} \le \bar \zeta_2,
\end{equation}
 and
 \begin{equation}\label{cTs:bound}
 \cT (\sigma_s) \overset{\eqref{a:sigmas}}{\leq} \cT (\bar \sigma).
\end{equation}
 Then 
\begin{equation}\label{Kst:bound}
 K_{s,t} (p) \overset{\eqref{eq:zetabound}}\le \left\lceil\max \left ( \ln \bar \zeta_1 ,   \ln \bar \zeta_2, 0 \right)\right\rceil.
\end{equation}
The  two bounds~\eqref{cTs:bound} and~\eqref{Kst:bound} above show that~\eqref{a:revb} is an upper  bound of~\eqref{a:K} with
 $ (\bar x, \bar \lambda) =  (x^{s,t} , \lambda^{s,t}), L = L_{s,t}, \sigma = \sigma_s$, and $\eta = \eta_{s,t}$.
\end{proof}

The proposition  above gives an upper bound on the complexity to solve each IGPPA step of  AGPPA with probability at least $1-p$.  In Theorem~\ref{thm:AGPPA}, we have a bound on the total number of IGPPA steps. Now it suffices
to combine these two bounds to show the total complexity of AGPPA. 
\begin{theorem}\label{thm:6} Let Assumption~\ref{ass:AGPPAF}, Assumption~\ref{ass:f0L} and Assumption~\ref{assume:Hoodst} hold. Then,
with probability at least $1-p$, AGPPA can find a solution $ (x,\lambda)$ satisfying
$$
\cE\left (x,\lambda\right) \leq \epsilon,
$$
within
\begin{align}\label{a:dfer}
O\left ( \bar N\left\lceil \ln \bar N + \max \bracket{ \ln \bar \zeta_1,  \ln \bar \zeta_2, 0}  \right\rceil \cT (\bar \sigma) \right)
\end{align}
number of elementary operations,  where
$\bar \zeta_1,\bar \zeta_2$ are defined in Proposition~\ref{prop:eff}, and $\bar N, \bar \sigma$ are defined in  Theorem~\ref{thm:AGPPA} with $
z^0= (x^{0,0}, \lambda^{0,0})$ and $\Omega=T_{\ell}^{-1} (\zero,\zero)$. \end{theorem}
\begin{proof}
By Proposition~\ref{prop:eff}, each IGPPA step terminates within
\[
O\left (\left\lceil \ln \bar N +\max \left ( \ln \bar \zeta_1, \ln  \bar \zeta_2, 0 \right)\right\rceil \cT (\bar \sigma)\right)
\]
number of elementary operations  with probability at least $1-p/\bar N$.  Based on Theorem~\ref{thm:AGPPA}, the total number of IGPPA steps cannot exceed $\bar N$.  Then  we conclude the result by applying the union bound property.
\end{proof}

We apply Corollary~\ref{cor:wrd} to get a better understanding of the complexity by removing problem-independent constants. Before that, let us add an assumption on $\cT$.

\begin{assume}\label{ass:4}
There are  constants $\vartheta_1>0$, $\vartheta_2 > 0$ and $\iota>0$ such that 
\begin{align}\label{a:cde}
\cT ( \sigma)=O (\vartheta_1\sigma^\iota+ \vartheta_2 \sigma^{\iota/2} + \Upsilon).
\end{align}
Here, the big $O$ hides the problem-independent constants.
\end{assume}

Here, we provide three examples for $\cA_F$ to solve problem~\eqref{decompose:F}, including AdaRES~\cite{fercoq2019adaptive}  for the general case, restarted APPROX~\cite{Fercoq2018Restarting}  for the case where $\phi(x) = \sum_{i=1}^n \phi_i(x_i)$ is separable, and  Katyusha~\cite{Allen2016Katyusha} for the case where $f(x) = \frac{1}{m}\sum_{i = 1}^m f_i(x)$.
\begin{example}\label{ex:AdaRES} Assume that per-iteration cost of Algorithm~\ref{APPROX} is determined by operation $\cA_F$, then
AdaRES~\cite{fercoq2019adaptive}  finds $\cA_F (x)$ satisfying  the HOOD property~\eqref{a:erd}, 
and the corresponding  $\cT$  satisfies  Assumption~\ref{ass:4} with
\begin{equation}\label{param:AdaRES}
\vartheta_1  = \norm{A}  \nnz (A),~\vartheta_2 = L_0\nnz (A), ~\Upsilon = \nnz (A),~\iota = 1.
\end{equation}
\end{example}
\begin{example}\label{ex:APPROX}
 Assume that the per-iteration cost of Algorithm~\ref{APPROX} is determined by operation $\cA_F$, $g(x) = \sum_{i=1}^n g_i(x_i)$ is separable, and $\cM^{x}$ is a diagonal matrix, then restarted APPROX~\cite{Fercoq2018Restarting}   finds $\cA_F (x)$ satisfying  the HOOD property~\eqref{a:erd}, and the corresponding  $\cT$  satisfies  Assumption~\ref{ass:4} with
\begin{equation}\label{param:APPROX}
\vartheta_1  = \max_{i \in [n]} \norm{a_i}  \nnz (A),~\vartheta_2 = L_0\nnz (A), ~\Upsilon = \nnz (A),~\iota = 1,
\end{equation}
where $a_i$ denotes the $i$th column vector of $A$.
\end{example}
Note that  if $g(x) = \sum_{i=1}^n g_i(x_i)$ is separable and $\cM^{x}$ is a diagonal matrix, then $\phi(x)$ is separable. For simplicity, in Example~\ref{ex:APPROX}, we only show the result of restarted APPROX in coordinatewise form (it can be extended to blockwise setting). The result can be easily obtained from~\cite{Fercoq2018Restarting} with the fact that  $f$ satisfies the expected separable overapproximation (ESO) assumption: 
 \[
 E \bracket{f(x + h_i e_i)} \le f(x) + \frac{1}{n} \bracket{ \langle \nabla f(x),  h\rangle +  \frac{1}{2} \sum_{i = 1}^{n} L_i h_i^2 },
 \]
 with
  \[
 L_i = L_0 + \frac{\sigma }{ \lambda_{\min} \bracket{\cM_{\lambda}}} \norm{a_i}^2,~\forall i = 1,\dots, n.
 \]
The ESO  property holds since $\nabla f$ is coordinate-wise Lipschitz continuous with vector $(L_1, \dots, L_n)$~\cite{Fercoq2018Restarting}:
\[
\norm{  \nabla_i  f (x + h_i e_i)  - \nabla_i f(x)} \le L_i \norm{h_i}, \forall i \in [n],~\forall h_i \in \bR,
\]
which can be directly derived from~\eqref{eq:nablaf} and~\eqref{eq:Lambdaconcrete}. Here, $\nabla_i f = e_i^{\top} \nabla f$ and $e_i$ is the unit vector with $i$-th coordinate being $1$.

\begin{example}\label{ex:Katyusha} Assume that the per-iteration cost of Algorithm~\ref{APPROX} is determined by operation $\cA_F$, $h^*(\lambda) = \sum_{i = 1}^{m}  h^*_i(\lambda_i)$ is separable and $\cM^{\lambda}$ is a diagonal matrix, then Katyusha~\cite{Allen2016Katyusha}  finds $\cA_F (x)$ satisfying  the HOOD property~\eqref{a:erd} 
and the corresponding  $\cT$  satisfies  Assumption~\ref{ass:4} with
\begin{equation}\label{param:Katyusha}
\vartheta_1  = \frac{ \norm{A}_F}{\sqrt{m}}  \nnz (A),~\vartheta_2 = L_0\nnz (A), ~\Upsilon = \nnz (A),~\iota = 1.
\end{equation}
\end{example}
The result in Example~\ref{ex:Katyusha} derives from the fact that 
\begin{equation}\label{ex3:eq1}
f(x)=\frac{1}{m} \sum_{i=1}^{m} f_i(x),
\end{equation}
and the gradient $\nabla f_i$ is $L_i$-Lipschitz continuous with
\begin{equation}\label{ex3:eq2}
L_i =  m \bracket{ L_0 +  \frac{\sigma}{\cM^{\lambda}_{ii}} \norm{A_i}^2},~\forall i \in [m],
\end{equation}
where
\[
f_i(x) = m \bracket{f_0(x) + \psi_i (A_i^T x, \bar \lambda_i, \sigma)},~\forall i \in [m],
\]
 \[\psi_i (u_i, \bar \lambda_i, \sigma) =  \max_{\lambda_i \in \bR} \left\{
u_i \lambda_i  - h^*_i\bracket{\lambda_i}- \frac{1}{2\sigma }\norm{ \lambda_i-\bar\lambda_i }_{\cM^{\lambda}_{ii}}^2
\right\}, ~\forall i \in [m],
\]
and $A_i$ is the $i$th row vector of $A$.
Note that~\eqref{ex3:eq1} is obtained from the assumption that $h^*(\lambda) = \sum_{i = 1}^{m}  h^*_i(\lambda_i)$ is separable, and $\cM^{\lambda}$ is a diagonal matrix. And $\nabla f_i$ is $L_i$-Lipschitz continuous with $L_i$ given  in~\eqref{ex3:eq2} since $\phi_i(\cdot, \bar \lambda_i, \sigma)$ is differentiable, and  the gradient is $\sigma/\cM^{\lambda}_{ii}$-Lipschitz continuous~\cite[Theorem 1]{Yu2005Smooth}.

\begin{corollary}\label{cor:complexity bound of AGPPA}
Under Assumption~\ref{ass:4}, the bound~\eqref{a:dfer} is 
\begin{equation}\label{bound1:eq}
\begin{aligned}
&O\left ( \left (\vartheta_1  (\alpha\varrho_{\sigma}  \kappa_r)^{\iota}+ \vartheta_2  (\alpha\varrho_{\sigma}  \kappa_r)^{\iota/2} + \Upsilon\right) \right) \times  \\
&O \left( \ln \kappa_r\ln\frac{\zeta r \kappa_r }{\epsilon} \ln \left ( \frac{r \kappa_r   (L_0+a)}{p} \ln \frac{\zeta  r \kappa_r }{\epsilon}\right)\right).
\end{aligned}
\end{equation}
Here, the big $O$ hides the problem-independent
constants.
\end{corollary}
\begin{remark}
Based on Example~\ref{ex:AdaRES}, Example~\ref{ex:APPROX} and Example~\ref{ex:Katyusha}, we have 
\[
\vartheta_2 = L_0\nnz (A), ~\Upsilon = \nnz (A),~\iota = 1
\]
 by choosing some appropriate inner solver to solve the subproblem~\eqref{a:mFst}.  Then, based on the definition of $a$ in~\eqref{def:ab}, the bound~\eqref{bound1:eq} given in Corollary~\ref{cor:complexity bound of AGPPA} further reduces to 
\begin{align}\label{a:floglo}
\cO \left (\vartheta_1 \kappa_r  \ln \kappa_r \ln  \frac{\zeta r  \kappa_r }{\epsilon}  \ln  \bracket{ \frac{r\kappa_r   (L_0+ \norm{A})}{p}  \ln \frac{\zeta r \kappa_r }{\epsilon}}\right).
\end{align}
In conclusion, we proved that the overall complexity of AGPPA applied to the convex constrained problem~\eqref{prob:constrained} is bounded by~\eqref{a:floglo} with probability at least $1-p$. We point out that 
\[
\vartheta_1 \le \norm{A},
\]
since we can always choose AdaRES method as the inner solver.
\end{remark}

\section{Application example} \label{section:LP}
We emphasize again that AGPPA and the given complexity results can be applied to  the convex programming problem~\eqref{prob:constrained} satisfying Assumption~\ref{ass:AGPPAF}, such as the linear-quadratic problem and the LASSO problem. We defer to future work the inclusion of other relevant models into the application pool of AGPPA. 
In this section, we illustrate the application of  AGPPA  for solving the LP problem. 
Following the same notations as~\cite{Yen2015Sparse}, we describe our LP problem as follows:
\begin{equation}\label{P}
\begin{aligned}
&\min_{x\in \bR^n}&& c^\top x \\
&\st&& A_Ix \le b_I\\
&&& A_Ex = b_E\\
&&& x_{1},\dots x_{n_b}\ge  0.
\end{aligned}
\end{equation}
Here, 
$c \in \bR^n, A_I \in \bR^{m_I \times n},  b_{I} \in \bR^{m_I}, A_E \in \bR^{m_E \times n}, b_{E} \in \bR^{m_E}$ and $n_b\in [n]$. Thus, we have $m_I$ inequalities, $m_E$ equalities and the first $n_b$ coordinates must be nonnegative. Denote $m: = m_I + m_E$, $b: = [b_I; b_E] \in \bR^m$, and $A:=[A_I; A_E] \in \bR^{m\times n}$. 

\subsection{Applicability of AGPPA}
\subsubsection{Bounded metric subregularity}
 We decompose $x\in \bR^n$ as $x=[x_b; x_I]$ with $x_b\in \bR^{n_b}$, and $\lambda \in \bR^m$ as $\lambda=[\lambda_I; \lambda_E]$ with $\lambda_I\in \bR^{m_I}$. Then, problem~\eqref{P} can be written as  a special case of the convex programming problem~\eqref{prob:constrained} with:
 \begin{equation}\label{LP:functions}
 f_0 (x) = c^{\top} x,~g(x) = \delta_{\{x_{b} \ge 0\}} (x),~h(u) = \delta_{U}(u),
 \end{equation}
 where
 \[
 U = \left\{ u \in \bR^m | u_1 \le b_1,\dots, u_{m_I} \leq b_{m_I},  u_{m_I +1} =b_{m_I+1}, \dots, u_m = b_m \right\}.
 \]
 The Fenchel  conjugate function of $h$ is given by:
 \begin{equation}\label{LP:conjugatefunctions}
h^*(\lambda) =  \langle b, \lambda \rangle + \delta_{\lambda_I \ge 0}(\lambda),~ \lambda \in \bR^m.
\end{equation}
Then, the associated Lagrangian function defined in~\eqref{a:defl} is:
\[
\ell (x, \lambda) = c^{\top}x  +\langle Ax ,\lambda\rangle - \bracket{ \langle b, \lambda \rangle + \delta_{\{\lambda_I \ge 0\}} (\lambda) }+ \delta_{\{x_{b} \ge 0\}} (x),
\]
and $T_{\ell}:\bR^{n+m}\rightarrow \bR^{n+m}$  defined in~\eqref{def:Tell} becomes:
\begin{equation}\label{eq:iddd}
T_{\ell}(x,\lambda)= \left\{ 
 (v,u)\in \bR^{n+m}\\
 \middle |
\begin{aligned}
& v\in c+A^\top \lambda+\partial \delta_{\{x_{b} \ge 0\}} (x) , \\
& u\in b-Ax+\partial \delta_{\{\lambda_I \ge 0\}} (\lambda) .
\end{aligned}  
\right\}.
\end{equation}
It is clear that $T_{\ell}$
is a polyhedral multifunction, and hence the bounded metric subregularity condition (Assumption~\ref{assume:growthcondition}) is satisfied.  We assume the existence of optimal solutions of the LP problem~\eqref{P}, which then implies the nonemptyness of the solution set: $$\Omega=T_{\ell}^{-1} (\zero,\zero)=\left\{ 
 (x,\lambda) \in \bR^{n+m}\\
 \middle |
\begin{aligned}
& x_b\geq 0 , \lambda_I\geq 0, \\
& -A^\top \lambda-c  \in  \partial \delta_{\{x_{b} \ge 0\}} (x),\\
&Ax - b \in \partial \delta_{\{\lambda_I \ge 0\}} (\lambda).
\end{aligned}  
\right\}.
$$
Note that the solution set $\Omega$ can be represented by the following system of linear inequalities: 
\begin{equation}\label{eq:omegas}\Omega=\left\{(x,\lambda) \in \bR^{n+m}\\
 \middle |
\begin{aligned}  &x_b\geq 0,  \lambda_I\geq 0, \\  &c^{\top}x + b^{\top}\lambda=0, \\
& [ A^{\top}\lambda +c]_{-}^{n_b}=0, \\  & [Ax - b]_{+}^{m_I}=0.
\end{aligned}
\right\}.\end{equation}
Denote by $\theta$  the smallest constant satisfying
\begin{align}\label{a:thetaSdef}
\dist \bracket{\bracket{x,\lambda}, \Omega}
\le \theta
\norm{ \left [c^{\top}x + b^{\top}\lambda; [ A^{\top}\lambda +c]_{-}^{n_b}; [Ax - b]_{+}^{m_I}\right ]},
\end{align}
for all $x,\lambda$ with $x_{b} \ge 0$ and $\lambda_I \ge 0$.  It is clear that $\theta$ is upper bounded by the Hoffman constant~\cite{Hoffman2015On} associated with system~\eqref{eq:omegas}. 

\begin{lemma}\label{LP:kappa bound} For any $x,\lambda$ with $x_{b} \ge 0$ and $\lambda_I \ge 0$, we have 
\begin{equation}\label{LP_kappa:eq1}
\dist \bracket{ (x,\lambda), \Omega}\le \theta \left ( \norm{ (x, \lambda)}^2 +1\right)^{1/2} \dist\bracket{0, T_{\ell} (x, \lambda)}.
\end{equation}
\end{lemma}
\begin{proof}
 Let any  $ (v,u) \in T_{\ell} (x,\lambda)$. By~\eqref{eq:iddd}, we have $v\leq c+A^\top \lambda$ and 
 $u\leq b-Ax$, from which we deduce that
 \begin{equation}\label{eq:LPTell}
 \begin{aligned}
&\| [A^\top \lambda +c]_-^{n_b} \|\leq \|  [v]_-^{n_b}\|, \\
&\| [Ax - b]_+^{m_I} \|\leq \|  [-u]_+^{m_I}\|.
\end{aligned}
\end{equation}
Furthermore, we also have $x^\top(-v+ c+A^\top \lambda)=0$ and $\lambda^\top(u+Ax-b)=0$, based on which we know that
 \begin{equation}\label{eq:LPTel2}
\begin{aligned}
&c^{\top}x + b^{\top}\lambda = v^{\top}x + u^{\top}\lambda.
\end{aligned}
\end{equation}
In view of~\eqref{a:thetaSdef}, we have  
\[\begin{aligned}
&\dist \bracket{ (x,\lambda), \Omega}\le\theta  \norm{[c^{\top}x + b^{\top}\lambda; [Ax - b]_+^{m_I}; [ A^{\top}\lambda+c]_{-}^{n_b}]} \\
&\overset{\eqref{eq:LPTell}+\eqref{eq:LPTel2}}\le\theta \norm{[v^{\top}x + u^{\top}\lambda; [-u]^{m_I}_{+}; [v]_{-}^{n_b}]}\\
&\le\theta \norm{[v^{\top}x + u^{\top}\lambda; u; v]} \\
&\le  \theta \left ( \norm{ (x, \lambda)}^2 +1\right)^{1/2}\norm{  (v,u)}, \\
\end{aligned}\]
which implies~\eqref{LP_kappa:eq1}.
 \end{proof}
 
 Lemma~\ref{LP:kappa bound} asserts that $T_\ell$ satisfies Assumption~\ref{assume:growthcondition} with  $\kappa_r$ bounded by:
 \begin{align}\label{a:kapprTell}
\kappa_r \le  \theta \bracket{ r^2+1}^{1/2}=O(\theta r).
\end{align}

\subsubsection{Computable stopping criterion}\label{LP:computable}
For simplicity, let $\cM = I$. Fix base points $\bar x\in \bR^n$ and $\bar\lambda\in \bR^m$, then the function $\psi$ defined in~\eqref{a:psi} is:
\[\begin{aligned}
&\psi (u, \bar\lambda, \sigma)&=&\max_{\lambda} \left\{
\langle u, \lambda \rangle - \bracket{ \langle b, \lambda \rangle + \delta_{\lambda_I\geq 0} (\lambda)}- \frac{1}{2\sigma }\norm{ \lambda-\bar\lambda }^2\right\}\\
&&=&\frac{1}{2\sigma}\norm{ \left[\bar \lambda + \sigma (u - b ) \right]^{m_I}_{+} }^2 -\frac{1}{2\sigma}\norm{\bar \lambda}^2.
\end{aligned}\]
And the function $\Lambda$ defined in~\eqref{def:LAMBDA} is:
\begin{equation}\label{LP:lambda}
\Lambda (x, \bar \lambda, \sigma) = \left[ \bar \lambda + \sigma (Ax-b) \right]^{m_I}_{+}.
\end{equation}
The function $F$ defined in~\eqref{eq:defF}  thus takes the following form:
\begin{equation}\label{LP:fphi}
F (x)  =  c^\top x+ \frac{1}{2\sigma}
\norm{\Lambda (x, \bar \lambda, \sigma)}^2 -\frac{1}{2\sigma}\norm{\bar \lambda}^2+  \frac{1}{2\sigma } \norm{ x -\bar x }^2 + \delta_{\{x_{b} \ge 0\}} (x),
\end{equation}
which can be written as 
$F (x)=f (x)+\phi (x)$
with 
\[\begin{aligned}
&f (x)\equiv c^\top x+ \frac{1}{2\sigma}
\norm{\Lambda (x, \bar \lambda, \sigma)}^2 -\frac{1}{2\sigma}\norm{\bar \lambda}^2,\\
&\phi(x)\equiv \frac{1}{2\sigma } \norm{ x -\bar x }^2 + \delta_{\{x_{b} \ge 0\}} (x).
\end{aligned}\]
The gradient of $f$ is:
\begin{equation} \label{LP:nablaf}
\nabla f (x)  = c+ A^\top \Lambda (x, \bar \lambda, \sigma),
\end{equation}
and thus
\begin{equation} \label{LP:partialF}
\begin{aligned}
&\dist (\zero,\partial F (x))\\
&= \sqrt{\sum_{i \in \cB} \left[\nabla_i f (x)+\frac{1}{\sigma} (x_i-\bar x_i)\right]^2_- +\sum_{i  \not \in \cB} \left (\nabla_i f (x)+\frac{1}{\sigma} (x_i-\bar x_i)\right)^2},
\end{aligned}
\end{equation}
with $\cB:= \left\{i\in [n_b]~|~x_i=0\right\} $ and $[\cdot]_{-}$ denoting the projection into $\bR_{-}$. Therefore, the stopping criterion~\eqref{stop:computable} is computable.

\subsubsection{Error residual function}\label{sec:erf}
With regard to the computable error residual function $\cE$, we provide three examples. 
\begin{example}[KKT-residual]
For any $x \in \bR^n$ and $\lambda \in \bR^m$, define the error residual function $\cE_1$ as
\[
 \cE_1 (x,\lambda) :=
\norm{\left[c^\top x + b^\top \lambda; [ A^\top \lambda +c]_{-}^{n_b}; [Ax - b]_{+}^{m_I}\right]},
\]
then $\cE_1$ satisfies~\eqref{a:cEl} with
\[
\zeta = \bracket{\norm{ [b;c]}^2 + \norm{A}^2}^{1/2}
\]
\end{example}

 \begin{example}[\cite{Li2019An}]
For any $x \in \bR^n$ and $\lambda \in \bR^m$, define the error residual function $\cE_2$ as
\begin{equation}\label{eq:normalized KKT residual}
 \cE_2 (x,\lambda) := \max \left\{\frac{|c^{\top} x + b^{\top} \lambda|}{1+|c^{\top} x | + |b^{\top} \lambda| }, \frac{\norm{ [Ax - b]_{+}^{m_I}}}{1+\norm{b}} , \frac{\norm{[c+ A^{\top}\lambda]_{-}^{n_b} }}{1+\norm{c}} \right\},
\end{equation}
then $\cE_2$ satisfies~\eqref{a:cEl} with
\[
\zeta = \max \left \{ \norm{ [b;c]}, \frac{\norm{A}}{1+\norm{b}},  \frac{\norm{A}}{1+\norm{c}} \right\}.
\]
\end{example}

\begin{example}[\cite{NIPS2017_ADMM,Yen2015Sparse}] For any $x \in \bR^n$ and $\lambda \in \bR^m$, define the error residual function $\cE_3$ as
\[
 \cE_3 (x,\lambda) := \max \left\{\frac{|c^{\top} x + b^{\top} \lambda|}{\max\{1,|c^{\top} x |\}},\norm{ [Ax - b]_{+}^{m_I}}_{\infty}, \norm{[c+ A^{\top}\lambda]_{-}^{n_b} }_{\infty} \right\},
\]
then $\cE_3$ satisfies~\eqref{a:cEl} with
\[
\zeta = \max \left \{ \norm{ [b;c]}, \max_{i \in [m]} \norm{a_i} , \max_{i \in [n]} \norm{A_i} \right\}.
\]
\end{example}
\begin{remark} The three examples of error residual functions  $\cE_1$,  $\cE_2$  and $\cE_3$ above are all computable and  satisfy~\eqref{a:cEl} for some $\zeta > 0$ upper bounded by
\begin{equation}\label{eq:upper zeta}
 \zeta\leq \norm{ [b;c]} + \norm{A}.
\end{equation}
\end{remark}
We just verified that Assumption~\ref{ass:AGPPAF} holds for the LP problem~\eqref{P} and the stopping criterion~\eqref{stop:computable} is computable, so we can apply AGPPA to solve it.

\subsection{Complexity results} 
In this subsection, we show  the complexity bound of AGPPA applied to the LP problem~\eqref{P}.
\subsubsection{Qualified inner solvers}
Based on~\eqref{LP:functions}, we know $g$ can be separated as
\[
g(x) = \sum_{i = 1}^{n} g_i(x_i),
\]
where for any $i \in [n_b]$,
\[
g_i(x_i) = \left\{
\begin{aligned} 
&0 && \operatorname{if} x_i \ge 0,\\
&+ \infty && \operatorname{otherwise},
\end{aligned}
\right.
\]
and for any $i = n_b+1, \cdots, n$, $g_i(x_i) = 0$. In addition, based on~\eqref{LP:conjugatefunctions}, we know $h^*$ can be separated as
\[
h^*(\lambda)  = \sum_{i=1}^m  h^*_i(\lambda_i),
\]
where for any $i \in [m_I]$,
\[
h^*_i(\lambda_i) = \left\{
\begin{aligned} 
&b_i\lambda_i && \operatorname{if} \lambda_i \ge 0,\\
&+ \infty &&  \operatorname{otherwise},
\end{aligned}
\right.
\]
and for any $i = m_I + 1, \cdots, m$, $h^*_i(\lambda_i) = b_i \lambda_i$.
Thus AdaRES~\cite{fercoq2019adaptive}, APPROX~\cite{Fercoq2018Restarting}, Katyusha~\cite{Allen2016Katyusha} are suitable algorithms for $\cA_F$ in Algorithm~\ref{AlgAlg:APPROX} based on the fact that $g$ and $h^*$ are separable. 

Therefore, if per-iteration cost of Algorithm~\ref{APPROX} is determined by operation $\cA_F$, the complexity bound of AGPPA applied to the LP problem~\eqref{P} derives directly from~\eqref{a:floglo}, by specifying the constants $\vartheta_1$, $\zeta$ and  $L_0$.
 Note that Assumption~\ref{ass:f0L} holds with 
 \begin{align}\label{a:LL0}
 L_0=0,
 \end{align} 
 the bound for $\zeta$ is already given by~\eqref{eq:upper zeta}, and $\vartheta_1$ is inner solver related.  It remains to show that per-iteration cost of Algorithm~\ref{APPROX} is determined by operation $\cA_F$. 
\subsubsection{Per-iteration cost of Algorithm~\ref{AlgAlg:APPROX}}

We claim that the per-iteration cost of Algorithm~\ref{AlgAlg:APPROX} for $F$ defined in~\eqref{LP:fphi} is determined by operation $\cA_F$ since the last four operations  in Remark~\ref{rm:ic}  cost at most $O(\nnz(A))$. Details are given in the following analysis.

Based on~\eqref{LP:lambda} and~\eqref{LP:nablaf}, we see that $
\Lambda (x, \bar \lambda, \sigma)$ and thus $\nabla f (x)$ can be computed in 
 $O (\nnz (A))$ operations. 
  Then, the computation of the function value $F (\cdot)$  defined in~\eqref{LP:fphi} is clearly upper bounded by $O(n)$ when $
\Lambda (x, \bar \lambda, \sigma)$  is known. In addition,  the proximal operator  $\cG_F$ defined in~\eqref{eq:ProximalGradientOperator} is 
$$
\cG_F (x)=\left[ \frac{L\sigma\left (x-\frac{1}{L}\nabla f (x)\right)+\bar x}{L\sigma+1}\right]^{n_b}_{+},
$$ 
which can be computed in $O (n)$ when $\nabla f (x)$ is known. 
   Finally $\dist (\zero,\partial F (x))$ defined in \eqref{LP:partialF} can also be computed in no more than $O (n)$ operations when $\nabla f (x)$ is known.
   Thereby,  the last four operations  in Remark~\ref{rm:ic}  cost at most $O(\nnz(A))$. 

Therefore, the upper bound of per-iteration cost  of Algorithm~\ref{AlgAlg:APPROX} for $F$ defined in~\eqref{LP:fphi}  satisfies~\eqref{a:cde} with parameters $\vartheta_1, \vartheta_2, \Upsilon$ and $\iota$ given in~\eqref{param:AdaRES},~\eqref{param:APPROX} and~\eqref{param:Katyusha} for $\cA_F$ being AdaRES, APPROX and Katyusha respectively. Then, we can choose  $\cA_F$ as the one from APPROX and Katyusha  with better complexity bound, which leads to
\begin{align}\label{a:ass4cons}
\vartheta_1 = \min \left\{ \max\limits_{i\in[n]} \norm{a_i}, \frac{ \norm{A}_F}{\sqrt{m}} \right\} \nnz(A).
\end{align}

\subsubsection{Overall complexity bound}
We are now in the position to deduce the overall complexity of AGPPA applied to the LP problem~\eqref{P}. Hereinafter, we shall measure the \textit{batch complexity}, which refers to the number of passes over data, i.e., the number of elementary operations divided by $\nnz(A)$.
\begin{theorem}\label{LP:bacth complexity}
Apply AGPPA to solve the LP problem~\eqref{P} with $\cM=I$ and the error residual function $\cE$ being $\cE_1, \cE_2$ or $\cE_3$ given in Section~\ref{sec:erf}. Let each IGPPA step, i.e., minimizing $F$ defined in~\eqref{LP:fphi}, be solved by Algorithm~\ref{AlgAlg:APPROX} or Algorithm~\ref{cQ_cAcG} with $\cA_F$ being the one from APPROX and Katyusha with better complexity. For any starting primal dual pair $(x^0,\lambda^0)$, let
\begin{align}\label{a:rdef}
r=\norm{(\bar x^0, \bar \lambda^0)} + \dist ((x^{0},\lambda^0) ,\Omega) +  \frac{ \gamma\varsigma \eta_0}{ (\varsigma -1) (1-\varrho_{\eta})},
\end{align}
with $(\bar x^0, \bar \lambda^0)$ being the projection of $(x^0,\lambda^0)$ into $\Omega$. Let  $\kappa_r$ be the constant satisfying~\eqref{growthcondition} for $T=T_{\ell}$.
With probability at least $1-p$, AGPPA finds a solution satisfying $\cE (x,\lambda)\leq \epsilon$ 
with the batch complexity bounded by
\begin{equation}\label{eq:LPcomplexity}
\cO \left (  \vartheta_1 \kappa_r \ln \kappa_r \ln  \left (\frac{ \zeta r \kappa_r }{\epsilon}\right)  \ln  \bracket{\frac{r \kappa_r  \norm{A}}{p}  \ln \frac{ \zeta r \kappa_r}{\epsilon}}\right),
\end{equation}
 where  $\kappa_r=O(\theta r)$, $\vartheta_1 = \min \left\{ \max\limits_{i\in[n]} \norm{a_i}, \frac{ \norm{A}_F}{\sqrt{m}} \right\}$ and $\zeta\leq  \norm{ [b;c]} + \norm{A} $. 
\end{theorem}
\begin{proof}
It suffices to plug in~\eqref{a:floglo} the estimations~\eqref{a:kapprTell},~\eqref{a:LL0},~\eqref{eq:upper zeta} and \eqref{a:ass4cons}.
\end{proof}

 \subsection{Comparison with related works}\label{section:LPcomplexity}

 In the past decades, numerous research works have been devoted to the development of numerical solutions for the LP problem. Commonly used LP solvers ubiquitously implement the interior-point method (IPM)~\cite{Nazareth2004The,Yurii2006Lectures,Kojima1988A} and the simplex method~\cite{Klee1970How,Simplex,Dantzig1990Origins}. These two classical methods are recognized to be highly efficient for low or medium sized LP problems.  However, their complexity bounds  are known to be at least quadratic in the number of variables or constraints~\cite{Nocedal2006Numerical}. With the ever-increasing size of the  LP problem to be solved, searching for more efficient solvers in the large-scaled setting has  attracted a lot of attention.   
 
Previous to our work, many papers have studied the applications of PPA to large-scale LP problems,  including the (proximal) ALM based solvers~\cite{Yen2015Sparse,G1992Augmented,ALMfoLP,Li2019An} and the ADMM based solvers~\cite{NIPS2017_ADMM,OldADMM,SCS}.  To facilitate the comparison, we shall omit the constants other than $r$, $\epsilon$ and $\kappa_r$ from the logarithmic terms appearing in batch complexity bounds. Namely, we simplify the batch complexity bound of AGPPA  given in~\eqref{eq:LPcomplexity} as follows:
\begin{equation}\label{eq:LPcomplexitys}
\cO \left (  \min \left\{ \max\limits_{i\in[n]} \norm{a_i}, \frac{ \norm{A}_F}{\sqrt{m}} \right\} \kappa_r \ln \kappa_r \ln  \left (\frac{ r\kappa_r }{\epsilon}\right)  \ln \left(r \kappa_r \ln \frac{ r \kappa_r}{\epsilon}\right)\right).
\end{equation}
 Plugging~\eqref{a:kapprTell} into~\eqref{eq:LPcomplexitys}, we get the batch complexity bound of AGPPA:
 \begin{align}\label{a:agppacm}
 \cO \left (  \min \left\{ \max\limits_{i\in[n]} \norm{a_i}, \frac{ \norm{A}_F}{\sqrt{m}} \right\} \theta r \ln (\theta r) \ln  \left (\frac{ \theta r }{\epsilon}\right)  \ln \left( \theta r \ln \frac{ \theta r }{\epsilon}\right)\right)
 \end{align}

Before giving  more details, we outline the main contribution of our work in comparison to the closely related works: 
 
 \begin{enumerate}
 \item We provide an iteration complexity bound~\eqref{eq:LPcomplexity} of PPA based LP solver. In  contrast,~\cite{Li2019An} shows the asymptotic superlinear convergence of PPA outer iterations.
 \item Our algorithm and theoretical analysis do not require any knowledge of the bounded metric subregularity parameter $\kappa_r$ satisfying~\eqref{growthcondition} for $T=T_{\ell}$.  In~\cite{Yen2015Sparse}, the theoretical results were derived by requiring  the proximal regularization parameter to be proportional to  the bounded metric subregularity parameter $\kappa$ of $T_{d}$ as defined in~\eqref{eq:Td}, which is generally unknown. 
 \item The inner problem stopping criterion~\eqref{stop:computablest} is implementable, which is not the case for the stopping criteria proposed in~\cite{NIPS2017_ADMM,Yen2015Sparse}. 
  \item Compared with the existing complexity bounds of the related methods (see Table~\ref{table:comparisoncomplexity}), the complexity bound of AGPPA has weaker dependence on the dimension of the problem and  on the Hoffman constant of the associated KKT system.
 \item Our algorithm is directly applicable to the LP problem in the general form of~\eqref{P}, while~\cite{Li2019An,NIPS2017_ADMM} were specifically designed for the standard LP problem as~\eqref{Sun:P} or its dual problem~\eqref{eq:D}. It is true that  the general LP problem can be transformed into the  standard form  or its dual problem. However, applying such a transformation at first may lead to worse performance in both a theoretical and a practical perspective.
 \end{enumerate}
 
\subsubsection{Transformation of the LP problem}\label{sec:tlp}
There are many methods specifically designed for LP problem  of the following form:
\begin{equation}
\min_{x \in \bR^n}~ c^{\top}x  ~~~\st ~~~Ax = b;\enspace  x_{i}\ge  0,  \forall i\in[n_b],
\label{Sun:P}
\end{equation}
or its dual problem
\begin{equation}
\min_{ \lambda \in \bR^m}~ b^{\top} \lambda~~~\st ~ -A_b^{\top}\lambda \le c_b ; \enspace - A_f^{\top}\lambda  =  c_f.
\label{eq:D}
\end{equation}
 Note that problem~\eqref{P} can be transformed into the form of~\eqref{Sun:P} as:
\begin{equation}\label{eq:trans1}
\min_{x\in \bR^n} c^\top x~\st~A_Ix + y = b_I, A_Ex = b_E, y \ge 0, x_{i}\ge  0, i\in[n_b].
\end{equation}
Denote
\[
\bar n := m_I + n,~\bar n_b := m_I +n_b,~\bar m_I := 0,~\bar m := m,
\]
and
\[
\bar c := [0;c],~\bar b_E := [b_I;b_E],~\bar x := [y;x],~ \bar A_E:=
\left(\begin{aligned}
&I_{m_I}&&A_I\\
&0&&A_E
\end{aligned}\right),
\]
then~\eqref{eq:trans1} is equivalent to 
\begin{equation}\label{eq:trans2}
\min_{ x\in \bR^{\bar n}} {\bar c}^\top  x~\st~\bar A_E x = \bar b_E; \enspace  x_{i}\ge  0, \forall i\in[\bar n_b].
\end{equation}
Same as the bound~\eqref{a:agppacm}, the batch complexity bound of AGPPA for problem~\eqref{eq:trans1}, or equivalently~\eqref{eq:trans2} is:
\begin{equation}\label{complexity:enlarge}
 \cO \left (  \min \left\{\max\left\{ \max\limits_{i\in[n]} \norm{a_i},  1\right\}, \frac{ \norm{A}_F + m_I}{\sqrt{m}} \right\} \bar\theta\bar r \ln (\bar \theta \bar r) \ln  \left (\frac{ \bar \theta \bar r }{\epsilon}\right)  \ln \left(\bar \theta \bar r \ln \frac{\bar \theta \bar r }{\epsilon}\right)\right),
 \end{equation}
 where $\bar r$ is the upper bound on the norm of all iteration points $(\bar x, \lambda)$, and $\bar \theta$ is the smallest constant satisfying
\begin{equation}\label{eq:enlargetheta}
\begin{aligned}
&\dist \bracket{\bracket{y,x,\lambda}, \bar \Omega}
\le \\
& \bar \theta
\norm{ \left [c^{\top}x + b^{\top}\lambda; [ A^{\top}\lambda +c]_{-}^{n_b};  [\lambda_{I}]_{+}^{m_I}; A_Ix + y -b_I; A_Ex -b_E\right ]},
\end{aligned}
\end{equation}
for all $(y, x, \lambda) \in \bR^{m_I + n + m}$ with $y \ge 0, x_{b} \ge 0$.
Here, $\bar \Omega$ is the set of saddle points of the Lagrangian function derived from problem~\eqref{eq:trans2}. Specifically, let 
\[
y = - [A_I x - b_I]_{-}^{m_I} \ge 0,~\lambda_{I} \ge 0,
\]
then~\eqref{eq:enlargetheta} reduces to 
\begin{equation}\label{eq:enlargetheta1}
\dist \bracket{\bracket{y,x,\lambda}, \bar \Omega} \le \bar \theta
\norm{ \left [c^{\top}x + b^{\top}\lambda; [ A^{\top}\lambda +c]_{-}^{n_b}; [Ax  -b]_{+}^{m_I}\right ]},
\end{equation}
for all $(x, \lambda) \in \bR^{ n + m}$ with $x_{b} \ge 0, \lambda_I \ge 0$.
Compare~\eqref{eq:enlargetheta1} with~\eqref{a:thetaSdef}, we learn that
\begin{equation}\label{eq:transformation result}
\theta \le \bar \theta.
\end{equation}
Thus, it is reasonable to claim that bound~\eqref{complexity:enlarge} performs worse than bound~\eqref{a:agppacm}. Similarly, the transformation from problem~\eqref{P} to the form of~\eqref{eq:D} will also lead to a worse complexity bound. 
 In addition, numerical experiments also suggest a loss of efficiency after applying such a transformation (see Section~\ref{section:numericalresults}).

 \subsubsection{Comparison with SNIPAL}
 
  When $n_b=0$ in~\eqref{P}, function $F$ defined in~\eqref{LP:fphi} becomes smooth with $\nabla F$ being semismooth, and thus the inner problem can be solved by the semismooth Newton (SSN) method. Based on this property, Li et al~\cite{Li2019An} proposed a semismooth Newton based inexact proximal  augmented Lagrangian  (SNIPAL) method  for  the LP problem.  
 In particular,  they focused on  how to  exploit the structure of the matrix $A$ and of the generalized Hessian of $F$ to efficiently solve each Newton system in the high dimensional setting $ (m\gg n)$. Asymptotic superlinear convergence for SNIPAL is obtained by requiring $\{\sigma_k\}_{k\geq 0}$ to tend to infinity along with the local superlinear convergence of  the SSN method.    
In this paper, we mainly focus on  giving explicit update formulas for $\{\sigma_k\}_{k\geq 0}$ and the overall complexity analysis. 
In addition, we emphasize that SNIPAL is only applicable to the dual form of the standard LP problem~\eqref{eq:D} while AGPPA is applicable to the general LP problem~\eqref{P}. 

\subsubsection{Comparison with linearized ADMM (LADMM)}\label{sec:LADMM}
LADMM~\cite{YuanADMM}  is a method of multipliers with inner problems exactly solved, which leads to a larger linear convergence rate, and thus a slower outer iteration convergence. To reach an $\epsilon$-KKT solution, i.e., to find a primal dual pair $ (x,\lambda)$ with $\cE_1 (x, \lambda) \le \epsilon$, LADMM has a batch complexity bound: 
\begin{equation}\label{eq:complexity:LADMM}
O \bracket{\norm{A}^2\kappa_{r'}^2 \ln  \frac{1}{ \epsilon}},
\end{equation}
which by~\eqref{a:kapprTell} yields:
\begin{align}\label{a:erewer}
O \bracket{\norm{A}^2\theta^2(r')^2 \ln  \frac{1}{ \epsilon}}.
\end{align}
 Here, $r'$ is an upper bound on the norm of all iteration points of LADMM, and we know that
  $$
 r'\leq \norm{(\bar x^0, \bar \lambda^0)} + \dist ((x^{0},\lambda^0) ,\Omega)\leq r .
 $$ 
 Comparing~\eqref{a:agppacm} with~\eqref{a:erewer}, we learn that the worst batch complexity bound of AGPPA scales better than that of LADMM for large-scale problems, and also has weaker dependence on the Hoffman constant $\theta$ of the KKT system~\eqref{eq:omegas}.

\subsubsection{Comparison with an  inexact ADMM (iADMM)}\label{sec:iADMM}
In~\cite{NIPS2017_ADMM}, Wang and Shroff proposed to apply  iADMM to solve the LP problem in the form of~\eqref{Sun:P}. They added  auxiliary variable $y$ and $n$ equalities $y = x$  to~\eqref{Sun:P}:
\begin{equation}
\min_{x \in \bR^n}~ c^{\top}x  ~~~\st ~~~Ax = b, y = x, y_{i}\ge  0, i\in[n_b],
\label{iADMM:P}
\end{equation}
and then applied the classical inexact ADMM to~\eqref{iADMM:P}.
 They proposed to approximately solve every inner problem  with an accelerated coordinate descent method~\cite{Allen2015Even}, until the function value at the current point  is close enough to the optimal value. 
 However, the proposed stopping criterion is not implementable as the optimal value of the inner problem is unknown.

It was shown in~\cite{NIPS2017_ADMM} that in order to have an $\epsilon$-dual optimal solution (i.e., a dual solution with distance to the dual optimal solution set bounded by $\epsilon$),  the batch complexity of iADMM is: 
\begin{equation}\label{eq:complexity:iADMM}
\cO \bracket{\max_{i \in [n]} \norm{a_i}  \bracket{r_x\norm{A} +r_z}^2\theta_{S^*}^2 \ln \frac{1}{\epsilon}\ln\frac{\theta_{S_*} }{\epsilon}}.
\end{equation}
Here, $r_x$ and $r_z$ are  upper bounds on the primal and dual iteration points respectively, and $\theta_{S_*}$ is the  Hoffman constant associated with the KKT system of problem~\eqref{iADMM:P}. In particular, with similar proof to that of~\eqref{eq:transformation result}, we have
\[
 \theta \le \theta_{S_*}.
\]

\subsubsection{Comparison with AL\_CD}\label{sec:ALCD}

The algorithm in~\cite{Yen2015Sparse}, named as AL\_CD, solves the LP problem~\eqref{P} by combining an  inexact augmented Lagrangian method with a randomized coordinate descent method~\cite{Yen2015Sparse} for the inner problems. This amounts to apply the proximal point method to the dual problem:
\[
\min_{ \lambda \in \bR^m}~\{b^{\top} \lambda + \delta_{\mathcal{F}_d} (\lambda)\}
\]
with
\[
\mathcal{F}_d =\{ \lambda \in \bR^m |   -A_b^{\top}\lambda \le c_b , - A_f^{\top} \lambda  =  c_f ,\lambda_{j}\ge  0, j\in[m_I]\}.
\]
 Denote by $\tilde r$ an upper bound on the norm of all dual iteration points  of AL\_CD, and $\kappa$ the bounded metric subregularity parameter such that:
\begin{equation}\label{growthcondition:d}
\dist \bracket{\lambda, T_d^{-1}(\zero)} \le \kappa \dist \bracket{0, T_d (\lambda)}, ~\forall \lambda~\st  \|\lambda \|\le \tilde r.
\end{equation}
Here, $T_d$  is the polyhedral multifunction defined by:
\begin{equation}\label{eq:Td}
T_d:  \lambda\rightarrow  b + \partial  \delta_{\mathcal{F}_d} (\lambda).
\end{equation}
The proximal regularization parameter of AL\_CD is required to be proportional to the bounded metric subregularity parameter $\kappa$, which is generally unknown.

 The inner problems of AL\_CD are non-strongly convex without 
 the proximal term for the primal variable. The stopping criterion of every inner problem in AL\_CD is also conceptual, which requires to know the optimal  solution of the inner problem. To obtain an $\epsilon$-dual optimal solution, the batch complexity bound of AL\_CD is:
\begin{equation}\label{eq:complexity:ALCD}
\cO \bracket{  \max\limits_{i\in[n]} \norm{a_i}^2  \theta_{\cS} \kappa \ln^2\bracket{ \frac{\kappa}{\epsilon}}},
\end{equation}
where 
 $\theta_{\cS}$ is the Hoffman's constant that depends on the polyhedron formed by the set of optimal solutions of the inner problems.

 We summarize the comparison of  batch complexity bounds in Table~\ref{table:comparisoncomplexity}.
  
 \begin{center}
\begin{table}
\centering
\begin{tabular}{c | c |c}
Algorithm & Measure &  Batch complexity bound \\
\hline
AGPPA & $\epsilon$-KKT  &$\cO \left (  \min \left\{ \max\limits_{i\in[n]} \norm{a_i}, \frac{ \norm{A}_F}{\sqrt{m}} \right\} \theta r \ln (\theta r) \ln  \left (\frac{ \theta r }{\epsilon}\right)  \ln \left( \theta r \ln \frac{ \theta r }{\epsilon}\right)\right)$\\
LADMM~\cite{YuanADMM} &  $\epsilon$-KKT & $O \bracket{\norm{A}^2 \theta^2 {r'}^2 \ln  \frac{1}{ \epsilon}}$ \\
iADMM~\cite{NIPS2017_ADMM} & $\epsilon$-dual &  $
\cO \bracket{\max\limits_{i\in[n]} \norm{a_i}   \bracket{r_x\norm{A} +r_z}^2\theta_{S^*}^2 \ln \frac{1}{\epsilon}\ln\frac{\theta_{S^*} }{\epsilon}}$\\
ALCD~\cite{Yen2015Sparse} & $\epsilon$-dual  & $
\cO \bracket{  \max\limits_{i\in[n]} \norm{a_i}^2  \theta_{\cS} \kappa \ln^2\bracket{ \frac{\kappa}{\epsilon}}}$\\
\noalign{\smallskip}\hline
\end{tabular}
\caption{Comparison of batch complexity bounds. More details on the constants can be found in Section~\ref{sec:LADMM}, Section~\ref{sec:iADMM} and Section~\ref{sec:ALCD}.}
\label{table:comparisoncomplexity} 
\end{table}
 \end{center}

\subsection{Hybrid inner solver}
Recall that we proposed in Algorithm~\ref{cQ_cAcG} to combine  an algorithm satisfying the HOOD property with any other algorithm with efficient numerical  performance while keeping the validity of the complexity bound~\eqref{a:floglo}.  We propose to combine APPROX or Katyusha  with the inexact projected semismooth Newton (PSSN) method, given in Section~\ref{appendix:PNCG} of the appendix. 

\section{Numerical results}\label{section:numericalresults}
When there is at least one coordinate with nonnegative constraint, we firstly transform the  problem into  the form of~\eqref{eq:D} as shown in Sectiion~\ref{sec:tlp}, then apply AGPPA to the transformed problem, 
 and we call the resulting algorithm AGPPAi. 
In this section, we compare the numerical performance of AGPPA\footnote{Our solver  AGPPA has been released in: \url{https://github.com/lumeng16/APPA}.}  with  AGPPAi, AL\_CD\footnote{The solver AL\_CD has been released in: \url{http://ianyen.site/LPsparse/}.}~\cite{Yen2015Sparse}, SCS\footnote{The solver SCS has been released in: \url{https://github.com/cvxgrp/scs}.}~\cite{SCS}, and  the commercial LP solver Gurobi\footnote{The solver Gurobi has been released in: \url{https://www.gurobi.com/  }.} version 9.1.1.    
Here, AL\_CD is an inexact ALM, SCS is an inexact ADMM, and  Gurobi includes the interior-point method (IPM) and the simplex methods. We employ the  error  residual  function  $\cE_2$ defined in~\eqref{eq:normalized KKT residual} for accuracy measure. %
   For AGPPA, AGPPAi,  AL\_CD, SCS and IPM, we stop the algorithm when either  the error residual $\cE_2$ is smaller than a threshold or the maximal running time is reached.  However, for the simplex method, we show the computation time when Gurobi terminates or the maximal running time is reached,   since it is not convenient to compute the error residual before the termination.

Although here  we compare the algorithms in single thread, AGPPA are applicable to multi-threads. We list some specific settings of the algorithms: 
\begin{enumerate}
\item 
We apply AL\_CD to both primal and dual LP problems and present the best result.
\item SCS has two versions, and the difference of them is the way to solve the linear systems generated by the inner problems. One applies an indirect solver based on the conjugate gradient method, and the other  uses a direct solver,  which applies a cached LDL factorization and may require larger memory. The direct solver works inefficiently for the test data sets, and thus we present the best result of SCS with  the indirect solver to both primal and dual LP problems. 
\item For Gurobi, we set method = 2 to use IPM and method = -1 to use simplex methods. Note that when method = -1 under the setting of single thread, it will automatically choose a method from the primal simplex method and the dual simplex method. For both IPM and simplex methods, if the presolve phase consumes too much time and does not reduce problem sizes, we turn off the presolve phase.
\item If $m_I > n_b$, we apply AGPPA to form~\eqref{P}, otherwise, we apply AGPPA to its dual form.  
\item Recall that the parameters of AGPPA are  $\rho$, $\sigma_0$, $\varrho_{\sigma}$, $\alpha$, $\gamma$, $\delta$, $\eta_0$, $\varsigma$, $\varrho_{\eta}$.  Let
\begin{equation}\label{choose:delta}
\delta =  \varrho_\delta*  \frac{ \rho - \sqrt{1-  \min \{\gamma, 2 \gamma -\gamma^2 } \}  }{1+\rho},
\end{equation}
with some $\varrho_\delta \in  (0,1)$. 
We set the initial guess for $\kappa_r$ as $1/\norm{A}_F$, then the initial proximal regularization parameter $\sigma_0$ is equal to
\begin{equation}\label{choose:sigma}
\sigma_0 = \frac{\alpha}{\norm{A}_F}.
\end{equation}
We set parameters  $\rho$, $\varrho_{\sigma}$, $\gamma$, $\eta_0$, $\varsigma$, $\varrho_{\eta}$ based on Table~\ref{table:input parameters}, and parameters $\delta$, $\alpha$ and $\sigma_0$ based on Table~\ref{table:deduced parameters}.
\end{enumerate}

\begin{table}
\centering
\begin{tabular}{lllllllll}
\hline\noalign{\smallskip}
parameter &$\cM$ & $\gamma$ & $\rho$ & $\varrho_\delta$& $\varrho_\eta$ & $\varrho_\sigma$ &$\varsigma$&$\eta_0$\\
\noalign{\smallskip}\hline\noalign{\smallskip}
value &$\cI$&1&0.7&0.9&0.9&5&1.1&1e16\\
\noalign{\smallskip}\hline
\end{tabular}
\caption{Default settings of input parameters}
\label{table:input parameters}
\end{table}

\begin{table}
\centering
\begin{tabular}{llll}
\hline\noalign{\smallskip}
 parameter  & $\delta$ & $\alpha$ & $\sigma_0$\\
\noalign{\smallskip}\hline\noalign{\smallskip}
equation &\eqref{choose:delta}&\eqref{eq:rho}&\eqref{choose:sigma}\\
\noalign{\smallskip}\hline
\end{tabular}
\caption{Default settings of deduced parameters}
\label{table:deduced parameters}
\end{table}

We numerically compare AGPPA with other algorithms in the mentioned settings above  for different types of LP problems, including LP problems generated from  L1-regularized multi-class support vector machine (L1-SVM)  problems, randomly generated sparse LP problems,    covering and packing  LP problems and some benchmark problems. We make comparison in three aspects: memory usage, running time to reach low precision ($\cE_2 (x,\lambda) \le$ 1e-3), and running time to reach medium precision ($\cE_2 (x,\lambda) \le$ 1e-5).   
For reference convenience, we also list the running time of the simplex method in the last column.  However, recall that this is  the running time when the simplex method terminates, and thus returns a high accuracy solution.

In all tables, o.m. means the test method for the test data set is ``out of memory (96GB)'', and *  means the running time reaches the maximal running time that is set for the test method (see details in the title of tables). Our code is written in C++ and all the experiments are conducted in the cluster consisted of 104 compute nodes outfitted with two 10-core Intel Xeon E5-2600 v3 (Haswell) processors.

\subsection{ L1-regularized multi-class SVM}
We consider LP instances  generated from the  L1-SVM problem, which is a classical machine learning problem. Let $x_i \in \mathbb{R}^{p_d}$  be the collected data, where  $p_d$ is the number of features. Let $y_i \in [k]$  be labels, where $k$ is the number of classes. Set $p_n$  as the number of samples. Then following the same notations in~\cite{Yen2015Sparse}, we show L1-SVM as:
\[\begin{aligned}
&\min&& \lambda \sum_{j=1}^k \norm{w_j}_{1} + \sum_{i=1}^{p_n} \xi_i\\
&\st && w_{y_i}^{\top} x_i - w_{j}^{\top}x_i \ge e_i^j - \xi_i && \forall i \in  [p_n],j \in[k],
\end{aligned}\]
where $e_i^j = 0$ if $y_i = j,e_i^j =1$ otherwise. Set $w_j = [w_j]_{+} + [w_j]_{-}$, then $\norm{w_j}_{1}  = {\bf 1}^{\top}  ([w_j]_{+}  - [w_j]_{-}) $ with $[w_j]_{+} \ge 0$ and $[w_j]_{-} \le 0.$ We can transform SVM into a LP problem~\eqref{P} with 
\[m_I =  (k-1)p_n,~m_E =0, ~  n_b = 2kp_d + p_n ,~ n_f = 0.\]
In the simulation, if the data is dense, we conventionally  normalize the data:
\[
x_{ij} = \frac{x_{ij}  - \frac{1}{p_d}\sum_{j=1}^{p_d} x_{ij}}{ \sqrt{ \sum_{j=1}^{p_d} \bracket{x_{ij}  - \frac{1}{p_d}\sum_{j=1}^{p_d} x_{ij}}^2}}, ~\forall j = 1, \dots, p_d,
\]
 otherwise, we scale the data:
 \[
x_{ij} = \frac{x_{ij} }{ \norm{x_i}}, ~\forall j = 1, \dots, p_d,
\]
where $x_{ij}$ are the $j$th coordinate of vector $x_i$.   We set the penalty parameter $\lambda$ to be $1$. All data sets in Table~\ref{L1SVM_data}  come from the LIBSVM\footnote{LBSVM website: \url{https://www.csie.ntu.edu.tw/~cjlin/libsvmtools/datasets/}} library.

In Table~\ref{L1SVM_data}, $m = m_I$, $n = n_b$ and sparsity means the percentage of non-zero numbers of matrix $A$ in the LP problem~\eqref{P} transformed from the L1-SVM problem. The last column of Table~\ref{L1SVM_data} refers to the size of the MPS file used to store the data. Note that the LP problem generated from the L1-SVM  problem has special structure. When we apply Gurobi to such  LP problems, during the presolve phase, it will reduce a number of rows or columns due to the dependence, which may substantially improve the performance of Gurobi. However, for other algorithms, including AGPPA, we use the default settings without presolve phase. Note that AGPPA still has a satisfactory advantage over Gurobi even without the acceleration from presolve phase. The numerical results are shown in Tables~\ref{ram:SVM}, Table~\ref{low precision:SVM}, and Table~\ref{medium precision:SVM}.

\begin{table}
\centering
\begin{tabular}{lllllllll}
\hline\noalign{\smallskip}
   Data  &$p_n$   & $p_d$ & $k$ &$m$&$n$ & sparsity & size\\
  \noalign{\smallskip}\hline\noalign{\smallskip} 
    real-sim& 72309&20958&2&72309&156141&0.1320\% & 625 MB \\
    rcv1 & 15564&47236&51& 778200&4833636&0.0055\% & 8.5 GB \\
    news20 & 15935&62061&20& 302765&2498375&0.0128\% & 4.0 GB\\
    avazu&14596137&999990&2&14596137 &18596097&0.00033\% & 39 GB\\
    \noalign{\smallskip}\hline
  \end{tabular}
  \caption{Data statistics for LP problems transformed from L1-SVM problems}
  \label{L1SVM_data}
\end{table}

\begin{table}
\centering
  \begin{tabular}{llllll|l}
\hline\noalign{\smallskip}
    Data &AGPPA &AGPPAi& AL\_CD &SCS& IPM & Simplex \\
  \noalign{\smallskip}\hline\noalign{\smallskip} 
    real-sim& 0.78& \textbf{0.46}&0.71&0.76&14.9&1.2\\
     rcv1 &10.5&10.5& \textbf{9.5}&11.1&o.m.&18.1\\
     news20&5.1&5.0& \textbf{4.5}&5.3&o.m.&8.5\\
      avazu&\textbf{47.0}&47.4&77.3&54.5&77.3&77.3\\
\noalign{\smallskip}\hline
  \end{tabular}
    \caption{RAM (GB). o.m. means ``out of memory (96GB)''.}
        \label{ram:SVM}
\end{table}

\begin{table}
\centering
  \begin{tabular}{llllll |l }
\hline\noalign{\smallskip}
    Data &AGPPA &AGPPAi& AL\_CD &SCS& IPM & Simplex  \\
  \noalign{\smallskip}\hline\noalign{\smallskip} 
    real-sim& \textbf{29}&248&72&1.57e3&16286 &4166\\
     rcv1 & \textbf{811}&*& 1002& * &o.m.& 44650 \\
     news20&\textbf{462}&* &595& * &o.m. &132082 \\
      avazu&\textbf{10293}&10343&19485& * & * &*\\
\noalign{\smallskip}\hline
  \end{tabular}
    \caption{Time for precision 1e-3 (s). * means the running time reaches 24 hours for AGPPAi and 56 hours for other methods. The fifth column uses the default output format in the SCS package.}
        \label{low precision:SVM}
\end{table}

\begin{table}
\centering
  \begin{tabular}{llllll|l}
\hline\noalign{\smallskip}
    Data &AGPPA &AGPPAi& AL\_CD &SCS& IPM & Simplex \\
  \noalign{\smallskip}\hline\noalign{\smallskip} 
    real-sim& \textbf{204}&3373&815&7.73e4&23178&4166\\
     rcv1 &\textbf{3636}&*&6787&*& o.m.& 44650\\
     news20&\textbf{12407}&*&13481&*&o.m.&132082\\
      avazu&\textbf{15581}&19301& * &*&*& *\\
\noalign{\smallskip}\hline
  \end{tabular}
    \caption{Time for precision 1e-5 (s). * means the running time reaches 24 hours for AGPPAi and 56 hours for other methods. The fifth column uses the default output format in the SCS package.}
        \label{medium precision:SVM}
\end{table}

\subsection{Randomly generated sparse LP}
We apply AGPPA to the LP problem of the form
\begin{equation}\label{generated LP}
\min_{x \in \bR^n}~c^{\top} x~~~\st~~~Ax \le b,
\end{equation}
where $c \in \bR^n, b \in \bR^m$, and $A \in \bR^{m\times n}$. Referring to~\cite{Li2019An,Mangasarian2004A}, we generate large synthetic matrix $A$ by 
\[
A = \operatorname{sprand} (m,n,\operatorname{sparsity}); A = 100* (A-0.5*\operatorname{spones} (A));
\]
where sparsity means the percentage of nonzero numbers of matrix $A$.

Here, the LP problem~\eqref{generated LP} can be transformed into the general LP form~\eqref{P} with $m_I = m, m_E = 0, n_b = 0$ and $n_f = n$. Then, all data sets have $m_I > n_b$, and therefore AGPPA is automatically  applied to the primal form. Since there is no nonnegative constraints in the considered LP problem, AGPPAi is equivalent to AGPPA. The data statistics are given in Table~\ref{table:data generated sparse LP}, and the numerical results are shown in Tables~\ref{ram:M}, Table~\ref{low precision:M}, and Table~\ref{medium precision:M}.
\begin{table}
	\centering
	\begin{subtable}[t]{2in}
		\centering
		  \begin{tabular}{lllllllll}
\hline\noalign{\smallskip}
    Data   &$m$&$n $&sparsity & size\\
    \noalign{\smallskip}\hline\noalign{\smallskip} 
    M1 & 5e5&1e5&1e-3 & 1.7 GB\\
    M2 & 1e6&1e5&1e-3 & 3.3 GB\\
    M3 &1e6&2e5&1e-3&6.5 GB\\
          M4 &1e7&1e5&1e-4 & 3.8 GB \\
     M5 &1e7& 2e5&1e-4&6.9 GB\\
     M6 &1e7& 5e5&1e-4 & 17 GB\\
    \noalign{\smallskip}\hline
  \end{tabular}
	\end{subtable}
	\qquad
	\begin{subtable}[t]{2in}
		\centering
		  \begin{tabular}{lllllllll}
\hline\noalign{\smallskip}
    Data   &$m$&$n $&sparsity & size\\
    \noalign{\smallskip}\hline\noalign{\smallskip} 
    M7 & 5e5&1e5&1e-4 & 192 MB\\
    M8 & 1e6&1e5&1e-4 & 378 MB\\
    M9 &1e6&2e5&1e-4 & 1.6 GB\\
          M10 &1e7&1e5&1e-5 & 742 MB \\
     M11 &1e7& 2e5&1e-5&1.8 GB\\
     M12 &1e7& 5e5&1e-5& 2.1 GB\\
    \noalign{\smallskip}\hline
  \end{tabular}
	\end{subtable}
	\caption{Data statistics for randomly generated sparse LPs}
	\label{table:data generated sparse LP}
\end{table}

\begin{table}
  \centering
  \begin{tabular}{lllll|l}
\hline\noalign{\smallskip}
     Data & AGPPA &AL\_CD& SCS &IPM & Simplex \\
    \noalign{\smallskip}\hline\noalign{\smallskip} 
    M1 & 3.1 &2.5 & \textbf{2.1}& 40.1 & 92.3 \\
    M2 & 5.9 &4.6 & \textbf{4.1}&42& 41\\
    M3 & 11.3 &9.5 & \textbf{7.5}&o.m.& 91.8 \\
     M4 & \textbf{6.1} & 7.3 &10.1& 41.8&17.5  \\
     M5 & \textbf{11.2} &12.7 &13.3&o.m.& 52.2 \\
     M6 &27.3& 30.2&\textbf{23.8}&o.m.&92.0 \\
               M7 &\textbf{0.32}&0.35 &0.57 & 24.2& 7.6 \\
     M8 & \textbf{0.61}&\textbf{0.61} &1.1 & 19.4 &8.0\\
     M9 &1.3&\textbf{1.1} &1.5&45.8 &8.4\\
     M10 &1.7& \textbf{1.3}&7.1 &5.6 & 6.7 \\
     M11 &2.1&\textbf{1.7} & 7.5&11.2 & 4.7  \\
      M12 &3.5 &\textbf{3.1} & 8.7&30.8 & 9.5 \\
    \noalign{\smallskip}\hline
  \end{tabular}
    \caption{RAM (GB). o.m. means ``out of memory (96GB)''.}\label{ram:M}
\end{table}

\begin{table}
\centering
  \begin{tabular}{lllll|l}
\hline\noalign{\smallskip}
     Data & AGPPA &AL\_CD & SCS &IPM & Simplex \\
    \noalign{\smallskip}\hline\noalign{\smallskip} 
    M1 & \textbf{1492} &12944 &1.72e3 & 73687 & *  \\
    M2 & \textbf{1769} &3251 & 1.50e4 &73517 & * \\
    M3 &\textbf{6382} &12475 &8.20e3 &o.m.  & * \\
     M4 & \textbf{1569} &9388 &2.51e4 &129034 & *  \\
     M5 & \textbf{3922} &24673 & * &o.m.  & *\\
     M6 &\textbf{14004} &42359 & 9.20e4 & o.m. &* \\
          M7 &\textbf{161}&* &6.85e2& 40844&*\\
     M8 &\textbf{224}& *&4.88e2 & 21015 &*\\
          M9 &\textbf{1251}& 1391& 1.03e3&62493 &*\\
     M10 &401&15915 &6.35e4&\textbf{303} &13972\\
          M11 &\textbf{489}&3532 &1.12e5&4016&*\\
           M12 &\textbf{1240}&26366 & 1.08e4&17433&*\\
    \noalign{\smallskip}\hline
  \end{tabular}
    \caption{Time for precision 1e-3 (s). * means the running time reaches 56 hours. The fourth column uses the default output format in the SCS package.}
      \label{low precision:M}
\end{table}

\begin{table}
\centering
  \begin{tabular}{lllll|l}
\hline\noalign{\smallskip}
     Data & AGPPA &AL\_CD & SCS &IPM & Simplex \\
    \noalign{\smallskip}\hline\noalign{\smallskip} 
    M1 & \textbf{8191}&14789 &3.31e4 & 92433& *\\
    M2 & \textbf{7317}&18048 & * &100646& *\\
    M3 &\textbf{33713}&67460 &1.41e5&o.m.& *\\
     M4 &\textbf{4676}&13087 & *&170938& *\\
     M5 &\textbf{10667}&38521 & * &o.m.& *\\
     M6 &\textbf{45294}&134146 & *& o.m.& *\\
          M7 &\textbf{1417}& *& 2.27e4 & 49174&*\\
     M8 &\textbf{1712}& *& 3.10e4& 25638 &*\\
     M9 & \textbf{5245}&7868 & 1.96e4& 85599&*\\
     M10 &6579& 30560 & * &\textbf{359} &13972\\
     M11 &5376& 143455&* &\textbf{4874}&*\\
     M12 &\textbf{11406}&* &* &21174&*\\
    \noalign{\smallskip}\hline
  \end{tabular}
    \caption{Time for precision 1e-5 (s). * means the running time reaches 56 hours. The fourth column uses the default output format in the SCS package.}
      \label{medium precision:M}
\end{table}

\subsection{Covering and packing LPs}
The covering LP problem is:
\begin{equation}\label{Covering}
\min~c^{\top}x~~~\st~~~Ax \ge e, x \ge 0,
\end{equation}
where $e \in \bR^m$ is the all one vector, $c \in \bR^n_{+}$ and $A \in \bR_+^{m\times n}$. Same as in~\cite{Li2019An},
we generate  large synthetic matrix $A$ by
\[ 
A = \operatorname{sprand} (m,n,\operatorname{sparsity}); A = \operatorname{round} (A);
\]
We test covering  LP problems~\eqref{Covering} with $m < n$ and note that the dual of the covering LP problem is the packing LP problem. 

Here, the covering LP problem can be transformed into the general LP form~\eqref{P}  with $m_I = m, m_E = 0, n_b = n, n_f = 0$. Then, all data sets have $m_I < n_b$, and therefore AGPPA is automatically  applied to the dual form. The data statistics are given in Table~\ref{Data statistics:C}, and the numerical results are shown in Tables~\ref{ram:C}, Table~\ref{low precision:C}, and Table~\ref{medium precision:C}.

\begin{table}
\centering
	\begin{subtable}[t]{2in}
	\centering
  \begin{tabular}{lllllllll}
    \hline\noalign{\smallskip}
   Data   &$m$&$n$&sparsity&  size\\
    \noalign{\smallskip}\hline\noalign{\smallskip} 
    C1 & 1e5&5e5&1e-3 &663 MB\\
    C2 & 1e5&1e6&1e-3 & 1.3 GB\\
    C3 &2e5&1e6&1e-3 & 3.9 GB\\
     C4 &1e5&1e7&1e-4 & 1.6 GB\\
     C5 &2e5&1e7&1e-4 & 2.9 GB\\
     C6 &5e5&1e7&1e-4 & 6.7 GB\\
    \noalign{\smallskip}\hline
  \end{tabular}
  	\end{subtable}
	\qquad
	\begin{subtable}[t]{2in}
	\centering
    \begin{tabular}{lllllllll}
    \hline\noalign{\smallskip}
    Data   &$m$&$n$&sparsity & size\\
    \noalign{\smallskip}\hline\noalign{\smallskip} 
    C7 & 1e5&5e5&1e-4 &84 MB \\
    C8 & 1e5&1e6&1e-4 & 164 MB\\
    C9 &2e5&1e6&1e-4 & 297 MB\\
     C10 &1e5&1e7&1e-5 &447 MB\\
     C11 &2e5&1e7&1e-5 &579 MB\\
     C12 &5e5&1e7&1e-5 &977 MB\\
       \noalign{\smallskip}\hline
       \end{tabular}
         \end{subtable}
    \caption{Data statistics for covering LPs}
      \label{Data statistics:C}
\end{table}

\begin{table}
\centering
  \begin{tabular}{llllll|l}
    \hline\noalign{\smallskip}
    Data   & AGPPA & AGPPAi &AL\_CD &SCS& IPM & Simplex \\
    \noalign{\smallskip}\hline\noalign{\smallskip} 
    C1 &1.4&1.5&\textbf{1.1}&1.3&37.6&2.1\\
    C2 &2.6&2.6&\textbf{2.2}&2.5&38.3&4.2\\
    C3 &5.6&5.6&\textbf{4.1}&4.3&o.m.&5.5\\
     C4 &3.2&3.5&\textbf{3.1}&8.5&20.8&7.1 \\
     C5 &6.2&6.1&\textbf{5.2}&10.1&o.m.&15.5 \\
     C6 &13.4&13.6&\textbf{11.5}&15.4&o.m.&29.6\\
          C7 &0.22&0.23&\textbf{0.17}&0.56&17.3&0.36\\
     C8 &0.41&0.40&\textbf{0.32}&0.96&17.5&0.72\\
      C9 &0.68&0.67&\textbf{0.54}&1.3&92.3&0.9\\
     C10 &1.5&1.5&\textbf{1.1}&7.0&5.1&3.6\\
     C11 &1.8&1.8&\textbf{1.3}&7.3&5.5&4.1\\
      C12 &3.1&2.5&\textbf{2.0}&8.2&60.3&5.3\\
    \noalign{\smallskip}\hline
  \end{tabular}
    \caption{RAM (GB). o.m. means ``out of memory (96GB)''.}
     \label{ram:C}
\end{table}

\begin{table}
\centering
  \begin{tabular}{llllll|l}
    \hline\noalign{\smallskip}
    Data   & AGPPA & AGPPAi &AL\_CD &SCS & IPM & Simplex  \\
    \noalign{\smallskip}\hline\noalign{\smallskip} 
    C1 &\textbf{161}&20506&440&1.40e5&76578 &31996\\
    C2 &\textbf{428}&49188&764& *&71718 &20936\\
    C3 &\textbf{569}& * &2099&*&o.m. & *\\
     C4 &\textbf{1592}&6524&3988& 8.35e4&22608 &4759 \\
     C5 &\textbf{2114}&13985&2426&*&o.m. &* \\
     C6 &\textbf{4003}&*&7836&*&o.m. & *\\
           C7 &\textbf{35}&433&167&6.35e2&22903&9034\\
      C8 &\textbf{97}&831&134&4.60e3&24131&5279\\
       C9 &\textbf{126}&2326&336&9.31e3&*&*\\
       C10 &177&228&207&5.26e3&\textbf{8}&10\\
        C11 &301&543&378&6.72e3&16&\textbf{10}\\
        C12 &921&2740&1114&2.08e4&145083&\textbf{56}\\
    \noalign{\smallskip}\hline
  \end{tabular}
    \caption{Time for precision 1e-3 (s). * means the running time reaches 24 hours for AGPPAi and 56 hours for other methods. The fifth column uses the default output format in the SCS package.}
      \label{low precision:C}
\end{table}

\begin{table}
\centering
  \begin{tabular}{llllll|l}
    \hline\noalign{\smallskip}
    Data   & AGPPA & AGPPAi&AL\_CD &SCS& IPM& Simplex \\
    \noalign{\smallskip}\hline\noalign{\smallskip} 
    C1 & \textbf{398}&53519&653& *&95332 & 31996\\
    C2 &\textbf{819}&*&2659&*&89494&20936\\
    C3 &\textbf{857}&*&3035&*&o.m.& *\\
     C4 &\textbf{7029}&38835&10696& *&32579 &4759\\
     C5 &\textbf{7261}&*&7522&*&o.m.& *\\
     C6 &\textbf{8459}&*&17224&*&o.m.& *\\
          C7 &355&7200&\textbf{327}&1.23e4&31833&9034\\
     C8 &406&8245&\textbf{378}&1.01e5&30399&5279\\
     C9 &\textbf{593}&25271&888&1.38e5&*& *\\
     C10 &247&346&914&*&\textbf{8}&10\\
     C11 &464&885&1648&*&16&\textbf{10}\\
     C12 &4425&16451&8835&*&181287&\textbf{56}\\
    \noalign{\smallskip}\hline
  \end{tabular}
    \caption{Time for precision 1e-5 (s). * means the running time reaches 24 hours for AGPPAi and 56 hours for other methods. The fifth column uses the default output format in the SCS package.}
      \label{medium precision:C}
\end{table}

\subsection{Benchmark data sets}\label{section:benchmark}

In this subsection, we consider benchmark data sets~\footnote{see benchmark data sets from:\\
$~~~~~~$https://www.netlib.org/lp/data/index.htm\\
$~~~~~~$http://www.gamsworld.org/performance/plib/credits.htm} for LP problems with the following form:
\begin{equation}\label{LP:benchmark}
\min~c^{\top}x~~~\st~~~A_{E}x = b_E, A_I x \le b_I , x \ge 0.
\end{equation}
Here,  all data sets have $m_I < n$, and therefore AGPPA is automatically  applied to the dual form. The data statistics are given in Table~\ref{Data statistics:benchmark}, and the numerical results are shown in Tables~\ref{ram:benchmark}, Table~\ref{low precision:benchmark}, and Table~\ref{medium precision:benchmark}.

\begin{table}
	\centering
  \begin{tabular}{llllllllc}
    \hline\noalign{\smallskip}
    Data   &$m_I$& $m_E$&$n$&sparsity & size\\
    \noalign{\smallskip}\hline\noalign{\smallskip} 
    OSA\_14&2337&0&52460& 0.257\% & 12 MB\\
    MAROS\_R7 &0&3136&9408&0.491\%& 4.6 MB\\
    OSA\_30 &4350&0&100024&0.138\% & 22 MB \\
     CRE\_B&4690&4958&72447& 0.037\% & 10 MB \\
      OSA\_60 &10280&0&232960& 0.058\% & 50 MB\\
      NUG\_12 &0&3192&8856& 0.136\% & 1.4 MB\\
      NUG\_15 &0&6330&22275&0.067\% & 3.5 MB \\
      NUG\_20 &0&15240&72600& 0.028\% & 12 MB  \\
      NUG\_30 &0&52260& 379350& 0.008\% & 56 MB\\
    \noalign{\smallskip}\hline
  \end{tabular}
	    \caption{Data statistics for benchmark data sets}
      \label{Data statistics:benchmark}
\end{table}

\begin{table}
\centering
  \begin{tabular}{llllll|l}
    \hline\noalign{\smallskip}
    Data   & AGPPA & AGPPAi &AL\_CD &SCS& IPM & Simplex \\
    \noalign{\smallskip}\hline\noalign{\smallskip} 
    OSA\_14&0.027&0.033&\textbf{0.020}&0.050&0.044&0.043 \\
    MAROS\_R7 &0.014&0.013&\textbf{0.011}&0.015&0.023&0.022 \\
    OSA\_30 &0.049&0.063&\textbf{0.036}&0.095&0.079&0.079 \\
     CRE\_B&0.031&0.029&\textbf{0.018}&0.068&0.056& 0.056\\
      OSA\_60 &0.104&0.124&\textbf{0.082}&0.218&0.179&0.178 \\
      NUG\_12 &\textbf{0.005}&\textbf{0.005}&\textbf{0.005}&0.011&0.014&0.018 \\
      NUG\_15 &0.011&0.010&\textbf{0.009}&0.024&0.027&0.037 \\
      NUG\_20 &0.035&0.035&\textbf{0.025}&0.071&0.073&0.179\\
      NUG\_30 &0.145&0.139&\textbf{0.100}&0.343&0.492&0.372 \\
    \noalign{\smallskip}\hline
  \end{tabular}
        \caption{RAM (GB).}
   \label{ram:benchmark}
\end{table}

\begin{table}
\centering
  \begin{tabular}{llllll|l}
    \hline\noalign{\smallskip}
    Data   & AGPPA & AGPPAi &AL\_CD &SCS & IPM & Simplex  \\
    \noalign{\smallskip}\hline\noalign{\smallskip} 
    OSA\_14&0.5&16772&21&1.92e4 &\textbf{0}&\textbf{0} \\
    MAROS\_R7 &58&54&1131&6.93e0&\textbf{0}&1 \\
    OSA\_30 &0.63&60261&32&*&\textbf{0}& \textbf{0}\\
     CRE\_B&*&*&42344&1.31e2&\textbf{1}& \textbf{1}\\
      OSA\_60 &2.1&3986&66&*&\textbf{1}&\textbf{1} \\
      NUG\_12 &435&580&25&2.01e0&\textbf{0}& 1\\
      NUG\_15 &9088&13086&85&6.28e0&\textbf{0}& 11\\
      NUG\_20 &46403&56589&380&4.18e1&\textbf{3}&303 \\
      NUG\_30 &*&*&2675&1.52e2&\textbf{58}& * \\
    \noalign{\smallskip}\hline
  \end{tabular}
    \caption{Time for precision 1e-3 (s). 0 means the running time is less than 0.5s. * means the running time reaches 24 hours. The fifth column uses the default output format in the SCS package.}
      \label{low precision:benchmark}
\end{table}

\begin{table}
\centering
  \begin{tabular}{llllll|l}
    \hline\noalign{\smallskip}
    Data   & AGPPA & AGPPAi&AL\_CD &SCS& IPM& Simplex \\
    \noalign{\smallskip}\hline\noalign{\smallskip} 
    OSA\_14&2928&16773&850&*&\textbf{0}&\textbf{0} \\
    MAROS\_R7 &83&85&1276&*&\textbf{0}&1 \\
    OSA\_30 &5937&60448&2698 &*&\textbf{0}& \textbf{0}\\
     CRE\_B&*&*&74092&3.80e4 &\textbf{1}& \textbf{1}\\
      OSA\_60 &*&*&7787&*&\textbf{1}&\textbf{1} \\
      NUG\_12 &2776&6165&71&1.73e3&\textbf{0}& 1 \\
      NUG\_15 &34328&47687&200&1.14e2&\textbf{1}& 11 \\
      NUG\_20 &*&*&1767&1.73e3&\textbf{3}&303 \\
      NUG\_30 &*&*&7457&3.12e3&\textbf{81}&* \\
    \noalign{\smallskip}\hline
  \end{tabular}
    \caption{Time for precision 1e-5 (s). 0 means the running time is less than 0.5s. * means the running time reaches 24 hours.  The fifth column uses the default output format in the SCS package.}
      \label{medium precision:benchmark}
\end{table}

\subsection{Conclusions  about the numerical results}
We draw some conclusions on the numerical results, including memory usage  based on 
Table~\ref{ram:SVM}, Table~\ref{ram:M}, Table~\ref{ram:C} and Table~\ref{ram:benchmark}, time efficiency up to low precision 1e-3 based on Table~\ref{low precision:SVM}, Table~\ref{low precision:M}, Table~\ref{low precision:C} and Table~\ref{low precision:benchmark},  and time efficiency up to medium precision 1e-5 based on Table~\ref{medium precision:SVM}, Table~\ref{medium precision:M}, Table~\ref{medium precision:C} and Table~\ref{medium precision:benchmark}.

It is easy to notice that AGPPA performs very well compared with the other solvers on the datasets in Table~\ref{L1SVM_data}, M1-M6 in Table~\ref{table:data generated sparse LP} and C1-C6 in Table~\ref{Data statistics:C}.  Meanwhile, AGPPA appears to be much slower than Gurobi on the benchmark datasets in Table~\ref{Data statistics:benchmark}.  We observe that the MPS files of the benchmark datasets in Table~\ref{Data statistics:benchmark} are of size less than 56MB and the memory usage of IPM on these datasets are all less than 0.5GB. This suggests that the matrix factorization step of IPM can be done in a very efficient way for these datasets, probably due to the relative small problem scale and the sparsity of the data.
However, for large-scale problems, by which we mean problems for which the MPS file size exceeds 2GB, we have the following observations:

\begin{enumerate}
\item  In terms of memory usage, AGPPA, AL\_CD and SCS are comparable, while IPM and simplex methods typically require high memory, and fail to produce the required solutions when they run out of memory.
\item \label{conclusion:transformation} 
AGPPAi shows significantly worse performance than AGPPA, which proves the necessity of directly dealing with the LP problem in the general form of~\eqref{P}. 
\item AGPPA v.s. AL\_CD: for both low and medium precision, we observe that AGPPA performs  better than AL\_CD. For most cases, AL\_CD
is approximately two to four times slower than  AGPPA. For the largest data set avazu in  Table~\ref{medium precision:SVM}, AGPPA shows a prominent advantage over AL\_CD since the update way of AL\_CD fails to efficiently update the proximal regularization parameters. 

\item AGPPA v.s. SCS: for many instances, SCS fails to give a solution within 56 hours.  Otherwise, it could be at least four times slower than AGPPA to reach the medium precision.
\item AGPPA v.s. IPM: IPM  encounters the memory problem, and fails on some instances. Otherwise,   IPM  could be at least ten times slower than AGPPA.

\item AGPPA v.s. simplex: the simplex method may fail to return a solution (of high-accuracy) within 56 hours. For those instances that the simplex method can solve, the running time of the simplex method (high accuracy) is approximately three to ten times higher than the time AGPPA needs (medium accuracy). 
\end{enumerate}



We conclude that AGPPA  provides an alternative for large-scale problems for which  IPM or simplex methods tend to be slow
or even fail to return solution due to memory shortage.
The experimental results on real and synthetic  large-scale LP problems confirm that AGPPA can provide an approximate solution of medium accuracy
in much less time than Gurobi 
when the matrix factorization step is expensive and the memory usage is tens of GB.   AGPPA also shows consistently superior performance  on  large instances than other PPA based solvers.

\section{Conclusion}\label{section:conclusion}
In this paper, we propose a new self-adaptive technique to update the proximal regularization parameters in the proximal point method, for a maximal monotone operator satisfying the bounded metric subregularity condition.   The proposed adaptive proximal point algorithm (AGPPA) is proved to have a  linear convergence rate without requiring any knowledge on the bounded metric subregularity parameter.   We apply AGPPA on a class of convex programming problems and analyze the iteration complexity bound if a linearly convergent inner solver can be applied to the subproblems. Our approach allows us to have a hybrid inner solver and thus can benefit from local fast convergence of second-order methods while keeping the same complexity bound. We illustrate the application to the LP problem and obtain a complexity bound with weaker dependence on the problem dimension and on the Hoffman constant associated with the KKT system. Finally we demonstrate the numerical efficiency of our method on various large-scale LP problems.   

\textbf{Acknowledgement} The computations were performed using research computing facilities offered by Information Technology Services, the University of Hong Kong.

\begin{appendices} 
\section{Supplementary proofs}
Denote
\[ 
Q_{\sigma\cM^{-1}T} (z) := z - \cJ_{\sigma\cM^{-1}T} (z).
\]
We present the properties of $\cJ_{\sigma \cM^{-1} T}$ and $Q_{\sigma\cM^{-1}T} $ in the following proposition, which is summarized in~\cite[Proposition 1]{Li2019An}.
\begin{proposition}[\cite{Li2019An}] \label{prop:useful}
 It holds for all positive real numbers $\sigma$  and all  self-adjoint positive definite linear operators $\cM$   that, for all $z, z'\in \cX$:
\begin{enumerate}[label=(\alph*)]\label{prop:uf} 
\item \label{first}$z = \cJ_{\sigma\cM^{-1} T}(z) + Q_{\sigma\cM^{-1} T}(z)$ and $\sigma^{-1}\cM Q_{\sigma\cM^{-1} T}(z) \in T(\cJ_{\sigma\cM^{-1} T}(z))$.
\item \label{second}$\langle \cJ_{\sigma\cM^{-1} T}(z) - \cJ_{\sigma\cM^{-1} T}(z^{\prime}) , Q_{\sigma\cM^{-1} T}(z) - Q_{\sigma\cM^{-1} T}(z^{\prime})\rangle_{\cM} \ge 0$.
\item \label{third}
$\norm{ \cJ_{\sigma\cM^{-1} T}(z) - \cJ_{\sigma\cM^{-1} T}(z^{\prime}) }_{\cM}^2 + \norm{ Q_{\sigma\cM^{-1} T}(z) - Q_{\sigma\cM^{-1} T}(z^{\prime}) }_{\cM}^2
$\\
$ \le \norm{ z- z^{\prime} }_{\cM}^2$.
\end{enumerate}
\end{proposition}
\begin{remark} If $\cM$ is an identity linear operator, i.e., $\cM = I$, then Proposition~\ref{prop:useful} reduces to~\cite[Proposition 1]{PPA}.
\end{remark}

\begin{proof}[\bf Proof of Theorem~\ref{thm:linearrate}]Consider any $z^* \in \Omega$ and define 
\begin{equation}\label{z:star}
z^{k+1}_* = \gamma  \cJ_{\sigma_k\cM^{-1} T}(z^k) + (1-\gamma) z^k,
\end{equation}
then we have
\[\begin{aligned}
&&\norm{z^{k+1} -z^*}_{\cM} \le& \norm{z_*^{k+1} - z^*}_{\cM}  + \gamma \norm{w^k -  \cJ_{\sigma_k\cM^{-1} T}(z^k)  }_{\cM}\\
&& \overset{\eqref{GPPA:stop}}{\le}& \norm{z^{k+1}_{*} - z^*}_{\cM}  + \delta_k \gamma \norm{w^k -  z^k  }_{\cM}\\
&&\overset{\eqref{GPPA:next iterate}}{=}& \norm{z^{k+1}_{*}- z^*}_{\cM}  +  \delta_k \norm{z^{k+1} -  z^k}_{\cM} \\
&&\le&\norm{z^{k+1}_{*}- z^*}_{\cM}  +  \delta_k \norm{z^{k+1} -z^*}_{\cM}  +  \delta_k \norm{z^k - z^*}_{\cM},  \\
\end{aligned}\]
which implies
\begin{equation}\label{eq:ratecombine}
\norm{z^{k+1} -z^*}_{\cM}  \le \frac{1}{1-\delta_k} \bracket{\norm{z^{k+1}_{*}- z^*}_{\cM}  +  \delta_k \norm{z^k - z^*}_{\cM}   }.
\end{equation}
Denote
\begin{equation}\label{def:mu}
\mu_k := \frac{\kappa_r }{\sqrt{\sigma_k^2 + \kappa_r^2 }}, 
\end{equation}
and
\[
\Pi_{\Omega}^{\cM}\bracket{\cJ_{\sigma_k\cM^{-1} T}\bracket{z^k}}: = \arg \min_{z \in \Omega} \norm{z - \cJ_{\sigma_k\cM^{-1} T}\bracket{z^k}}_{\cM},\]
\[\Pi_{\Omega}^{\cM}\bracket{z^k}: = \arg \min_{z \in \Omega} \norm{z - z^k}_{\cM},
\]
 then by Proposition~\ref{prop:useful}~\ref{third} and~\eqref{eq:key}, we have
\begin{equation}\label{eq:gamma0-2}
\begin{aligned}
&\norm{z^k- \Pi_{\Omega}^{\cM}\bracket{\cJ_{\sigma_k\cM^{-1} T}\bracket{z^k}}}_{\cM} \\ &\le\norm{\cJ_{\sigma_k\cM^{-1} T}\bracket{z^k}- \Pi_{\Omega}^{\cM}\bracket{\cJ_{\sigma_k\cM^{-1} T}\bracket{z^k}}}_{\cM}  +  \norm{Q_{\sigma_k\cM^{-1} T}\bracket{z^k}}_{\cM}  \\
&\le\dist_{\cM} \bracket{\cJ_{\sigma_k\cM^{-1} T}(z^k), \Omega} + \dist_{\cM} \bracket{z^k,  \Omega}\\
&\le ( \mu_k + 1) \dist_{\cM} \bracket{z^k,  \Omega}.
\end{aligned}
\end{equation}
\textbf{Case I}: $\gamma \in [1,2),$ and let 
\[
z^* =  \Pi_{\Omega}^{\cM}\bracket{\cJ_{\sigma_k\cM^{-1} T} \bracket{z^k}}.
\]
Let $a =  \cJ_{\sigma_k\cM^{-1} T}\bracket{z^k}, b = z^k, c =z^*$ for equality
\[
\langle a-c ,b-c \rangle_{\cM}  = \frac{1}{2} \bracket{\norm{a-c}_{\cM}^2  - \norm{a-b}_{\cM}^2 + \norm{b-c}^2_{\cM}}, 
\]
 then we have
\begin{equation}\label{eq0:gamma1-2}
\begin{aligned}
& \norm{z^{k+1}_* - z^*}^2_{\cM}=\norm{ \gamma \bracket{ \cJ_{\sigma_k\cM^{-1} T}(z^k)  - z^*} + \bracket{1-\gamma} \bracket{z^k-z^*}}_{\cM}^2\\
&=  \gamma^2 \dist_{\cM}^2\bracket{\cJ_{\sigma_k\cM^{-1} T}(z^k), \Omega} + (1-\gamma)^2\norm{ z^k- z^*}_{\cM}^2 \\
&+ 2 \gamma(1-\gamma) \left\langle \cJ_{\sigma_k\cM^{-1} T}(z^k)  - z^* , z^k -z^*\right\rangle_{\cM}\\
&=  \gamma  \dist_{\cM}^2\bracket{\cJ_{\sigma_k\cM^{-1} T}(z^k), \Omega}+ (1-\gamma)\norm{ z^k-z^*}_{\cM}^2 \\
&+\bracket{\gamma^2 - \gamma} \norm{z^k -  \cJ_{\sigma_k\cM^{-1} T}(z^k)}^2_{\cM}.
\end{aligned}
\end{equation}
By Proposition~\ref{prop:useful}~\ref{third}, we have
\begin{equation}\label{eq1:gamma1-2}
\begin{aligned}
&\norm{  z^k -  \cJ_{\sigma_k\cM^{-1} T}(z^k) }_{\cM}^2 \\
& \le \dist^2_{\cM}\bracket{ z^k , \Omega} - \norm{ \cJ_{\sigma_k\cM^{-1} T}(z^k) - \Pi_{\Omega}^{\cM} \bracket{z^k}  }_{\cM}^2 .
\end{aligned}\end{equation}
Plugging~\eqref{eq1:gamma1-2} into~\eqref{eq0:gamma1-2} leads to 
\[
\begin{aligned}
& \norm{z^{k+1}_* - z^*}^2_{\cM} \le  \gamma \dist_{\cM}^2\bracket{\cJ_{\sigma_k\cM^{-1} T}(z^k), \Omega} + (1-\gamma)\norm{ z^k-z^*}_{\cM}^2 \\
&+\bracket{\gamma^2 - \gamma} \bracket{\dist^2_{\cM}\bracket{ z^k , \Omega}- \norm{ \cJ_{\sigma_k\cM^{-1} T}(z^k) - \Pi_{\Omega}^{\cM} \bracket{z^k}  }_{\cM}^2 },
\end{aligned}
\]
together with 
\[
\gamma \ge 1,~\norm{ z^k-z^*} \ge \dist_{\cM} \bracket{z^k, \Omega},\]
and 
\[\norm{ \cJ_{\sigma_k\cM^{-1} T}(z^k) - \Pi_{\Omega}^{\cM} \bracket{z^k}  } \ge  \dist_{\cM} \bracket{\cJ_{\sigma_k\cM^{-1} T}(z^k), \Omega},
\]
we have
\begin{equation}\label{eq:exactrategamma1-2}
\begin{aligned} 
&\norm{z^{k+1}_* - z^*}_{\cM} \\
&\le\sqrt{ \bracket{\gamma-1}^2 \dist^2_{\cM}\bracket{ z^k , \Omega} + \bracket{2-\gamma} \gamma \dist_{\cM}^2 \bracket{ \cJ_{\sigma_k\cM^{-1} T}(z^k), \Omega}}\\
 &\overset{\eqref{eq:key}}{\le}\sqrt{1- \frac{\gamma(2-\gamma)\sigma_k^2}{ \sigma_k^2 + \kappa_r^2  }} \dist_{\cM} \bracket{z^k,  \Omega}.
 \end{aligned}
\end{equation}
Therefore, combine~\eqref{eq:ratecombine},~\eqref{eq:gamma0-2} and~\eqref{eq:exactrategamma1-2}, and we have
\begin{equation}\label{eq:rate1}
\dist_{\cM_{k+1}}\bracket{z^{k+1}, \Omega}\le \frac{\sqrt{1-\frac{\gamma(2-\gamma) \sigma_k^2}{ \sigma_k^2 + \kappa_r^2 }}+  \delta_k ( \mu_k + 1) }{1-\delta_k}\dist_{\cM} \bracket{z^k,  \Omega}.  
\end{equation}
\textbf{Case II}: $\gamma \in (0,1]$, and let
\[
z^* =  \gamma\Pi_{\Omega}^{\cM}\bracket{\cJ_{\sigma_k\cM^{-1} T}\bracket{z^k}}  + \bracket{1-\gamma}\Pi_{\Omega}^{\cM}\bracket{z^k},
\]
then $z^* \in \Omega$. With Lemma~\ref{lemma:key} and $\gamma \in (0,1]$, we have
\[
\begin{aligned}
& \norm{z^{k+1}_* - z^*}^2_{\cM}\\
&=\norm{ \gamma \bracket{ \cJ_{\sigma_k\cM^{-1} T}(z^k)  - \Pi_{\Omega}^{\cM}\bracket{\cJ_{\sigma_k\cM^{-1} T}\bracket{z^k}}} + \bracket{1-\gamma} \bracket{z^k- \Pi_{\Omega}^{\cM}\bracket{z^k} }}_{\cM}^2\\
&=  \gamma^2 \norm{ \cJ_{\sigma_k\cM^{-1} T}(z^k) - \Pi_{\Omega}^{\cM}\bracket{\cJ_{\sigma_k\cM^{-1} T}\bracket{z^k}}}^2_{\cM} + (1-\gamma)^2\norm{ z^k-  \Pi_{\Omega}^{\cM}\bracket{z^k}}_{\cM}^2 \\
&+ 2 \gamma(1-\gamma) \left\langle \cJ_{\sigma_k\cM^{-1} T}(z^k)  - \Pi_{\Omega}^{\cM}\bracket{\cJ_{\sigma_k\cM^{-1} T}\bracket{z^k}} , z^k - \Pi_{\Omega}^{\cM}\bracket{z^k} \right\rangle_{\cM}\\
&= (1-\gamma)\norm{ z^k-  \Pi_{\Omega}^{\cM}\bracket{z^k}}_{\cM}^2  + \gamma \norm{\cJ_{\sigma_k\cM^{-1} T}(z^k)- \Pi_{\Omega}^{\cM}\bracket{\cJ_{\sigma_k\cM^{-1} T}\bracket{z^k}}}_{\cM}^2   \\
&- \gamma(1-\gamma) \norm{ \cJ_{\sigma_k\cM^{-1} T}(z^k)  - \Pi_{\Omega}^{\cM}\bracket{\cJ_{\sigma_k\cM^{-1} T}\bracket{z^k}} - \bracket{z^k - \Pi_{\Omega}^{\cM}\bracket{z^k}}}_{\cM}^2\\
&\le\bracket{1-\gamma}\dist_{\cM}^2 \bracket{z^k,  \Omega} + \gamma \dist_{\cM}^2 \bracket{\cJ_{\sigma_k\cM^{-1} T}(z^k), \Omega} \\
&\label{eq:exactrategamma0-1} \overset{\eqref{eq:key}}{\le}\bracket{1- \frac{\gamma\sigma_k^2}{ \sigma_k^2 + \kappa_r^2 }} \dist_{\cM}^2 \bracket{z^k,  \Omega},
\end{aligned}
\]
which implies
\begin{equation}\label{eq:exactrategamma0-1}
\norm{z^{k+1}_* - z^*}_{\cM} \le \sqrt{1-\frac{ \gamma\sigma_k^2}{ \sigma_k^2 + \kappa_r^2 }} \dist_{\cM} \bracket{z^k,  \Omega}.
\end{equation}
In addition, with~\eqref{eq:gamma0-2} , we have
\begin{equation}\label{eq:gamma0-1}
\begin{aligned}
&\norm{z^k - z^*}_{\cM}  &\le& \gamma\norm{z^k- \Pi_{\Omega}^{\cM}\bracket{\cJ_{\sigma_k\cM^{-1} T}\bracket{z^k}}}_{\cM}   + (1-\gamma) \dist_{\cM} \bracket{z^k,  \Omega}\\
&&\le& (\gamma \mu_k + 1) \dist_{\cM} \bracket{z^k,  \Omega},
\end{aligned}
\end{equation}
Combining~\eqref{eq:ratecombine},~\eqref{eq:exactrategamma0-1} and~\eqref{eq:gamma0-1}, we have
\begin{equation}\label{eq:rate2}
\dist_{\cM_{k+1}}\bracket{z^{k+1},\Omega}  \le \frac{ \sqrt{1-\frac{ \gamma\sigma_k^2}{ \sigma_k^2 + \kappa_r^2 }} + \delta_k \bracket{  \gamma \mu_k +1} }{1-\delta_k}\dist_{\cM} \bracket{z^k,  \Omega}.
\end{equation}
The linear rate is proved by~\eqref{eq:rate1} and~\eqref{eq:rate2} with~\eqref{def:mu}.
\end{proof}

\begin{proof}[\textbf{Proof of Proposition~\ref{prop:leftright}}]
By Proposition~\ref{prop:useful}~\ref{third}, we have
\begin{equation}\label{finite:eq3}
\norm{Q_{\sigma_k\cM^{-1}T} (z^k) }_{\cM} \le\dist_{\cM} \bracket{z^{k},  \Omega}.
\end{equation}
Then
\[
\begin{aligned}
&\norm{z^{k+1} - z^k}_{\cM}& = & \gamma \norm{w^k -  \cJ_{\sigma_k\cM^{-1}T} (z^k) + \cJ_{\sigma_k\cM^{-1}T} (z^k) - z^k}_{\cM}\\
&&\le&  \gamma \norm{w^{k} - \cJ_{\sigma_k\cM^{-1}T} (z^k) }_{\cM} +\gamma\norm{ Q_{\sigma_k\cM^{-1}T} (z^k) }_{\cM}\\
&&\overset{\eqref{GPPA:stop}}{\le}& \delta \gamma \norm{w^{k} - z^k}_{\cM} + \gamma \norm{Q_{\sigma_k\cM^{-1}T} (z^k) }_{\cM}\\
&&\overset{\eqref{finite:eq3}}{\le}& \delta \norm{z^{k+1} - z^k}_{\cM} + \gamma\dist_{\cM} \bracket{z^{k},  \Omega},
\end{aligned}
\]
leading to
\[
\frac{1-\delta}{\gamma}\norm{z^{k+1} - z^k}_{\cM} \le \dist_{\cM} \bracket{z^{k},  \Omega}.
\]
Thus, the left part is proved. Now we aim to prove the right part. Since
\[\begin{aligned}
 &\norm{z^{k+1} - z^k}_{\cM}  &=& \gamma \norm{w^{k} - \cJ_{\sigma_k\cM^{-1}T} (z^k) + \cJ_{\sigma_k\cM^{-1}T} (z^k) -  z^k}_{\cM} \\
 &&\ge&\gamma \norm{Q_{\sigma_k\cM^{-1}T} (z^k) }_{\cM}  - \gamma  \norm{ w^{k} - \cJ_{\sigma_k\cM^{-1}T} (z^k) }_{\cM}  \\
 &&\overset{\eqref{GPPA:stop}}{\ge}&  \gamma\norm{Q_{\sigma_k\cM^{-1}T} (z^k)  }_{\cM}   - \delta \gamma \norm{w^{k} - z^k}_{\cM} \\
 && = & \gamma\norm{Q_{\sigma_k\cM^{-1}T} (z^k)  }_{\cM}   - \delta \norm{z^{k+1} - z^k}_{\cM}  ,
\end{aligned}\]
we have
\begin{equation}\label{lemma3:eq4} 
\norm{ Q_{\sigma_k\cM^{-1}T} (z^k) }_{\cM} \le \frac{1+\delta}{\gamma} \norm{z^{k+1} - z^k}_{\cM}.
\end{equation}
In addition, based on Lemma~\ref{lemma:key}, we have
\[
\begin{aligned}
&\dist_{\cM} \bracket{z^{k},  \Omega}&\le& \norm{\cJ_{\sigma_k\cM^{-1}T} (z^k) -  z^k  }_{\cM}+ \dist_{\cM}\bracket{ \cJ_{\sigma_k\cM^{-1}T} (z^k) , \Omega} \\
&&\overset{\eqref{eq:key}}{\le}& \norm{Q_{\sigma_k\cM^{-1}T} (z^k)  }_{\cM}+  \frac{\kappa_r }{\sqrt{\sigma_k^2 + \kappa_r^2 }} \dist_{\cM}\bracket{z^k, \Omega},
\end{aligned}
\]
which implies 
\begin{equation}\label{lemma3:eq7}
\bracket{1-  \frac{\kappa_r }{\sqrt{\sigma_k^2 + \kappa_r^2 }}}\dist_{\cM}\bracket{z^k, \Omega} \le\norm{ Q_{\sigma_k\cM^{-1}T} (z^k)  }_{\cM}.
\end{equation}
Combining~\eqref{lemma3:eq7} and~\eqref{lemma3:eq4}, we have
\[
\dist_{\cM}\bracket{z^k, \Omega} \le \frac{1 + \delta}{\gamma \bracket{1-  \frac{\kappa_r }{\sqrt{\sigma_k^2 + \kappa_r^2 }}}}\norm{z^{k+1} - z^k}_{\cM},
\]
and thus the right part is proved. 
\end{proof}

\begin{proof}[\textbf{Proof of Propostion~\ref{prop:d}}]
If~\eqref{eq:GPPAlinearrate3} holds and 
\[k\ge 
\log_{\frac{1}{\rho}}\bracket{\frac{R (\sigma_k)\zeta \dist_{\cM}\bracket{z^0,\Omega}}{\lambda_{\min}(\cM) \epsilon}},\]
 then we have 
\begin{equation}\label{eq:spleepy1}
\begin{aligned}
&\norm{z^{k+1} - z^k}_{\cM} &\le &  \frac{C\lambda_{\min}(\cM) \epsilon}{  R (\sigma_k)\zeta \dist_{\cM}\bracket{z^0,\Omega}}
 \norm{z^{1} - z^0}_{\cM}\\
 && \le& \frac{C\lambda_{\min}(\cM) \epsilon}{  R (\sigma_k)\zeta \dist_{\cM}\bracket{z^0,\Omega}} \frac{\gamma}{1-\delta}  \dist_{\cM} \bracket{z^{0},  \Omega}\\
 && \le&  \frac{  \gamma \bracket{1-  \sqrt{ \frac{\kappa_r^2 }{\sigma^2_k + \kappa_r^2 }}} \lambda_{\min}(\cM) \epsilon}{ \bracket{1 + \delta} \zeta }, \\
\end{aligned}
\end{equation}
where the second inequality derives from the first inequality in Proposition~\ref{prop:leftright} and the last inequality is from the definition of $R(\sigma)$ as in~\eqref{eq:R}. In addition,  
\begin{equation}\label{eq:spleepy2}
\cE(z^k) \overset{\eqref{a:cEl}}\le \zeta \dist(z^k, \Omega) \le \frac{\zeta}{\lambda_{\min} (\cM)} \dist_{\cM}(z^k, \Omega).
\end{equation}
Then we have 
\[\begin{aligned}
&\cE(z^k) & \overset{\eqref{eq:spleepy2}} \le&  \frac{\zeta}{\lambda_{\min} (\cM)} \dist_{\cM}(z^k, \Omega) \\
&&\le& \frac{\zeta}{\lambda_{\min} (\cM)} \frac{1 + \delta}{\gamma \left (1-  \sqrt{\frac{\kappa_r^2}{\sigma_k^2 + \kappa_r^2 }}\right)}\norm{z^{k+1} - z^k}_{\cM}\\
&&\overset{\eqref{eq:spleepy1}} \le& \epsilon,
\end{aligned}\]
where  the second inequality derives from the second inequality in Proposition~\ref{prop:leftright}.
\end{proof}

\begin{proof}[\textbf{Proof of Corollary~\ref{cor:wrd}}] 
First, we note that 
\[
\frac{1}{1- \sqrt{\frac{\kappa_r^2}{\sigma_0^2 + \kappa_r^2}}} \le 2 \bracket{1+ \frac{\kappa_r^2}{\sigma_0^2}},
\]
from which we deduce
\[
\bar R = O(\zeta r \kappa_r^2).
\]
It follows that 
\begin{equation}\label{forget:eq1}
\left\lceil \max \bracket{\log_{\frac{1}{\rho}} \bracket{\frac{\bar R}{\epsilon}}, 0} + 1  \right\rceil = O \bracket{ \ln \frac{\zeta r \kappa_r}{\epsilon}}.
\end{equation}
It is also easy to see that
\begin{equation}\label{forget:eq2}
\bar s = \left\lceil \max \bracket{ \log_{\varrho_\sigma} \bracket{\frac{\kappa_r \alpha}{\sigma_0},0}}  \right\rceil = O \bracket{\ln \kappa_r}.
\end{equation}
Hence,
\[\begin{aligned}
&- \ln \bar \eta = - \ln \eta_0 - \bar s \ln \varrho_{\eta}  + \varsigma \ln \left\lceil \max \bracket{\log_{\frac{1}{\rho}} \bracket{\frac{\bar R}{\epsilon}}, 0} + 1  \right\rceil \\
& \overset{\eqref{forget:eq1} + \eqref{forget:eq2}} =  O \bracket{\ln \kappa_r} + O \bracket{ \ln\ln \frac{\zeta r \kappa_r}{\epsilon}}
\end{aligned},
\]
and
\[
\bar N = \bar s  \left\lceil \max \bracket{\log_{\frac{1}{\rho}} \bracket{\frac{\bar R}{\epsilon}}, 0} + 1  \right\rceil \overset{\eqref{forget:eq1} + \eqref{forget:eq2}} = O \bracket{\ln \kappa_r \ln \frac{\zeta r \kappa_r}{\epsilon}}.
\]

\end{proof}

\begin{proof}[\textbf{Proof of Corollary~\ref{cor:complexity bound of AGPPA}}]
First,  ignoring problem-independent constants, from definitions in~\eqref{def:zeta12} we get
\begin{align*}
&\ln \bar \zeta_1 = O \bracket{\ln \frac{r}{p} + \ln \frac{1}{\bar \eta}+ \ln  (L_0+a)+\ln\bar \sigma},  \\
&\ln \bar \zeta_2 = O\bracket{\ln  (L_0+a)+\ln\bar \sigma + \ln \frac{1}{p}},
\end{align*}
 and therefore
\begin{align*}
&\left\lceil \ln \bar N + \max \bracket{ \max (\ln \bar \zeta_1, \ln \bar \zeta_2)   , 0}  \right\rceil 
\\&=O\left (
\ln \bar N+\ln\frac{r}{p}+\ln\frac{1}{\bar\eta}+\ln  (L_0+a)+\ln\bar \sigma
\right).
\end{align*}
Hence, the bound~\eqref{a:dfer} is of order
$$
O\left ( (\vartheta_1\bar  \sigma^{\iota} + \vartheta_2 \bar  \sigma^{\iota/2} +  \Upsilon)
 \bar N\left (
\ln \bar N+\ln\frac{r}{p}+\ln\frac{1}{\bar\eta} + \ln  (L_0+a)+\ln\bar \sigma\right)\right).
$$
By~\eqref{a:sigma} we have $$O (\bar \sigma^{\iota})=O (  (\alpha \varrho_{\sigma}\kappa_r)^\iota ),$$
with $\alpha$ and $\varrho_{\sigma}$ being user-defined parameters.   
Then the result  directly follows  Corollary~\ref{cor:wrd}. 
\end{proof}

\section{Acceleration with projected Semismooth Newton method}\label{appendix:PNCG}
The inner problem~\eqref{LP:fphi} can be reformulated as:
\[
\min_{x_{b} \ge 0} \left\{ \tilde f (x) \equiv  c^{\top}x+ \frac{1}{2\sigma} \norm{\left[\bar \lambda +  \sigma \bracket{Ax-b}  \right]^{m_I}_{+}}^2 +  \frac{1}{2\sigma } \norm{ x -\bar x }^2 \right\}.
\]
Note that $\tilde f$ is strongly convex and continuously differentiable over $\bR^n$ with
\begin{equation}\label{PSSN:gradient}
\nabla \tilde f (x) = c + A^\top\left[\bar \lambda +  \sigma \bracket{Ax-b}  \right]^{m_I}_{+}  + \frac{1}{\sigma}  (x -\bar x).
\end{equation}
Since the  projection function $[\cdot]_{+}^{m_I}$ is a Lipschitz continuous piecewise affine function, $\nabla \tilde f (x)$ is strongly  semismooth~\cite{facchinei2007finite}. 

Notice that $\nabla \tilde f$ is a semismooth function~\cite[Definition 3.5]{Li2016A}. Then the   generalized Hessian of $\tilde f$  can be expressed as
\begin{equation}\label{def:Hessian}
\nabla^2\tilde f (x)  = \sigma A^{\top} D \bracket{ \left[\bar \lambda +  \sigma \bracket{Ax-b}  \right]^{m_I}_{+} } A +\frac{1}{\sigma}
I_{n} , 
\end{equation}
where $\cI_n \in \bR^{n \times n}$ is an identity matrix, and  $D (\cdot): \bR^{m} \rightarrow \mathbb{R}^{m \times m}$ maps to a diagonal matrix with 
\[
 D_{ii} (w) =  \left\{
 \begin{aligned}
 &0,&&\mathrm{if}~i \le m_I~\mathrm{and}~w_i < 0, \\
 &1, &&\mathrm{otherwise}. 
 \end{aligned}
 \right.
 \]
Projected Semismooth Newton (PSSN) method uses the active strategy. Concretely, it only updates the coordinates of $x$ from the  active set  defined by:
\begin{equation}\label{def:activeset}
\mathcal{A} :=  [n] \backslash \left \{i \in [n_b] \middle| x_i = 0, \frac{\partial \tilde f (x)}{ \partial x_i}   > 0   \right\}.
\end{equation}
Use $[\cdot]_{\mathcal{A}}$ for vector with coordinates to be retained from $\mathcal{A}$  and  to be zero otherwise, or for matrix with the entries to be retained for both rows and columns from  $\mathcal{A}$ and  to be zero otherwise. Then, we show the process of PSSN as Algorithm~\ref{alg:PNstep}.
\renewcommand{\thealgorithm}{4}
\begin{algorithm}
	\caption{}
	\textbf{Parameters:} $\mu \in \bracket{0,1/2}, \nu \in \bracket{0,1}, \varrho \in  (0,1), \tau \in  (0,1]$
	\begin{algorithmic}
	\State 1. Compute the active set as~\eqref{def:activeset}.
	\State 2. Compute $\nabla_{\cA} \tilde f (x)$ and $\nabla^2_{\cA} \tilde f (x)$ based on~\eqref{PSSN:gradient} and~\eqref{def:Hessian} respectively.
	\State 3. 
	Solve the linear system
\begin{equation}\label{linear system}
  \nabla^2_{\mathcal{A}} \tilde f (x) y = -\nabla_{\mathcal{A}} \tilde f (x)
\end{equation}
	~~~~exactly or by the conjugate gradient  (CG) algorithm to find $y$ such that 
	\begin{equation}\label{alg1:stop}
	\norm{\nabla^2_{\mathcal{A}} \tilde f (x) y + \nabla_{\cA}\tilde f (x)}  \le \min \left\{ \nu \norm{ \nabla_{\cA}\tilde f (x)}, \norm{ \nabla_{\cA}\tilde f (x)}^{1+\tau} \right\}.
	\end{equation}
	\State 4. (Line search) Compute $\varrho^{j} $ with $j$ to be the first nonnegative integer satisfying
\[	
  \tilde f ([ x + \varrho^{j} y_{\mathcal{A}} ]_{+}^{n_b})  \le  \tilde f (x) + \mu \varrho^j \langle \nabla_{\mathcal{A}} \tilde f  (x), y_{\mathcal{A}} \rangle.
  \]
	\State 5.  Output $x^{+} = [x  + \varrho^j y_{\cA}]_{+}^{n_b}$.
	\end{algorithmic}
	\label{alg:PNstep}
\end{algorithm} 

Different from PN-CG in~\cite[Section 3.4]{Yen2015Sparse}, we use stopping criterion~\eqref{alg1:stop} for linear system~\eqref{linear system} instead of
\begin{equation}\label{stop:PNCG}
\norm{\nabla^2_{\mathcal{A}} \tilde f (x) y  +  \nabla_{\cA}\tilde f (x)}  \le  \nu \norm{ \nabla_{\cA}\tilde f (x)}.
\end{equation}
With this modification, we find that when $n_b = 0$, then Algorithm~\ref{alg:PNstep} reduces to semismooth Newton (SSN) method in~\cite{Li2019An} with slight modification, i.e., we apply stopping criterion~\eqref{alg1:stop} for linear system~\eqref{linear system} instead of  
\begin{equation}\label{stop:SunPPA}
\norm{\nabla^2_{\mathcal{A}} \tilde f (x) y + \nabla_{\cA}\tilde f (x)}  \le \min \left\{ \nu, \norm{ \nabla_{\cA}\tilde f (x)}^{1+\tau} \right\}.
\end{equation}
Note that when $n_b = 0$, same as the classical SSN method, it is trivial to check that Algorithm~\ref{alg:PNstep} still globally converges and keeps the local superlinear convergence, which shows the advantage of our  stopping criterion over~\eqref{stop:PNCG}. 
 In addition, the form~\eqref{alg1:stop}  is  more appropriate for the choice of parameter $\nu$ than~\eqref{stop:SunPPA}. Thus, we have an adequate motivation to study the convergence of  Algorithm~\ref{alg:PNstep} for the case where $n_b > 0$  and we will do further research in the subsequent paper. Again, we emphasize that there is no specific requirement on the theoretical convergence of Algorithm~\ref{alg:PNstep} when it is used in Algorithm~\ref{cQ_cAcG}.

\end{appendices}


%
%

\bibliographystyle{spmpsci}      
\bibliography{template}   

\end{document}